\documentclass[11pt]{article}
\usepackage[top=2.5cm,bottom=2.5cm,left=1.5cm,right=1.5cm]{geometry}
\usepackage{authblk}
\usepackage{amsmath}
\usepackage{amssymb}
\usepackage{amsthm}
\usepackage{color}
\usepackage{bm}
\usepackage{cases}
\usepackage[colorlinks,linkcolor=blue,anchorcolor=green,citecolor=red]{hyperref}
\usepackage{appendix}
\usepackage{graphicx}
\usepackage{subfigure}
\usepackage{booktabs}
\usepackage{lscape}
\usepackage[figuresright]{rotating}
\usepackage{cite}
\usepackage{algorithm}
\usepackage{algpseudocode}  
\allowdisplaybreaks[4]
\numberwithin{equation}{section}
\newtheorem{definition}{Definition}
\newtheorem{lemma}{Lemma}[section]
\newtheorem{theorem}{Theorem}[section] 
\newtheorem{proposition}{Proposition}[section]
\newtheorem{remark}{Remark}[section] 

\def\u{\mbox{\boldmath $u$}}   

\title{A second-order accurate, positivity-preserving 
numerical scheme for the Poisson-Nernst-Planck-Navier-Stokes system}
\author[a]{Yuzhe Qin}
\author[b,*]{Cheng Wang}
\affil[a]{Key Laboratory of Complex Systems and Data Science of Ministry of Education $\&$ School of Mathematical Sciences, 
Shanxi University, Taiyuan 030006, China; Email: yzqin@sxu.edu.cn}
\affil[b]{Mathematics Department, University of Massachusetts, North Dartmouth, MA 02747, USA} 
\affil[*]{Corresponding author, cwang1@umassd.edu} 
\date{}

\begin{document}
\bibliographystyle{plain}
\maketitle
\begin{abstract} 
In this paper, we propose and analyze a second order accurate (in both time and space) numerical scheme for the Poisson-Nernst-Planck-Navier-Stokes system, which describes the ion electro-diffusion in fluids. In particular, the Poisson-Nernst-Planck equation is reformulated as a non-constant mobility gradient flow in the Energetic Variational Approach. The marker and cell finite difference method is chosen as the spatial discretization, which facilitates the analysis for the fluid part. 
In the temporal discretization, the mobility function is computed by a second order extrapolation formula for the sake of unique solvability analysis, while a modified Crank-Nicolson approximation is applied to the singular logarithmic nonlinear term. 
Nonlinear artificial regularization terms are added in the chemical potential part, so that the positivity-preserving property could be theoretically proved. Meanwhile, a second order accurate, semi-implicit approximation is applied to the convective term in the PNP evolutionary equation, and the fluid momentum equation is similarly computed. In addition, an optimal rate convergence analysis is provided, based on the higher order asymptotic expansion for the numerical solution, the rough and refined error estimate techniques. The following combined theoretical properties have been established for the second order accurate numerical method: (i) second order accuracy, (ii) unique solvability and positivity, (iii) total energy stability, and (iv) optimal rate convergence. A few numerical results are displayed to validate the theoretical analysis.

	\bigskip



\end{abstract}

\section{Introduction} 
The coupled Poisson-Nernst-Planck-Navier-Stokes (PNPNS) system is an important model to describe the diffusion process of charged particles, originated from bio-electronic application. This well-known electro-fluid model has been used to study the dynamics of electrically charged fluids, the motion of ions or molecules and their interactions under the influence of electric fields and the surrounding fluid. In electro hydrodynamics, the ionic motion with different valences suspended in a solution is driven by the fluid flow and an electric potential, which results from both an applied potential on the boundary and the distribution of charges carried by the ions. In addition, ionic diffusion is driven by the concentration gradients of the ions themselves. Conversely, fluid flow is forced by the electrical field created by the ions, which arise frequently in a large number of physical, biophysical, and industrial processes. For more details of the physical background issues of this system, we refer the readers to \cite{Martin2004, Ben2002Nonlinear, Rubinstein1990Electro, Igor2007The, Qin2023POF} and the references therein. 

Several papers have analyzed the mathematical property of PNPNS system. Based on semigroup ideas, the existence of a unique smooth local solution for smooth initial data, with non-negativity preserved for the ion concentrations, was obtained in \cite{Joseph2002Analytical}. In \cite{Markus2009Analysis}, Schmuck proved the global existence and uniqueness of weak solutions in a two or three dimensional bounded domain, with blocking boundary condition for the ions and homogeneous Neumann boundary condition for the electric potential. Besides, Bothe et al. \cite{Dieter2014Global} investigated the Robin boundary condition for the electric potential, and established the global existence and stability in two-dimensional domain. Wang et al. \cite{Wang2016SIAM} derived a hydrodynamic model of the compressible conductive fluid, so-called generalized PNPNS system, and developed a general method to prove that the system is globally asymptotically stable under small perturbations around a constant equilibrium state. They also obtained an optimal decay rate of the solution and its derivatives of any order under certain conditions. Constantin published a series of papers, such as \cite{Peter2019On,Peter2022Existence}, to analyze the global existence of smooth solutions with different boundary conditions. 

Since the ion concentration must be non-negative, it would be very important to develop a numerical scheme preserving positivity for the ion concentrations. For certain gradient flow models with a singular energy potential, such as the Flory-Huggins-Cahn-Hilliard equation, some existing works have been reported to establish the  positivity-preserving property of the associated numerical schemes~\cite{Qin2021BFS, WangchengCHNS2022, ChenCHNS2024, Chen2019Positivity, Liu2022JCP}. Meanwhile, for the PNP system, the corresponding analysis becomes more challenging, due to the lack of the standard diffusion energy in the variational energetic structure. Some efforts have been made to deal with this issue. For example, Shen and Xu developed a set of numerical schemes for the PNP equations in \cite{Shen2021Unconditionally}, and the numerical schemes are proved to be mass conservative, uniquely solvable and positivity-preserving. He et al. proposed a positivity-preserving and free energy dissipative numerical scheme for the PNP system in \cite{He2019PNP}, which could be linearly solved. Subsequently, the theoretical analysis for this linear scheme was provided in \cite{Qin2023PNP}. Moreover, Liu et al. considered the PNP system in the energetic variational formulation and proposed both the first and second order numerical schemes~\cite{Liu2021PNP, Liu2023PNP}, which preserve three theoretical properties: unique solvability/positivity-preserving, unconditional energy stability, and optimal rate convergence analysis.  

The numerical effort for the PNPNS system turns out to be even more challenging, due to the highly coupled nature between the PNP evolution and fluid motion. Tsai et al. \cite{Tsai2005Numerical} employed an artificial compressibility approach in capillary electrophoresis microchips, and tested some injection systems with different configurations. Prohl and Schmuck \cite{Andreas2010Convergent} used the implicit temporal discretization, combined with the finite element spatial approximation, to preserve the non-negativity of the ions. A projection method without non-negativity preserving was also considered in the work. He and Sun considered a few finite element schemes for the PNPNS system \cite{He2018Mixed}, which preserves the positivity and/or some form of energy dissipation under certain conditions and specific spatial discretization. Liu and Xu \cite{Liu2017Efficient} proposed a few numerical methods, with different accuracy orders, by combining various time-stepping stencils and the spectral spatial discretization. The proposed numerical schemes result in several elliptic equations, with time-dependent coefficients, to be solved at each time step. Among these proposed algorithms, only the first order one has been theoretically proved to be positivity-preserving. Meanwhile, based on the popular scalar auxiliary variable (SAV) approach, Zhou and Xu \cite{Zhou2023Efficient} proposed a few first/second-order accurate numerical schemes for the evolutional PNPNS system, in which the ion positivity is preserved.  

Of course, based on the energetic variational approach, any numerical analysis for the PNPNS system has to face three serious theoretical issues: ion concentration positivity and unique solvability, total energy stability, and optimal rate convergence estimate. In fact, most existing numerical works for the PNPNS system have addressed one or two theoretical issues, while a numerical design that combines all three theoretical properties turns out to be even more challenging than that of the PNP system. In addition, the construction of a second order scheme to preserve these three theoretical properties would be more difficult. In this article, we propose and analyze a second order numerical scheme for the PNPNS system, in which all three properties will be theoretically justified. To facilitate the numerical design, the PNP part is reformulated as a non-constant mobility $H^{-1}$ gradient flow in the energetic variational approach. The highly nonlinear and singular nature of the logarithmic energy potential has always been the essential difficulty to design a second order accurate scheme in time, while preserving the variational energetic structures. In the temporal discretization for the PNP part, a second order accurate extrapolation is taken to the mobility function, for the sake of unique solvability. In the chemical potential expansion, a modified Crank-Nicolson approximation is applied to the singular logarithmic nonlinear term, and such a treatment leads to a stability estimate in terms of the Flory-Huggins energy. Furthermore, nonlinear artificial regularization terms are added in the chemical potential expansion, which could facilitate the positivity-preserving analysis for the ion concentration variables. In the fluid convection and the convection terms for the ion concentration variables, a second order accurate, semi-implicit method is used. The coupled source terms in the fluid momentum equation is similarly computed. Meanwhile, the marker and cell (MAC) finite difference spatial discretization is used, which in turn makes the computed velocity vector divergence-free at a discrete level, so that it is orthogonal to the pressure gradient in the discrete $\ell^2$ space. This fact will also play a key role in the numerical analysis. The singular nature of the logarithmic terms, combined with the monotonicity of the numerical system, enable us to theoretically justify its unique solvability/positivity-preserving property. With such an established property, an unconditional total energy stability becomes an outcome of a careful energy estimate of the numerical system. In addition, an optimal rate convergence analysis will also be derived for the proposed scheme, which is accomplished by the higher order asymptotic expansion for the numerical solution, combined with the rough and refined error estimate techniques. In the authors' knowledge, there is first work of second order accurate numerical scheme for the PNPNS system that preserves all three theoretical properties. 

The remainder of this paper is organized as follows. In Section~\ref{sec: PNP}, we first describe the PNPNS system, and its reformulation based on EnVarA method. In Section~\ref{sec: scheme}, the second order accurate numerical scheme is constructed, based on the reformulated PNPNS system. The unique solvability/positivity-preserving property is proved in Section~\ref{sec: positivity}, and an unconditional total energy stability is established in Section~\ref{sec: stability}. Moreover, the optimal rate convergence analysis is provided in Section~\ref{sec: convergence}. A few numerical examples are presented in Section \ref{sec: validation} as well, which validates the robustness of thee proposed scheme. Finally,  some concluding remarks are made in Section \ref{sec: conclusion}.

\section{The governing equation and its energy law} \label{sec: PNP}

\subsection{Dimensional system and the energy law}

We consider the dimensionless system and omit the dimensionalization process. The general PNPNS system with $K$ species of ions in the electrolyte solution is given as follows:  
\begin{subequations}\label{eqn: dimensionless main}
	\begin{align}
		& \frac{\partial c_{q}}{\partial t} 
		+ \nabla \cdot \left(\bm{u} c_{q}\right) 
		= \frac{D_{q}}{Pe} \nabla \cdot \left( c_{q} \nabla \mu_{q} \right),  
		\label{main eqn: cq}\\ 
		& \mu_{q} = z_{q} \phi + \ln c_{q},  
		\label{main eqn: mui}\\
		& - \epsilon^{2} \Delta \phi 
		= \sum_{q=1}^{K} z_{q} c_{q},  
		\label{main eqn: phi}\\
        & Re \Big( \frac{\partial \bm{u}}{\partial t} 
        + ( \bm{u} \cdot \nabla ) \bm{u} \Big) 
        + \nabla \psi 
        = \Delta \bm{u} 
        - \sum_{q=1}^{K} z_{q} c_{q} \nabla \phi, 
		\label{main eqn: u1}\\ 
		& \nabla \cdot \bm{u} = 0, 
		\label{main eqn: nablau} 
	\end{align}
\end{subequations} 
where $c_{q}$ is the concentration of $q$-th ion, $\bm{u}$ is the velocity, $\phi$ is the electric potential, $\psi$ is the pressure, $Pe$ is the Peclet number, $D_{q}$ is the diffusion coefficient of $q$-th ion, $\mu_{q}$ is the chemical potential of $q$-th ion, $z_{q}$ is the valence of $q$-th ion, $\epsilon$ is the dielectric coefficient, and $Re$ is the Reynolds number. Periodic boundary condition could be taken, for simplicity of analysis. As an alternate choice, the following physically relevant boundary condition could also be considered,  so that the above system \eqref{eqn: dimensionless main} becomes self-contained:  
\begin{equation}\label{eqn: bdc1}
    \partial_{\bm{n}} c_{q}|_{\Gamma} = 0, \quad 
    \partial_{\bm{n}} \phi |_{\Gamma} = 0, \quad 
   ( \bm{u} \cdot \bm{n} ) |_{\Gamma} = 0, \quad 
   \partial_{\bm{n}}  ( \bm{u} \cdot \bm{\tau} ) |_{\Gamma} = 0 , \quad 
    \partial_{\bm{n}} \psi |_{\Gamma} = 0. 
\end{equation}

The total energy dissipation property has been derived in an existing work~\cite{Markus2009Analysis}. 

\begin{theorem}  \cite{Markus2009Analysis} [Total energy law] \label{thm: energy law}
The following energy dissipation law is satisfied for system \eqref{eqn: dimensionless main}: 
\begin{equation} 
    \frac{\mathrm{d}E_{total}}{\mathrm{d}t}  \le 0 , \quad 
    E_{total} :=  \int_{\Omega} 
    \Big( \Big(\sum_{q=1}^{K} c_{q} (\ln c_{q} - 1)
    + \frac{\epsilon^{2}}{2} |\nabla \phi|^{2} \Big)
    + \frac{Re}{2} |\bm{u} |^{2} \Big) \mathrm{d}\bm{x} . 
\end{equation}
\end{theorem}

\begin{remark}
For the PNPNS system with a periodic boundary condition, the energy dissipation law is also valid; the technical details are left to interested readers. 
\end{remark}

\subsection{Reformulated system}
For the sake of numerical convenience, we consider the two-particle PNPNS system, namely, $p$ for positive ion and $n$ for negative ion. Meanwhile, the dimensionless constants will not cause any essential difficulty in the numerical analysis, so that a uniform value is set for all these constants: $D_{q} = 1$, $Pe = 1$, $\epsilon = 1$, $Re = 1$ and $z_{p} = 1$, $z_{n} = -1$. In turn, system \eqref{eqn: dimensionless main} could be equivalently rewritten as the following simplified form: 
\begin{subequations}\label{eqn: main system}
	\begin{align}
		& \partial_{t} p 
		+ \nabla \cdot \left( \bm{u} p \right) 
		= \nabla \cdot \left( p \nabla \mu_{p} \right), 
		\label{main: p}\\
  	& \partial_{t} n 
		+ \nabla \cdot \left( \bm{u} n \right) 
		= \nabla \cdot \left( n \nabla \mu_{n} \right), 
		\label{main: n}\\
		&\mu_{p} = \ln p + \left( - \Delta \right)^{-1} \left( p - n \right), 
		\label{main: mu_p}\\
        &\mu_{n} = \ln n + \left( - \Delta \right)^{-1} \left( n - p \right), 
		\label{main: mu_n}\\
        & \partial_{t} \bm{u} 
        + ( \bm{u}\cdot \nabla ) \bm{u} 
        + \nabla \psi  
        = \Delta \bm{u} 
        - p \nabla \mu_{p}
        - n \nabla \mu_{n}, 
		\label{main: u}\\ 
		& \nabla \cdot \bm{u} = 0 .  
		\label{main: nabla u} 
	\end{align}
\end{subequations} 
Meanwhile, either the periodic boundary condition, or the homogeneous physical boundary condition, could be imposed: 
\begin{equation}\label{eqn: bdc}
    \partial_{\bm{n}} p |_{\Gamma} = 0, \quad 
    \partial_{\bm{n}} n |_{\Gamma} = 0, \quad 
    \partial_{\bm{n}} \phi |_{\Gamma} = 0, \quad 
    ( \bm{u} \cdot \bm{n} ) |_{\Gamma} = 0, \quad 
   \partial_{\bm{n}}  ( \bm{u} \cdot \bm{\tau} ) |_{\Gamma} = 0 , \quad 
    \partial_{\bm{n}} \psi |_{\Gamma} = 0. 
\end{equation}
\begin{remark}
    System \eqref{eqn: main system} is equivalent to the original system \eqref{eqn: dimensionless main}, with an introduction of a new pressure function $\tilde{\psi} = \psi - \sum\limits_{q=1}^{K}c_{q}$. Such a new variable is physically relevant, since $\sum\limits_{q=1}^{K}c_{q}$ could be regarded as the osmotic pressure. For the sake of convenience, we omit the $\tilde{\cdot}$ symbol and still use $\psi$ to represent the pressure function. 
\end{remark}

\begin{theorem}\label{thm: energy dissipation}
    The following energy dissipation law is valid for system \eqref{eqn: main system}: 
    \begin{equation} 
    \frac{\mathrm{d}E_{total}}{\mathrm{d}t} \le 0, \quad 
    E_{total} :=  \int_{\Omega} 
    \Big( p (\ln p - 1 ) + n ( \ln n - 1 )
    + \frac{1}{2} |\nabla \phi|^{2} 
    + \frac{1}{2} |\bm{u} |^{2} \Big) \mathrm{d}\bm{x} .  
    \label{energy dissipation-1} 
\end{equation}
\end{theorem}
The proof of the above inequality is essentially the same as the one in Theorem \ref{thm: energy law}. The technical details are left to the interested readers.

\section{The second order accurate numerical scheme} \label{sec: scheme}

\subsection{The finite difference spatial discretization}
The numerical scheme is based on the equivalent reformulation system \eqref{eqn: main system}. 
For simplicity of presentation, 
we consider a two-dimensional domain $\Omega=(0,L_{x})\times(0,L_{y})$, 
with $L_{x}=L_{y}=L>0$. Let $N$ be a positive integer such that $h = L/N$, 
which stands for the spatial mesh size. 
All the scalar variables, such as ion concentration $c_{q}$, 
electric potential $\phi$, chemical potential $\mu_{q}$ and pressure $\psi$, 
are evaluated at the cell-centered mesh points: $((i + 1/2)h, ( j + 1/2)h)$, 
at the component-wise level. 
In this section, we use $f$ to represent the scalar variable, 
and $\bm{v}$ as the vector variable. 
In turn, the discrete gradient of $f$ is evaluated at the mesh points 
$(ih, ( j + 1/2)h),((i + 1/2)h, jh)$, respectively: 
\begin{equation}
    (D_{x}f)_{i,j+\frac{1}{2}} 
    = \frac{f_{i+\frac{1}{2},j+\frac{1}{2}}-f_{i-\frac{1}{2},j+\frac{1}{2}}}{h}, \quad 
    (D_{y}f)_{i+\frac{1}{2},j} 
    = \frac{f_{i+\frac{1}{2},j+\frac{1}{2}}-f_{i+\frac{1}{2},j-\frac{1}{2}}}{h} .  
\end{equation}
Similarly, the wide-stencil differences for cell centered functions could be introduced as
\begin{equation}
    (\tilde{D}_{x}f)_{i+\frac{1}{2},j+\frac{1}{2}} 
    = \frac{f_{i+\frac{3}{2},j+\frac{1}{2}}-f_{i-\frac{1}{2},j+\frac{1}{2}}}{2h}, \quad 
    (\tilde{D}_{y}f)_{i+\frac{1}{2},j+\frac{1}{2}} 
    = \frac{f_{i+\frac{1}{2},j+\frac{3}{2}}-f_{i+\frac{1}{2},j-\frac{1}{2}}}{2h}. 
\end{equation}
The five-point Laplacian takes a standard form. 
Meanwhile, a staggered grid is used for the velocity field, 
in which the individual components of a given velocity, 
say, $\bm{v} = (v^{x},v^{y})$, are defined at the east-west cell edge points $(ih, ( j + 1/2)h)$, 
and the north-south cell edge points $((i + 1/2)h, jh)$, respectively. 
This staggered grid is also known as the marker and cell (MAC) grid; 
it was first proposed in \cite{Harlow1965MAC} to deal with the incompressible Navier-Stokes equations, 
and the detailed analyses have been provided in \cite{E2002MAC,Wang2000MAC}, etc. 

The discrete divergence of $\bm{v} = (v^x, y^y )^T$ is defined at the cell center points $((i+1/2) h,(j+1/2) h)$ as follows:
\begin{equation}
    \left(\nabla_{h} \cdot \bm{v}\right)_{i+1/2, j+1/2}
    :=\left(D_{x} v^{x}\right)_{i+1/2, j+1/2}+\left(D_{y} v^{y}\right)_{i+1/2, j+1/2}.
\end{equation}
One key advantage of the MAC grid approach is that the discrete divergence of the velocity vector will always be identically zero at every cell center point. Such a divergence-free property comes from the special structure of the MAC grid and assures that the velocity field is orthogonal to the corresponding pressure gradient at the discrete level; also see reference~\cite{E2002MAC}.

For $\bm{u}=\left(u^{x}, u^{y}\right)^{T}, \bm{v}=\left(v^{x}, v^{y}\right)^{T}$, evaluated at the staggered mesh points $\left(x_{i}, y_{j+1/2}\right)$, $\left(x_{i+1/2}, y_{j}\right)$, respectively, and the cell centered variable $f$, the following terms are computed as
\begin{subequations}
    \begin{align}
        \bm{u} \cdot \nabla_{h} \bm{v} 
        & =\left(\begin{array}{l}
        u_{i, j+1/2}^{x} \tilde{D}_{x} v_{i, j+1/2}^{x}
        +\mathcal{A}_{x y} u_{i,j+1/2}^{y} \tilde{D}_{y} v_{i, j+1/2}^{x} \\
        \mathcal{A}_{x y} u_{i+1/2, j}^{x} \tilde{D}_{x} v_{i+1/2, j}^{y}
        +u_{i, j+1/2}^{y} \tilde{D}_{y} v_{i+1/2, j}^{y}
        \end{array}\right), \\
        \nabla_{h} \cdot\left(\bm{v} \bm{u}^{T}\right) 
        & =\left(\begin{array}{l}
        \tilde{D}_{x}\left(u^{x} v^{x}\right)_{i, j+1/2}
        +\tilde{D}_{y}\left(\mathcal{A}_{x y} u^{y} v^{x}\right)_{i, j+1/2} \\
        \tilde{D}_{x}\left(\mathcal{A}_{x y} u^{x} v^{y}\right)_{i+1/2, j}
        +\tilde{D}_{y}\left(u^{y} v^{y}\right)_{i+1/2, j}
        \end{array}\right), \\
        \mathcal{A}_{h} f \nabla_{h} \mu & =\left(\begin{array}{l}
        \left.\left(D_{x} \mu \cdot \mathcal{A}_{x} f\right)_{i, j+1/2}\right)_{i, j+1/2} \\
        \left.\left(D_{y} \mu \cdot \mathcal{A}_{y} f\right)_{i+1/2, j}\right)_{i+1/2, j}
        \end{array}\right), \\
        \nabla_{h} \cdot\left(\mathcal{A}_{h} f \bm{u}\right) 
        & =D_{x}\left(u^{x} \mathcal{A}_{x} f\right)_{i+1/2,j+1/2}
        +D_{y}\left(u^{y} \mathcal{A}_{y} f\right)_{i+1/2,j+1/2},
    \end{align}
\end{subequations}
in which the following averaging operators have been employed:
\begin{subequations}
    \begin{align}
        \mathcal{A}_{x y} u_{i+1/2, j}^{x} 
        & =\frac{1}{4}\left(u_{i, j-1/2}^{x}+u_{i, j+1/2}^{x}
        +u_{i+1, j-1/2}^{x}+u_{i+1, j+1/2}^{x}\right), \\
        \mathcal{A}_{x} f_{i, j+1/2} 
        & =\frac{1}{2}\left(f_{i-1/2, j+1/2}+f_{i+1/2, j+1/2}\right) .
    \end{align}
\end{subequations}
A few other average terms, such as $\mathcal{A}_{x y}u_{i,j+1/2}^{y},\mathcal{A}_{y} f_{i+1/2,j}$, could be defined in the same manner.

\begin{definition}
    For any pair of variables $u^{a}, u^{b}$ which are evaluated at the mesh points $(i, j+1/2)$, (such as $u, D_{x} f, D_{x} \mu, D_{x} p$, et cetera.), the discrete $\ell^{2}$-inner product is defined by
    \begin{equation}
        \langle u^{a}, u^{b} \rangle_{A}
        =h^{2} \sum_{j=1}^{N} \sum_{i=1}^{N} u_{i, j+1/2}^{a} u_{i, j+1/2}^{b} ; 
    \end{equation}
for any pair of variables $v^{a}$, $v^{b}$ which are evaluated at the mesh points $(i+1/2,j)$ (such as $v, D_{y} f, D_{y} \mu, D_{y} p$, et cetera.), the discrete $\ell^{2}$-inner product is defined by
    \begin{equation}
        \langle v^{a}, v^{b} \rangle_{B}
        =h^{2} \sum_{j=1}^{N} \sum_{i=1}^{N} v_{i+1/2, j}^{a} v_{i+1/2,j}^{b} ; 
    \end{equation}
for any pair of variables $f^{a}$, $f^{b}$ which are evaluated at the mesh points $(i+1/2,j+1/2)$, the discrete $\ell^{2}$-inner product is defined by 
    \begin{equation}
        \langle f^{a}, f^{b} \rangle_{C}
        =h^{2} \sum_{j=1}^{N} \sum_{i=1}^{N} f_{i+1/2, j+1/2}^{a} f_{i+1/2,j+1/2}^{b} . 
    \end{equation}
In addition, for two velocity vector $\bm{u}=\left(u^{x}, u^{y}\right)^{T}$ and $\bm{v}=\left(v^{x}, v^{y}\right)^{T}$, we denote their vector inner product as
    \begin{equation}
        \langle\bm{u}, \bm{v}\rangle_{1}
        =\left\langle u^{x}, v^{x}\right\rangle_{A}+\left\langle u^{y}, v^{y}\right\rangle_{B}.
    \end{equation} 
    The associated $\ell^{2}$ norms, namely, $\|\cdot\|_{2}$ norm, 
    can be defined accordingly. 
    It is clear that all the discrete $\ell^{2}$ inner products defined above are second order accurate. 
    In addition to the standard $\ell^{2}$ norm, 
    we also introduce the $\ell^{p}, 1 \leq p<\infty$, and $\ell^{\infty}$ norms for a grid function $f$ evaluated at mesh points $(i+1/2, j+1/2)$:
    \begin{equation}
        \|f\|_{\infty}:=\max _{i,j}\left|f_{i+1/2,j+1/2}\right|,~ 
        \|f\|_{p}:=\Big( h^{2}\sum_{i,j=1}^{N}\left|f_{i+1/2,j+1/2}\right|^{p} \Big)^{\frac{1}{p}}, ~1 \leq p<\infty .
    \end{equation}
\end{definition}

Meanwhile, the discrete average is denoted as $\overline{f} := \frac{1}{| \Omega |} \langle f , 1 \rangle_C$, for any cell centered function $f$. For the convenience of the later analysis, an $\langle \cdot , \cdot \rangle_{-1,h}$ inner product and $\| \cdot \|_{-1, h}$ norm need to be introduced, for any $\varphi  \in \mathring{\mathcal C}_{\Omega} := \left\{ f \middle| \ \langle f , 1 \rangle_C = 0 \right\}$:  
	\begin{equation} 
\langle \varphi_1, \varphi_2  \rangle_{-1,h} = \langle \varphi_1 ,  (-\Delta_h )^{-1} \varphi_2 \rangle_C , \quad \| \varphi  \|_{-1, h } = ( \langle \varphi ,  ( - \Delta_h )^{-1} (\varphi) \rangle_C )^\frac12 ,
	\end{equation} 
where the operator $\Delta_h$ is equipped with either periodic or discrete homogeneous Neumann boundary condition.

\begin{lemma} \cite{WangchengCHNS2022, WangchengCHNS2024} 
    For two discrete grid vector functions $\bm{u}=\left(u^{x}, u^{y}\right), \bm{v}=\left(v^{x}, v^{y}\right)$, where $u^{x}, u^{y}$ and $v^{x}, v^{y}$ are defined on east-west and north-south respectively, and two cell centered functions $f, g$, the following identities are valid, if $\bm{u}, \bm{v}$ are implemented with homogeneous Dirichlet boundary condition and homogeneous Neumann boundary condition is imposed for $f$ and $g$:
\begin{subequations}
    \begin{align}
        & \left\langle\bm{v}, \bm{u} \cdot \nabla_{h} \bm{v}\right\rangle_{1}
        +\left\langle\bm{v}, \nabla_{h} \cdot\left(\bm{v} \bm{u}^{T}\right)\right\rangle_{1}=0, 
        \label{lem: discrete inner} \\
        & \left\langle\bm{u}, \nabla_{h} f\right\rangle_{1}=0, \quad \text {if} \quad \nabla_{h} \cdot \bm{u}=0, 
        \label{lem: discrete inner of u and f}\\
        - & \left\langle\bm{v}, \Delta_{h} \bm{v}\right\rangle_{1}=\left\|\nabla_{h} \bm{v}\right\|_{2}^{2}, \\
        - & \left\langle f, \Delta_{h} f\right\rangle_{C}=\left\|\nabla_{h} f\right\|_{2}^{2} ,\\
        - & \left\langle g, \nabla_{h} \cdot\left(\mathcal{A}_{h} f \bm{u}\right)\right\rangle_{C} 
        = \left\langle\bm{u}, \mathcal{A}_{h} f \nabla_{h} g\right\rangle_{1}. 
        \label{lem: discrete inner of g and afu}
    \end{align}
\end{subequations}
The same conclusion is true if all the variables are equipped with periodic boundary condition. 
\end{lemma}

The following Poincar\'e-type inequality will be useful in the later analysis.

\begin{proposition} \label{prop:1} 
\begin{enumerate}
    \item There are constants $C_{0} , \, \breve{C}_0 > 0$, independent of $h > 0$, such that 
    $\|f \|_{2} \leq C_{0} \|\nabla_{h}f \|_{2}$, 
    $\|f \|_{-1, h} \leq C_0 \| f \|_{2}$, $\| f \|_2 \le \breve{C}_0 h^{-1} \| f \|_{-1, h}$, 
    for all $f \in \mathring{\mathcal{C}}_{\Omega}:= \{f| \left\langle f, 1\right\rangle_{C} = 0\}$.  
    \item For a velocity vector $\bm{v}$, with a discrete no-penetration boundary condition $\bm{v}\cdot\bm{n} = 0$ on $\partial \Omega$, 
    a similar Poincar\'e inequality is also valid: $\left\|\bm{v}\right\|_{2} \leq C_{0}\left\|\nabla_{h} \bm{v}\right\|_{2}$, 
    with $C_{0}$ only dependent on $\Omega$.
    \end{enumerate} 
\end{proposition}

\subsection{The second order accurate numerical scheme}
\begin{subequations}\label{scheme: main}
    \begin{align}
        & \frac{\hat{\bm{u}}^{m+1}-\bm{u}^{m}}{\tau}
        + \frac{1}{2} \Big(\tilde{\bm{u}}^{m+1/2} \cdot \nabla_{h} \hat{\bm{u}}^{m+1/2}
        + \nabla_{h} \cdot \Big( \hat{\bm{u}}^{m+1/2} (\tilde{\bm{u}}^{m+1/2} )^{T} \Big) \Big)
        + \nabla_{h} \psi^{m}
        - \Delta_{h} \hat{\bm{u}}^{m+1/2} \nonumber 
        \\
        & = - \mathcal{A}_{h} \tilde{p}^{m+1/2} \nabla_{h} \mu_{p}^{m+1/2} 
        - \mathcal{A}_{h} \tilde{n}^{m+1/2} \nabla_{h} \mu_{n}^{m+1/2}, 
        \label{scheme: u_hat}\\
        & \frac{n^{m+1}-n^{m}}{\tau} 
        + \nabla_{h} \cdot \Big (\mathcal{A}_{h} \tilde{n}^{m+1/2} \hat{\bm{u}}^{m+1/2} \Big) 
        = \nabla_{h} \cdot \Big( \breve{n}^{m+1/2} \nabla_{h} \mu_{n}^{m+1/2} \Big), 
        \label{scheme: n}  \\
        & \frac{p^{m+1}-p^{m}}{\tau} 
        + \nabla_{h} \cdot \Big( \mathcal{A}_{h} \tilde{p}^{m+1/2} \hat{\bm{u}}^{m+1/2} \Big)
        = \nabla_{h} \cdot \Big( \breve{p}^{m+1/2} \nabla_{h} \mu_{p}^{m+1/2} \Big), 
        \label{scheme: p} \\
        & \mu_{n}^{m+1/2} 
        = \frac{n^{m+1} \ln n^{m+1}-n^{m} \ln n^{m}}{n^{m+1}-n^{m}}-1
        + \tau \ln \frac{n^{m+1}}{n^{m}} 
        + (-\Delta_{h})^{-1} (n^{m+1/2}-p^{m+1/2} ), 
        \label{scheme: mu_n}\\
        & \mu_{p}^{m+1/2} 
        = \frac{p^{m+1} \ln p^{m+1}-p^{m} \ln p^{m}}{p^{m+1}-p^{m}}-1
        + \tau \ln \frac{p^{m+1}}{p^{m}}
        + (-\Delta_{h})^{-1} (p^{m+1/2}-n^{m+1/2} ) , 
        \label{scheme: mu_p} \\
        & \frac{\bm{u}^{m+1}-\hat{\bm{u}}^{m+1}}{\tau} 
        + \frac{1}{2} \nabla_{h} (\psi^{m+1}-\psi^{m} )=0, 
        \label{scheme: u}\\
        & \nabla_{h} \cdot \bm{u}^{m+1}=0, 
        \label{scheme: nabla_u}
    \end{align}
\end{subequations} 
where
\begin{equation}\label{eqn:interpolation}
    \begin{aligned}
    & \tilde{\bm{u}}^{m+1/2}:=\frac{3}{2} \bm{u}^{m}- \frac{1}{2} \bm{u}^{m-1}, \quad 
    \hat{\bm{u}}^{m+1/2}:= \frac{1}{2} \hat{\bm{u}}^{m+1}+ \frac{1}{2} \bm{u}^{m},
    \\
    &\tilde{p}^{m+1/2}:=\frac{3}{2} p^{m}- \frac{1}{2} p^{m-1}, \quad 
    \tilde{n}^{m+1/2}:=\frac{3}{2} n^{m}- \frac{1}{2} n^{m-1},
    \\
    & n^{m+1/2} := \frac{1}{2}\left(n^{m+1}+n^{m}\right), \quad 
    p^{m+1/2} := \frac{1}{2}\left(p^{m+1}+p^{m}\right), \quad  
\\ 
    &\breve{\varphi}^{m+1/2}_{i+1/2,j} := \left\{ \begin{array}{l} 
    A_{x}\tilde{\varphi}^{m+1/2}_{i+1/2,j} ,  \quad \mbox{if $A_{x}\tilde{\varphi}^{m+1/2}_{i+1/2,j} >0$} , \\
    ( (A_{x}\tilde{\varphi}^{m+1/2}_{i+1/2,j} )^{2}+\tau^8 )^{1/2},  \quad  
    \mbox{if $A_{x}\tilde{\varphi}^{m+1/2}_{i+1/2,j} \le 0$} , 
    \end{array} \right. \quad  (\varphi = n, p) ,   
     \\
    &\breve{\varphi}^{m+1/2}_{i,j+1/2} := \left\{ \begin{array}{l} 
    A_{y}\tilde{\varphi}^{m+1/2}_{i,j+1/2} ,  \quad \mbox{if $A_{y} \tilde{\varphi}^{m+1/2}_{i,j+1/2} >0$} , \\ 
    ( (A_{y}\tilde{\varphi}^{m+1/2}_{i,j+1/2} )^{2}+\tau^8 )^{1/2},  \quad 
     \mbox{if $A_{y}\tilde{\varphi}^{m+1/2}_{i,j+1/2} \le 0$} , 
    \end{array} \right. \quad  (\varphi = n, p) ,  
\end{aligned}
\end{equation}
with either periodic boundary condition, or the discrete physical boundary condition:
\begin{equation}
\begin{aligned}
    & \left. ( \hat{\bm{u}}^{m+1} \cdot \bm{n} ) \right|_{\Gamma} = 0, \quad 
   \left.  ( \nabla_h ( \hat{\bm{u}}^{m+1} \cdot \bm{\tau} )  \cdot \bm{n} ) \right|_{\Gamma} = 0 , \quad 
    \left. \bm{u}^{m+1} \cdot \bm{n} \right|_{\Gamma} = 0, \quad 
 \left. ( \nabla\psi^{m+1} \cdot \bm{n}  ) \right|_{\Gamma} = 0, 
\\
  & 
    \left. \partial_{\bm{n}} p^{m+1}\right|_{\Gamma}=\left.\partial_{\bm{n}} n^{m+1}\right|_{\Gamma}=0, 
    \quad 
    \left.\partial_{\bm{n}} \mu_{p}^{m+1/2}\right|_{\Gamma}=\left.\partial_{\bm{n}}\mu_{n}^{m+1/2}\right|_{\Gamma}=0,  \quad 
\left. ( \nabla_h \phi^{m+1} \cdot \bm{n} ) \right|_{\Gamma} = 0. 
\end{aligned}
\end{equation} 

\begin{lemma}
    At the initial time step, we could take a backward evaluation of the PDE system to obtain a locally second order accurate approximation to $n^{-1}$, $p^{-1}$ and $\u^{-1}$. In turn, a numerical implementation of the proposed algorithm \eqref{scheme: main} results in a second order local truncation error at $m=0$. 
\end{lemma}

It is clear that the mass conservation identity is valid for the ion concentration variables: 
\begin{equation}
    \overline{p^{m+1}}=\overline{p^{m}}=\cdots=\overline{p^{0}}, \quad 
    \overline{n^{m+1}}=\overline{n^{m}}=\cdots=\overline{n^{0}}.  \label{mass conserv-1} 
\end{equation} 
To simplify the notation in the later analysis, the following smooth function is introduced: 
\begin{equation}
    F_{a}(x):=\frac{G(x)-G(a)}{x-a}, \, \, \, G\left(x\right) = x \ln x , ~\forall x > 0,  \quad 
    \mbox{for any fixed $a >0$} , 
\end{equation} 
This notation leads to a rewritten form of \eqref{scheme: mu_n} and \eqref{scheme: mu_p}:
\begin{subequations}
    \begin{align}
        \mu_{n}^{m+1/2}
        = & F_{n^{m}} (n^{m+1} )-F_{n^{m}} (n^{m+1} )-1 
        + \tau (\ln n^{m+1}-\ln n^{m} )  
        \nonumber \\
        & 
        + (-\Delta_{h} )^{-1} (n^{m+1/2}-p^{m+1/2} ) , 
        \\
        \mu_{p}^{m+1/2} 
        = & F_{p^{m}} (p^{m+1} ) - F_{p^{m}} (p^{m+1} )-1 
        + \tau (\ln p^{m+1}-\ln p^{m} )
        \nonumber \\
        & + (-\Delta_{h} )^{-1} (p^{m+1/2}-n^{m+1/2} ) .  
    \end{align}
\end{subequations} 
Meanwhile, the following Calculus-style estimates will be frequently used in the later analysis. 
\begin{lemma}\label{lem: F} \cite{WangchengCH2021, WangchengCHNS2024, Liu2023PNP}
    Let $a>0$ be fixed, then
\begin{enumerate}
  \item $\displaystyle F_{a}^{\prime}(x)=\frac{G^{\prime}(x)(x-a)-(G(x)-G(a))}{(x-a)^{2}} \geq 0$, for any $x>0$.
  \item $F_{a}(x)$ is an increasing function of $x$, and $F_{a}(x) \leq F_{a}(a)=\ln a+1$ for any $0<x<a$.
\end{enumerate}
\end{lemma}

\section{The unique solvability and positivity-preserving property} \label{sec: positivity}

Since the implicit part of the numerical scheme~\eqref{scheme: main} corresponds to a monotone, singular, while non-symmetric nonlinear system, a four-step process is needed to establish its unique solvability and positivity preserving analysis.

{\bf Step 1:} 
A connection between $\hat{\bm{u}}^{m+1}$ and $(\mu_{n}^{m+1/2}, \mu_{p}^{m+1/2})$ is needed. The following equivalent form of \eqref{scheme: u_hat} is observed: 
\begin{subequations}
    \begin{align}
        & \frac{2\hat{\bm{u}}^{m+1/2} - 2\bm{u}^{m}}{\tau}
        + \frac{1}{2} \Big(\tilde{\bm{u}}^{m+1/2} \cdot \nabla_{h} \hat{\bm{u}}^{m+1/2}
        + \nabla_{h} \cdot \Big(\hat{\bm{u}}^{m+1/2} (\tilde{\bm{u}}^{m+1/2} )^{T} \Big) \Big)
        + \nabla_{h} \psi^{m}
        - \Delta_{h} \hat{\bm{u}}^{m+1/2} 
        \nonumber\\
        &  = - \mathcal{A}_{h} \tilde{p}^{m+1/2} \nabla_{h} \mu_{p}^{m+1/2} 
        - \mathcal{A}_{h} \tilde{n}^{m+1/2} \nabla_{h} \mu_{n}^{m+1/2}, 
        \label{scheme: u step1.1}\\
        & \hat{\bm{u}}^{m+1} = 2 \hat{\bm{u}}^{m+1/2}-\bm{u}^{m}. 
        \label{scheme: u step1.2}
\end{align}
\end{subequations} 
Of course, for any given field $(\mu_{n},\mu_{p} )$, a velocity vector $\bm{v}=\mathcal{L}_{h}^{N S} (\mu_{n},\mu_{p} )$ could be defined as the unique solution of the following discrete convection-diffusion equation:
\begin{equation}\label{def: L_h^NS}
    \begin{aligned}
        & \frac{2\bm{v} - 2\bm{u}^{m}}{\tau}
        + \frac{1}{2} \Big( \tilde{\bm{u}}^{m+1/2} \cdot \nabla_{h} \bm{v}
        + \nabla_{h} \cdot (\bm{v} (\tilde{\bm{u}}^{m+1/2} )^{T} ) \Big)
         + \nabla_{h} \psi^{m}  - \Delta_{h} \bm{v} 
        \\ 
        & = - \mathcal{A}_{h} \tilde{p}^{m+1/2} \nabla_{h} \mu_{p}^{m+1/2} 
        - \mathcal{A}_{h} \tilde{n}^{m+1/2} \nabla_{h} \mu_{n}^{m+1/2}. 
    \end{aligned}
\end{equation}
Subsequently, the intermediate velocity vector is obtained as $\hat{\bm{u}}^{m+1/2}=\mathcal{L}_{h}^{N S} ( \mu_{n}^{m+1/2},\mu_{p}^{m+1/2} )$, combined with the formula \eqref{scheme: u step1.2} for $\hat{\bm{u}}^{m+1}$. In addition, $\bm{u}^{m+1}$ becomes the discrete Helmholtz projection of $\hat{\bm{u}}^{m+1}$ into divergence-free space, as implied by \eqref{scheme: u}, \eqref{scheme: nabla_u}. 

{\bf Step 2:} 
A connection between $(n^{m+1},p^{m+1} )$ and $(\mu_{n}^{m+1/2},\mu_{p}^{m+1/2} )$ is needed in the further analysis. A substitution of $\hat{\bm{u}}^{m+1/2}= \mathcal{L}_{h}^{N S} (\mu_{n}^{m+1/2},\mu_{p}^{m+1/2} )$ into \eqref{scheme: n} and \eqref{scheme: p} gives 
\begin{subequations}\label{scheme: step2}
\begin{align}
    &\frac{n^{m+1}-n^{m}}{\tau}
    +\nabla_{h} \cdot \Big(\mathcal{A}_{h} \tilde{n}^{m+1/2} \mathcal{L}_{h}^{N S}
    (\mu_{n}^{m+1/2},\mu_{p}^{m+1/2} ) \Big)
    =\nabla_{h} \cdot \Big(\breve{n}^{m+1/2} \nabla_{h} \mu_{n}^{m+1/2} \Big) , 
    \label{num: n} \\
    &\frac{p^{m+1}-p^{m}}{\tau}
    +\nabla_{h} \cdot \Big( \mathcal{A}_{h} \tilde{p}^{m+1/2} \mathcal{L}_{h}^{N S}
    (\mu_{n}^{m+1/2},\mu_{p}^{m+1/2} ) \Big)
    =\nabla_{h} \cdot \Big(\breve{p}^{m+1/2} \nabla_{h} \mu_{p}^{m+1/2} \Big) .  
    \label{num: p}
\end{align}
\end{subequations} 
In turn, we define $\bm{\mu} = (\mu_{n},\mu_{p} )$, $\bm{c}= (n,p )$ and 
$\mathcal{L}_{h}^{NP}:2 (\mathbb{R}^{N^{2}} )^{2} \rightarrow 2 (\mathbb{R}^{N^{2}} )^{2}$ as  
\begin{subequations}
\begin{align}
&\mathcal{L}_{h}^{P} (\mu_{p} ) 
    = \nabla_{h} \cdot \Big(\mathcal{A}_{h} \tilde{p}^{m+1/2} \mathcal{L}_{h}^{N S} (\bm{\mu} ) \Big)
    - \nabla_{h} \cdot \Big( \breve{p}^{m+1/2} \nabla_{h} \mu_{p} \Big), \label{def: L_h^P}
\\
&\mathcal{L}_{h}^{N} (\mu_{n} )
    = \nabla_{h} \cdot \Big( \mathcal{A}_{h} \tilde{n}^{m+1/2} \mathcal{L}_{h}^{N S} (\bm{\mu} ) \Big)
    - \nabla_{h} \cdot \Big( \breve{n}^{m+1/2} \nabla_{h} \mu_{n} \Big). \label{def: L_h^N}
\end{align}
\end{subequations} 
To simplify the notation, the above system could be rewritten as 
\begin{equation}\label{def: L_h^NP}
    \mathcal{L}_{h}^{NP} (\bm{\mu} )
    = \nabla_{h} \cdot \Big(\mathcal{A}_{h} \tilde{\bm{c}}^{m+1/2} \mathcal{L}_{h}^{N S} (\bm{\mu} ) \Big)
    - \nabla_{h} \cdot \Big( \breve{\bm{c}}^{m+1/2} \nabla_{h} \bm{\mu} \Big).
\end{equation} 
Of course, $\mathcal{L}_{h}^{NP}$ is a linear operator, with either periodic or homogeneous Neumann boundary condition. Therefore, an equivalent representation of \eqref{scheme: step2} is available: 
\begin{equation}\label{scheme: c step2}
    \frac{\bm{c}^{m+1}-\bm{c}^{m}}{\tau}=-\mathcal{L}_{h}^{NP} (\bm{\mu}^{m+1/2} ). 
\end{equation} 

{\bf Step 3:} 
To facilitate the theoretical analysis, we have to prove that the operator $\mathcal{L}_{h}^{NP}$ is invertible, so that $(\mathcal{L}_{h}^{NP})^{-1}$ is well defined. Following similar ideas in \cite{WangchengCHNS2022}, we are able to derive the next two properties of $\mathcal{L}_{h}^{NP}$.

\begin{lemma}\label{lem: L_h^NP}
The linear operator $\mathcal{L}_{h}^{NP}$ preserves the monotonicity estimate: 
\begin{equation}\label{lem eqn: L_h^NS}
    \langle\mathcal{L}_{h}^{NP} (\bm{\mu}_{1} )-\mathcal{L}_{h}^{NP} (\bm{\mu}_{2} ), 
    \bm{\mu}_{1}-\bm{\mu}_{2} \rangle_{C} 
    \geq \left\|\sqrt{\breve{\bm{c}}^{m+1/2}} \nabla_{h} \left(\bm{\mu}_{1}-\bm{\mu}_{2}\right)\right\|_{2}^{2}  
    \geq 0 ,  
\end{equation} 
for any $\bm{\mu}_{1}, \bm{\mu}_{2}$. In addition, the equality is realized if and only if $\bm{\mu}_{1}=\bm{\mu}_{2}$, if we require $\overline{\bm{\mu}}_{1}=\overline{\bm{\mu}}_{2}=0$. As a result, the operator $\mathcal{L}_{h}^{NP}$ is invertible. 
\end{lemma} 
\begin{proof}
    Given $\bm{\mu}_1$, $\bm{\mu}_2$, a difference function is defined as $\bm{\mu}_{D}:=\bm{\mu}_{1}-\bm{\mu}_{2}$. Since $\mathcal{L}_{h}^{NP}$ is a linear operator, the following expansion becomes available:
    \begin{equation}\label{eqn: G_mu}
        \mathcal{L}_{h}^{NP} (\bm{\mu}_{1} )
        - \mathcal{L}_{h}^{NP} ( \bm{\mu}_{2} )
        = \mathcal{L}_{h}^{NP} ( \bm{\mu}_{D} )
        = \nabla_{h} \cdot \Big(\mathcal{A}_{h} \tilde{\bm{c}}^{m+1/2} \mathcal{L}_{h}^{N S} (\bm{\mu}_{D} ) \Big)
        - \nabla_{h} \cdot \Big(\breve{\bm{c}}^{m+1/2} \nabla_{h} \bm{\mu}_{D} \Big).
    \end{equation}
    Taking a discrete inner product with \eqref{eqn: G_mu} by $\bm{\mu}_{D}$ leads to 
    \begin{equation}
    \begin{aligned} 
        \left\langle\mathcal{L}_{h}^{NP}\left(\bm{\mu}_{D}\right), \bm{\mu}_{D}\right\rangle_{C}
        = & 
        - \left\langle\mathcal{A}_{h} \tilde{\bm{c}}^{m+1/2} \mathcal{L}_{h}^{N S}\left(\bm{\mu}_{D}\right), 
        \nabla_{h}\bm{\mu}_{D}\right\rangle_{1} 
        + \left\langle\breve{\bm{c}}^{m+1/2} \nabla_{h} \bm{\mu}_{D}, \nabla_{h} \bm{\mu}_{D}\right\rangle_{1} 
    \\
      = &      
           - \left\langle\mathcal{A}_{h} \tilde{\bm{c}}^{m+1/2} \nabla_{h}\bm{\mu}_{D}, 
            \mathcal{L}_{h}^{N S}\left(\bm{\mu}_{D}\right)\right\rangle_{1}
            + \left\langle\breve{\bm{c}}^{m+1/2} \nabla_{h} \bm{\mu}_{D}, \nabla_{h} \bm{\mu}_{D}\right\rangle_{1}. 
        \end{aligned}
        \label{eqn: inner mu}  
    \end{equation}
In addition, we define $\bm{v}_j :=\mathcal{L}_{h}^{N S} (\bm{\mu}_j )$, $j=1, 2$, and $\bm{v}_{D}:=\bm{v}_{1}-\bm{v}_{2}=\mathcal{L}_{h}^{NS} \left(\bm{u}_{D}\right)$, based on the linearity of $\mathcal{L}_{h}^{N S}$. Meanwhile, the definition of $\mathcal{L}_{h}^{N S}$ in \eqref{def: L_h^NS} indicates that
    \begin{equation}\label{eqn: difference v}
        \begin{aligned}
        & \frac{2\bm{v}_{D}}{\tau}
        + \frac{1}{2} \Big(\tilde{\bm{u}}^{m+1/2} \cdot \nabla_{h} \bm{v}_{D} 
        + \nabla_{h} \cdot \Big( \bm{v}_{D} (\tilde{\bm{u}}^{m+1/2} )^{T} \Big) \Big) 
        - \Delta_{h} \bm{v}_{D} \\
        & + \mathcal{A}_{h} \tilde{p}^{m+1/2} \nabla_{h} \mu_{p,D}^{m+1/2} 
        + \mathcal{A}_{h} \tilde{n}^{m+1/2} \nabla_{h} \mu_{n,D}^{m+1/2} = 0. 
    \end{aligned} 
    \end{equation}
Of course, the non-homogeneous source terms, namely $\bm{u}^{m}/\tau$ and $\nabla_{h} \psi^{m}$, disappear in this difference equation. Therefore, taking a discrete inner product with \eqref{eqn: difference v} by $\bm{v}_{D}=\mathcal{L}_{h}^{NS} \left(\bm{\mu}_{D}\right)$ yields 
    \begin{equation}\label{eqn: norm v}
        \frac{2}{\tau}\|\bm{v}_{D}\|_{2}^{2}
        +\left\|\nabla_{h} \bm{v}_{D}\right\|_{2}^{2}
        +\left\langle\mathcal{A}_{h} \tilde{\bm{c}}^{m+1/2} \nabla_{h}\bm{\mu}_{D}, 
            \mathcal{L}_{h}^{N S}\left(\bm{\mu}_{D}\right)\right\rangle_{1}=0,
    \end{equation}
in which the following identities have been used: 
    \begin{subequations}
        \begin{align}
            & \langle\tilde{\bm{u}}^{m+1/2} \cdot \nabla_{h} \bm{v}_{D} 
            + \nabla_{h} \cdot (\bm{v}_{D} (\tilde{\bm{u}}^{m+1/2} )^{T} ) , 
            \bm{v}_{D} \rangle_{1}=0, \\
            & -\left(\bm{v}_{D}, \Delta_{h} \bm{v}_{D}\right)
            =\left\|\nabla_{h} \bm{v}_{D}\right\|_{2}^{2} .
        \end{align}
    \end{subequations}
   A combination of \eqref{eqn: norm v} and \eqref{eqn: inner mu} results in 
    \begin{equation}
        \left\langle\mathcal{L}_{h}^{NP}\left(\bm{\mu}_{D}\right), \bm{\mu}_{D}\right\rangle_{C}
        = \frac{2}{\tau}\|\tilde{\bm{v}}\|_{2}^{2}
        + \left\|\nabla_{h} \tilde{\bm{v}}\right\|_{2}^{2} 
        + \left\|\sqrt{\breve{\bm{c}}^{m+1/2}} \nabla_{h} \bm{\mu}_{D}\right\|_{2}^{2},
    \end{equation} 
    or equivalently, 
    \begin{equation}
            \left\langle \mathcal{L}_{h}^{NP} \left(\bm{\mu}_{1}\right)
            - \mathcal{L}_{h}^{NP}\left(\bm{\mu}_{2}\right), 
            \bm{\mu}_{1}-\bm{\mu}_{2}\right\rangle_{C} 
            = \left\langle\mathcal{L}_{h}^{NP}\left(\bm{\mu}_{D}\right), \bm{\mu}_{D}\right\rangle_{C} 
            \geq \left\|\sqrt{\breve{\bm{c}}^{m+1/2}} \nabla_{h} \bm{\mu}_{D}\right\|_{2}^{2} 
            \geq 0,
    \end{equation}
so that \eqref{lem eqn: L_h^NS} has been proved. Meanwhile, it is clear that the equality is valid if and only if $\bm{\mu}_{D} \equiv 0$, i.e., $\bm{\mu}_{1}=\bm{\mu}_{2}$, under the requirement that $\overline{\bm{\mu}}_{1}=\overline{\bm{\mu}}_{2}=0$. The proof of Lemma~\ref{lem: L_h^NP} is completed.
\end{proof}

It is clear that the inverse operator $(\mathcal{L}_{h}^{NP} )^{-1}$ also maps $2\mathbb{R}^{N^{2}}$ into $2\mathbb{R}^{N^{2}}$, since the linear operator $\mathcal{L}_{h}^{NP}$ does. As a direct consequence of Lemma \ref{lem: L_h^NP}, the following monotonicity analysis is available.

\begin{proposition}\label{pro: G-1}
    The linear operator $\left(\mathcal{L}_{h}^{NP}\right)^{-1}$ also preserves the monotonicity estimate: 
    \begin{equation}\label{eqn: monotonicity G-1}
        \begin{aligned}
           \langle (\mathcal{L}_{h}^{NP} )^{-1} (\bm{c}_{1} )
            - (\mathcal{L}_{h}^{NP} )^{-1} (\bm{c}_{2} ), 
            \bm{c}_{1}-\bm{c}_{2} \rangle_{C} 
            & \geq \sum_{q=1}^{K}\left\|\sqrt{\mathcal{A}_{h}c_{q}^{m}}
            \nabla_{h} (\mu_{q}^{(1)} -\mu_{q}^{(2)} )\right\|_{2}^{2}  \\
            & \geq C_{1}^{2} \| (\mathcal{L}_{h}^{NP} )^{-1} (\bm{c}_{1}
            -\bm{c}_{2} ) \|_{2}^{2},
        \end{aligned}
    \end{equation}
    for any $\bm{c}_{1}, \bm{c}_{2}$, with $\overline{\bm{c}}_{1}=\overline{\bm{c}}_{2}=0$. In fact, the constant $C_{1}$ is associated with the minimum of $\bm{c}^{m}$, and the discrete Poincar\'e regularity, $\|\nabla_h f \|_w \ge C_0^{-1} \|f\|_2$, for any $f$ with $\overline{f}=0$, as indicated by Proposition~\ref{prop:1}. In addition, the equality is valid if and only if $\bm{c}_{1}=\bm{c}_{2}$.
\end{proposition}

\begin{proof}
    We denote $\bm{\mu}=\left(\mathcal{L}_{h}^{NP}\right)^{-1}\left(\bm{c}\right)$, an equivalent statement of $\bm{c} = \mathcal{L}_{h}^{NP}\left(\bm{\mu}\right)$. Therefore, an application of \eqref{lem eqn: L_h^NS} implies that
\begin{equation}
    \begin{aligned}
        \left\langle (\mathcal{L}_{h}^{NP} )^{-1} (\bm{c}_{1} )
        - (\mathcal{L}_{h}^{NP} )^{-1} (\bm{c}_{2} ), 
        \bm{c}_{1}-\bm{c}_{2} \right\rangle_{C} 
        & =\left\langle\mathcal{L}_{h}^{NP} (\bm{\mu}_{1} )
        -\mathcal{L}_{h}^{NP} (\bm{\mu}_{2} ), 
        \bm{\mu}_{1}-\bm{\mu}_{2}\right\rangle_{C} 
        \\
        &\geq \|\sqrt{\breve{\bm{c}}^{m+1/2}} \nabla_{h} \bm{\mu}_{D} \|_{2}^{2} 
        \geq C_{1}^{2} \|\bm{\mu}_{1}-\bm{\mu}_{2} \|_{2}^{2} 
        \\
        & =C_{1}^{2} \| (\mathcal{L}_{h}^{NP} )^{-1} (\bm{c}_{1}
        -\bm{c}_{2} ) \|_{2}^{2} \geq 0 , 
    \end{aligned}
\end{equation}
with $C_1 = ( \min \breve{\bm{c}}^{m+1/2} )^\frac12 C_0^{-1}$. Of course, the equality is valid if and only if $\bm{c}_{1}=\bm{c}_{2}$. The proof of Proposition \ref{pro: G-1} is finished.
\end{proof} 

Based on the construction \eqref{def: L_h^NS} and the definition \eqref{def: L_h^NP} for $\mathcal{L}_{h}^{NP}$, the following homogenization formula will be helpful in the later analysis:
\begin{equation}\label{eqn: decomposion L_h^NP}
\mathcal{L}_{h}^{NP} (\bm{\mu} ) 
= \mathcal{L}_{h,1}^{NP} (\bm{\mu} )  + \mathcal{L}_{h,2}^{NP},  \quad 
\mathcal{L}_{h,2}^{NP}
= \nabla_{h} \cdot\Big(\mathcal{A}_{h} \tilde{\bm{c}}^{m+1/2} (\bm{u}^{m}-\frac{\tau}{2} \nabla_{h} \psi^m ) \Big),  \text { for any } \overline{\bm{\mu}}=0,
\end{equation}
in which $\mathcal{L}_{h,1}^{NP}$ corresponds to a homogeneous linear operator. In more details, such a homogeneous operator satisfies the linearity property, in comparison with the operator $\mathcal{L}_{h}^{NS}$ given by \eqref{def: L_h^NS}. Moreover, the following $\| \cdot \|_2$ bound could always be assumed for the non-homogeneous source term, dependent on the numerical solution at the previous time steps: 
\begin{equation}
    \left\|\mathcal{L}_{h,2}^{NP}\right\|_{2} \leq A^{*} .
\end{equation}

\begin{proposition}\label{estimate: G-1}
    For any $\bm{c}$ with $\overline{\bm{c}}=0$, the following $\|\cdot\|_{2}$ 
    bound is valid: 
    \begin{align}
         \| (\mathcal{L}_{h}^{NP} )^{-1} (\bm{c} ) \|_{2} 
        & \leq C_{1}^{-2} ( \|\bm{c} \|_{2}+A^* ) . 
        \label{estimate:L_h^NP L2} 
    \end{align}
\end{proposition} 
\begin{proof}
    We denote $\bm{\mu} = \left( \mathcal{L}_{h}^{NP} \right)^{-1}\left(\bm{c}\right)$, for any $\bm{c}$ with $\overline{\bm{c}}=0$. The homogenization decomposition \eqref{eqn: decomposion L_h^NP} implies that $\mathcal{L}_{h,1}^{NP} (\bm{\mu} ) = \bm{c}_{D}:=\bm{c}-\mathcal{L}_{h,2}^{NP}$. On the other hand, the monotonicity estimate \eqref{lem eqn: L_h^NS} indicates that 
\begin{subequations}
    \begin{align}
        \left\langle \mathcal{L}_{h,1}^{NP} \left(\bm{\mu}\right), 
        \bm{\mu}\right\rangle_{C} 
        & \geq \left\|\sqrt{\breve{\bm{c}}^{m+1/2}} \nabla_{h} \bm{\mu} \right\|_{2}^{2} 
        \geq C_{1}^{2}\left\|\bm{\mu}\right\|_{2}^{2}, \quad \text { so that } \\
        \left\|\bm{\mu}\right\|_{2}^{2} 
        & \leq C_{1}^{-2}\left\langle \mathcal{L}_{h,1}^{NP} \left(\bm{\mu}\right), \bm{\mu}\right\rangle_{C} 
        \leq C_{1}^{-2}\left\|\mathcal{L}_{h,1}^{NP} \left(\bm{\mu}\right)\right\|_{2} 
        \cdot\left\|\bm{\mu}\right\|_{2}, \\
        \left\|\bm{\mu}\right\|_{2} 
        & \leq C_{1}^{-2}\left\|\mathcal{L}_{h,1}^{NP} \left(\bm{\mu}\right)\right\|_{2},
\end{align}
\end{subequations}
with an application of Cauchy inequality. Therefore, the following inequality is available: 
\begin{equation}
    \begin{aligned}
        \left\|\left(\mathcal{L}_{h}^{NP}\right)^{-1}\left(\bm{c}\right)\right\|_{2} 
        & =\left\|\bm{\mu}\right\|_{2} 
        \leq C_{1}^{-2}\left\|\mathcal{L}_{h,1}^{NP} \left(\bm{\mu}\right)\right\|_{2}
        =C_{1}^{-2}\left\|\bm{c}-\mathcal{L}_{h,2}^{NP} \right\|_{2} 
        \leq C_{1}^{-2} ( \|\bm{c} \|_{2} + \|\mathcal{L}_{h,1}^{NP} \|_{2} ),
    \end{aligned}
\end{equation}
which is exactly \eqref{estimate:L_h^NP L2}. 
The proof of Proposition \ref{estimate: G-1} is completed. 
\end{proof} 

Based on \eqref{num: p}, \eqref{num: n} and \eqref{def: L_h^NP}, \eqref{scheme: c step2}, we conclude that the numerical solution \eqref{scheme: n} - \eqref{scheme: u} could be equivalently represented as the following nonlinear system, in terms of $\bm{c}^{m+1}$ :
\begin{equation}\label{eqn: mu+c}
    \begin{aligned}
        & \bm{\mu}^{m+1}
        + \frac{1}{\tau} \left(\mathcal{L}_{h}^{NP}\right)^{-1}\left(\bm{c}^{m+1}-\bm{c}^{m}\right) \\
        = & \frac{\bm{c}^{m+1} \ln \bm{c}^{m+1}-\bm{c}^{m} \ln \bm{c}^{m}}{\bm{c}^{m+1}-\bm{c}^{m}}-1
        + \tau \ln \frac{\bm{c}^{m+1}}{\bm{c}^{m}}  
        + (-\Delta_{h} )^{-1} \Big(\bm{c}^{m+1/2}-M\bm{c}^{m+1/2} \Big) 
        \\
        & 
        + \frac{1}{\tau}\left(\mathcal{L}_{h}^{NP}\right)^{-1}\left(\bm{c}^{m+1}-\bm{c}^{m}\right)=0,  \, \, \, 
        M = 
    \begin{pmatrix}
        \bm{0} & \bm{1} \\
        \bm{1} & \bm{0}
\end{pmatrix}.
    \end{aligned}
\end{equation}

{\bf Step 4:} We are going to prove the existence of $\bm{c}^{m+1}$ in \eqref{eqn: mu+c}. Because of the fact that the operator $(\mathcal{L}_{h}^{NP} )^{-1}$ is non-symmetric, a direct application of the discrete energy minimization technique does not work out. Moreover, the Browder-Minty lemma is not directly available to this system, either, which comes from the singularity of $\ln \bm{c} $ as $\bm{c} \rightarrow 0$. To overcome these subtle difficulties, we have to construct a fixed point sequence to justify the analysis; similar ideas have been reported in~\cite{WangchengCHNS2022, WangchengCHNS2024} to deal with Flory-Huggins-Cahn-Hilliard-Navier-Stokes system. 

Define the nonlinear iteration, at the $(k+1)$-th stage: 
\begin{subequations}\label{def: iteration}
    \begin{align}
        \mathcal{G}_{h} (n^{(k+1)} ) 
        := & F_{n^{m}} (n^{(k+1)} )-1 
        + \tau (\ln n^{(k+1)} -\ln n^{m} ) 
        \nonumber \\
        & + (-\Delta_{h} )^{-1} \Big( \frac{1}{2} (n^{(k+1)} + n^{m} )-\frac{1}{2} (p^{(k+1)} + p^{m} ) \Big)
        + A n^{(k+1)}
        \nonumber \\
        = & - \frac{1}{\tau} (\mathcal{L}_{h}^{NP} )^{-1} (n^{(k)}-n^{m} ) 
        + A n^{(k)}, 
        \text { with } n^{(0)}=n^{m}, 
        \label{def: Gh n} \\
        \mathcal{G}_{h} (p^{(k+1)} ) 
        := & F_{p^{m}} (p^{(k+1)} )-1 + \tau (\ln p^{(k+1)} -\ln p^{m} ) 
        \nonumber \\
        & + (-\Delta_{h} )^{-1} \Big(\frac{1}{2} (p^{(k+1)} + p^{m} )-\frac12 (n^{(k+1)} + n^{m} ) \Big)
        + A p^{(k+1)}
        \nonumber \\
        = & - \frac{1}{\tau} (\mathcal{L}_{h}^{NP} )^{-1} (p^{(k)}-p^{m} )
        + A p^{(k)}, 
        \text { with } p^{(0)}=p^{m} . 
        \label{def: Gh p}
    \end{align}
\end{subequations} 
The unique solvability and positivity-preserving property of the numerical system~\eqref{def: iteration}, 
at each iteration stage, is stated in following proposition. 
The proof follows similar ideas as in \cite{Chen2019Positivity}, 
and the technical details are skipped for simplicity of presentation. 

\begin{proposition}
    Given cell-centered functions $p^{m}, ~p^{m-1}, ~p^{(k)}$ 
    and $n^{m}, ~n^{m-1}, ~n^{(k)}$, with a positivity condition 
    $p^{m},~p^{m-1},~n^{m}, ~n^{m-1}>0$, and 
    $\overline{p^{m}}=\overline{p^{m-1}}=\overline{p^{(k)}}=\beta_0 <1$, 
    $\overline{n^{m}}=\overline{n^{m-1}}=\overline{n^{(k)}}=\beta_0 <1$, 
    then there exists a unique solution $n^{(k+1)}$ to \eqref{def: Gh n} 
    and $p^{(k+1)}$ to \eqref{def: Gh p}, with $p^{(k+1)}>0$, $n^{(k+1)}>0$, 
    at a point-wise level, 
    and $\overline{p^{(k+1)}}=\overline{n^{(k+1)}}=\beta_0$. 
    Meanwhile, by the fact that $p^{m}, ~p^{m-1}, ~n^{m}, ~n^{m-1}$ 
    are discrete variables, 
    there is $0<\delta_{m-1}, ~\delta_{m}, ~\delta_{(k)}<1/2$, 
    such that $p^{m}, \, n^m \geq \delta_{m}$, $p^{m-1} , \, n^{m-1}  \geq \delta_{m-1}$, $p^{(k)} , \, n^{(k)} \geq \delta_{(k)}$. In addition, $p^{(k+1)}$ and $n^{(k+1)}$ preserves the estimate $p^{(k+1)} , \,  n^{(k+1)} \geq \delta_{(k+1)}$, where $\delta_{(k+1)}=\min (1/2, \hat{\delta})$ and $\hat{\delta}$ satisfies the following equality:
    \begin{equation}\label{pro: C^*}
        \begin{aligned}
            & \tau (\ln \hat{\delta} - \ln \beta_{0} )
            + 2 \tau | \ln \delta_m | + 1 
            - \frac{2 ( G (\beta_{0} )-G ( \delta_m ) )}{\beta_{0} - \delta_m}
            + C^{*} = 0, 
        \end{aligned}
    \end{equation}  
with $C^{*}= ( \tau^{-1} C_{1}^{-2} h^{-\frac{d}{2}} + C_0) ( \max( \| n^{(k)} - n^m \|_2 , \| p^{(k)} - p^m \|_2 ) + A^* ) +A \beta_0 |\Omega |h^{-2}$. 
\end{proposition}


The main result of this section is stated below.
\begin{theorem} \label{thm: solvability} 
    Given cell-centered functions $n^{m}, \, n^{m-1}, \, p^{m} , \, p^{m-1} > 0$, at a point-wise level, and $\overline{n^{m}}=$ $\overline{n^{m-1}}=\overline{p^{m}}=$ $\overline{p^{m-1}}=\beta_{0}$, then there exists the unique cell-centered solution $n^{m+1}$ and $p^{m+1}$ to \eqref{scheme: main}, with $n^{m+1} ,  \, p^{m+1} > 0$, at a point-wise level, and $\overline{n^{m+1}}=\overline{p^{m+1}}=\beta_{0}$.
\end{theorem} 
\begin{proof}
With $n^{(0)}=n^{m}$, $p^{(0)}=p^{m}$, the iteration solution generated by \eqref{def: iteration} satisfies  $n^{(1)}, p^{(1)} \ge \delta_{(1)} = \min (\frac12, \hat{\delta}_{(1)} )$, where $\hat{\delta}_{(1)}$ satisfies equality \eqref{pro: C^*}. In more details, the following inequality is observed 
\begin{equation}
\begin{aligned} 
    C^{*}= &  ( \tau^{-1} C_{1}^{-2} h^{-\frac{d}{2}} + C_0) ( \max( \| n^{(k)} - n^m \|_2 , \| p^{(k)} - p^m \|_2 ) + A^* ) +A \beta_0 |\Omega |h^{-2} 
\\
    \le & 
     ( \tau^{-1} C_{1}^{-2} h^{-\frac{d}{2}} + C_0 ) ( 2 M_h |\Omega|^{\frac{1}{2}} + A^* ) 
     +A \beta_0 |\Omega |h^{-2} =: \hat{C}^{*} ,   \quad 
     M_h = h^{-2} \beta_0. 
\end{aligned} 
\end{equation} 
Notice that $\| n^m \|_\infty , \, \| p^m \|_\infty , \, \| n^{(k)} \|_\infty , \, \| p^{(k)} \|_\infty \le M_h$, for any $k \ge 0$, as long as these numerical solutions stay positive. Subsequently, we are able to replace $C^{*}$ by $\hat{C}^{*}$, and obtain a modified equality of \eqref{pro: C^*}, so that $\hat{\delta}$ becomes a constant independent of the iteration stage $k$. Therefore, if the sequence generated by \eqref{def: iteration} has a limit, denoted as $(n^{m+1}, p^{n+1})$, its lower bound has to satisfy \eqref{pro: C^*}, with $C^{*}$ replaced by $\hat{C}^{*}$, since the later one is independent of iteration stage $k$.  

Next, we have to prove that $\mathcal{G}_{h}$ is a contraction mapping, so that the existence analysis could be derived by taking $k \rightarrow+\infty$ on both sides of \eqref{def: Gh n} and \eqref{def: Gh p}. To perform such an analysis, the difference function between two consecutive iteration stages is defined as 
\begin{equation}
    \zeta^{(k)}_{n}:=n^{(k)}-n^{(k-1)}, \quad 
    \zeta^{(k)}_{p}:=p^{(k)}-p^{(k-1)}, \text { for } k \geq 1. 
\end{equation} 
By the fact that $\overline{n^{(k)}}=\overline{n^{(k-1)}}= \overline{p^{(k)}}=\overline{p^{(k-1)}}=\beta_0$, we immediately get $\overline{\zeta^{(k)}_{n}} = 0$ and $\overline{\zeta^{(k)}_{p}} = 0$.

Taking a difference of \eqref{def: Gh n} and \eqref{def: Gh p}, between the $k^{\text {th}}$ and $(k+1)^{\text {st}}$ solutions, yields 
\begin{subequations}
    \begin{align}
        & \mathcal{G}_{h} (n^{(k+1)} )-\mathcal{G}_{h} (n^{(k)} ) 
        \nonumber \\
        = & F_{n^{m}} (n^{(k+1)} ) -F_{n^{m}} (n^{(k)} )
        + \tau (\ln n^{(k+1)}-\ln n^{(k)} ) + A \zeta^{(k+1)}_{n} 
        + (-\Delta_h )^{-1} \zeta^{(k+1)}_{n} 
        \nonumber \\
        = & -\frac{1}{\tau} (\mathcal{L}_{h}^{NP} )^{-1} (\zeta^{(k)}_{n} ) + A \zeta^{(k)}_{n}, 
        \label{dif: n} \\
        & \mathcal{G}_{h} (p^{(k+1)} )-\mathcal{G}_{h} (p^{(k)} ) 
        \nonumber \\
        = & F_{p^{m}} (p^{(k+1)} ) -F_{p^{m}} (p^{(k)} )
        + \tau (\ln p^{(k+1)}-\ln p^{(k)} ) 
        + A \zeta^{(k+1)}_{p} 
        + (-\Delta_h )^{-1} \zeta^{(k+1)}_{p} 
        \nonumber \\
        = & -\frac{1}{\tau} (\mathcal{L}_{h}^{NP} )^{-1} (\zeta^{(k)}_{p} )
        + A \zeta^{(k)}_{p}. 
        \label{dif: p}
    \end{align}
\end{subequations} 

Taking a discrete inner product with \eqref{dif: n} and \eqref{dif: p}, by $\zeta^{(k+1)}_{n}$ and $\zeta^{(k+1)}_{p}$ separately, gives 
\begin{subequations} \label{iteration est-1} 
    \begin{align}
        & \left\langle F_{n^{m}} (n^{(k+1)} ) -F_{n^{m}} (n^{(k)} ), \zeta^{(k+1)}_{n} \right\rangle_{C} 
         + \tau \left\langle \ln n^{(k+1)} -\ln n^{(k)}, \zeta^{(k+1)}_{n}\right\rangle_{C} 
         + A \| \zeta^{(k+1)}_{n} \|_2^2  
        \nonumber \\ 
        & 
        + \| \zeta^{(k+1)}_{n} \|_{-1, h}^2  
        = -\frac{1}{\tau} \langle (\mathcal{L}_{h}^{NP} )^{-1} (\zeta^{(k)}_{n} ), \zeta^{(k+1)}_{n} \rangle_{C} 
        + A \langle \zeta^{(k)}_{n} , \zeta^{(k+1)}_{n} \rangle_C , 
        \label{ite: n} \\
        & \left\langle F_{p^{m}} (p^{(k+1)} ) - F_{p^{m}} (p^{(k)} ), \zeta^{(k+1)}_{p} \right\rangle_{C} 
        + \tau \left\langle \ln p^{(k+1)}-\ln p^{(k)}, \zeta^{(k+1)}_{p} \right\rangle_{C} 
        + A \| \zeta^{(k+1)}_{p} \|_2^2 
        \nonumber \\ 
        & 
        + \| \zeta^{(k+1)}_{p} \|_{-1, h}^2 
        = -\frac{1}{\tau} \langle (\mathcal{L}_{h}^{NP} )^{-1} \zeta^{(k)}_{p} ), \zeta^{(k+1)}_{p} \rangle_{C} 
        + A \langle \zeta^{(k)}_{p}, \zeta^{(k+1)}_p \rangle_C .  
        \label{ite: p}
    \end{align}
\end{subequations} 
As an application of Lemma \ref{lem: F}, combined with the monotonicity of the logarithmic function, it is clear that the first two terms of \eqref{ite: n} and \eqref{ite: p} have to be non-negative:
\begin{subequations}\label{ite: F}
    \begin{align}
        & \left\langle F_{p^{m}} (p^{(k+1)} ) -F_{p^{m}} (p^{(k)} ), 
        \zeta^{(k+1)}_{p} \right\rangle_{C} \geq 0,  \quad 
        \left\langle F_{n^{m}} (n^{(k+1} ) -F_{n^{m}} (n^{(k)} ), 
        \zeta^{(k+1)}_{n} \right\rangle_{C} \geq 0, 
        \\
        & 
        \langle\ln p^{(k+1)}-\ln p^{(k)}, \zeta^{(k+1)}_{p} \rangle_{C} \geq 0,  \quad 
        \langle\ln n^{(k+1)}-\ln n^{(k)}, \zeta^{(k+1)}_{n} \rangle_{C} \geq 0.
\end{align}
\end{subequations} 
In terms of the iteration relaxation, an application of triangular equality reveals that 
\begin{subequations}\label{ite: zeta}
    \begin{align}
        \langle\zeta^{(k+1)}_{n}, \zeta^{(k+1)}_{n}-\zeta^{(k)}_{n} \rangle_{C} 
        = & \frac{1}{2} \Big( \|\zeta^{(k+1)}_{n} \|_2^{2}- \|\zeta^{(k)}_{n} \|_{2}^{2}
        + \|\zeta^{(k+1)}_{n}-\zeta^{(k)}_{n} \|_{2}^{2} \Big), 
        \\ 
         \langle\zeta^{(k+1)}_{p}, \zeta^{(k+1)}_{p}-\zeta^{(k)}_{p} \rangle_{C} 
        = & \frac{1}{2} \Big( \|\zeta^{(k+1)}_{p} \|_2^{2}- \|\zeta^{(m)}_{p} \|_{2}^{2}
        + \|\zeta^{(k+1)}_{p}-\zeta^{(k)}_{p} \|_{2}^{2} \Big). 
    \end{align}
\end{subequations}
Regarding the two electronic potential diffusion terms, an application of inverse inequality in  Proposition~\ref{prop:1} implies that 
\begin{equation} 
   \| \zeta^{(k+1)}_{n} \|_{-1, h}^2  \ge \breve{C}_0^{-2} h^2 \| \zeta^{(k+1)}_{n} \|_2^2 ,  \quad 
   \| \zeta^{(k+1)}_{p} \|_{-1, h}^2  \ge \breve{C}_0^{-2} h^2 \| \zeta^{(k+1)}_{p} \|_2^2  . 
\end{equation} 
The right hand side terms of \eqref{ite: n} and \eqref{ite: p} are related to the asymmetric operator $\left(\mathcal{L}_{h}^{NP}\right)^{-1}$. The following bounds could be derived: 
\begin{subequations}\label{ite: L_h^NP}
        \begin{align}
            \langle (\mathcal{L}_{h}^{NP} )^{-1} (\zeta^{(k)}_{n} ), \zeta^{(k+1)}_{n} \rangle_{C} 
            & = \langle (\mathcal{L}_{h}^{NP} )^{-1} (\zeta^{(k)}_{n} ), \zeta^{(k)}_{n} \rangle_{C}
            + \langle (\mathcal{L}_{h}^{NP} )^{-1} (\zeta^{(k)}_{n} , \zeta^{(k+1)}_{n}-\zeta^{(k)}_{n} \rangle_{C} 
            \nonumber \\
            & \ge \langle (\mathcal{L}_{h}^{NP} )^{-1} (\zeta^{(k)}_{n} ), 
            \zeta^{(k+1)}_{n}-\zeta^{(k)}_{n} \rangle_{C} 
            \nonumber \\
            & \ge - \| (\mathcal{L}_{h}^{NP} )^{-1} (\zeta^{(k)}_{n} ) \|_{2} \cdot 
            \|\zeta^{(k+1)}_{n}-\zeta^{(k)}_{n} \|_{2} 
            \nonumber \\
            & \ge -C_1^{-2} \|\zeta^{(k)}_{n} \|_{2} \cdot \|\zeta^{(k+1)}_{n}-\zeta^{(k)}_{n} \|_{2} 
            \nonumber \\
            & \ge - \frac12 \breve{C}_0^{-2} \tau h^2 \|\zeta^{(k)}_{n} \|_{2}^{2} 
            - \frac{\breve{C}_0^2 C_1^4}{2 \tau h^2} \|\zeta^{(k+1)}_{n}-\zeta^{(k)}_{n} \|_{2}^{2}, 
            \label{ite: L_h_N} 
            \\
            \langle (\mathcal{L}_{h}^{NP} )^{-1} (\zeta^{(k)}_{p} ), \zeta^{(k+1)}_{p} \rangle_{C} 
            & \ge - \frac12 \breve{C}_0^{-2} \tau h^2 \|\zeta^{(k)}_{p} \|_{2}^{2} 
            - \frac{\breve{C}_0^2 C_1^4}{2 \tau h^2} \|\zeta^{(k+1)}_{p}-\zeta^{(k)}_{p} \|_{2}^2 ,  \, \, \, 
            \mbox{(similarly)} .  
            \label{ite: L_h_P}
\end{align}
\end{subequations}
Notice that the inequality used in the fourth step, $\| (\mathcal{L}_{h}^{NP} )^{-1} (\zeta^{(k)}_{n} ) \|_{2}  \le C_1^{-2}  \|\zeta^{(k)}_{n} \|_{2}$, comes from estimate~\eqref{estimate:L_h^NP L2} in Proposition~\ref{estimate: G-1}, combined with the fact that all the non-homogeneous parts have been cancelled. Therefore, a substitution of \eqref{ite: F}-\eqref{ite: L_h^NP} into~\eqref{iteration est-1} results in 
\begin{equation}
    \begin{aligned} 
        & (\frac{A}{2}+ \breve{C}_0^{-2}  h^2 ) \|\zeta^{(k+1)}_{n} \|_{2}^{2}
        + \frac{A}{2} \|\zeta^{(k+1)}_{n}-\zeta^{(k)}_{n} \|_{2}^{2} 
        \le (\frac{A}{2}+ \frac12 \breve{C}_0^{-2}  h^2 ) \|\zeta^{(k)}_{n} \|_{2}^{2} 
        + \frac{\breve{C}_0^2 C_1^4}{2 \tau^2 h^2} \|\zeta^{(k+1)}_{n}-\zeta^{(k)}_{n} \|_{2}^{2}, 
        \\
        &  (\frac{A}{2}+ \breve{C}_0^{-2}  h^2 ) \|\zeta^{(k+1)}_{p} \|_{2}^{2}
        + \frac{A}{2} \|\zeta^{(k+1)}_{p}-\zeta^{(k)}_{p} \|_{2}^{2} 
        \le (\frac{A}{2}+ \frac12 \breve{C}_0^{-2}  h^2 ) \|\zeta^{(k)}_{p} \|_{2}^{2} 
        + \frac{\breve{C}_0^2 C_1^4}{2 \tau^2 h^2} \|\zeta^{(k+1)}_{p}-\zeta^{(k)}_{p} \|_{2}^{2} .  
    \end{aligned}
    \label{iteration est-2} 
\end{equation} 
As a result, by taking $A \geq A_{0}:= \breve{C}_0^2 C_1^4 \tau^{-2} h^{-2}$, a constant that may depend on $\tau$, $h$ and $\Omega$, and setting $B_0 = \frac{A}{2}+ \breve{C}_0^{-2}  h^2$, $B_1 = \frac{A}{2}+ \frac12 \breve{C}_0^{-2}  h^2$, we arrive at the following inequality:
\begin{equation} 
        B_0 \|\zeta^{(k+1)}_{n} \|_{2}^{2} \le B_1  \|\zeta^{(k)}_{n} \|_{2}^{2}  , \quad 
        B_0 \|\zeta^{(k+1)}_{p} \|_{2}^{2} \le B_1 \|\zeta^{(k)}_{p} \|_{2}^{2} . 
\end{equation}
Consequently, the nonlinear iteration \eqref{def: iteration} is guaranteed to be a contraction mapping, due to the fact that $B_0 > B_1$. This finishes the proof of Theorem~\ref{thm: solvability}.
\end{proof} 

\section{Total energy stability analysis}\label{sec: stability}
In the finite difference setting, the discrete energy is defined as 
\begin{equation}
    \begin{aligned}
        & E_{h} (n, p) := \langle n ( \ln n - 1) + p ( \ln p -1 ) , \mathrm{1}\rangle_C 
        +\frac{1}{2}\|n-p\|_{-1, h}^{2}, \, \, \,  
        E_{h, total} (n , p , \bm{u} ) := E_{h} ( n , p ) + \frac{1}{2} \|\bm{u} \|_{2}^{2}. 
        \label{total energy-1} 
    \end{aligned}
\end{equation}
The total energy dissipation law is stated in the following theorem.
\begin{theorem}\label{thm: energy}
    The following inequality is valid for the numerical solution of \eqref{scheme: main}, for any $m \ge 0$:
\begin{equation}
    \begin{aligned}
        & \tilde{E}_{h}\left(n^{m+1}, p^{m+1}, \bm{u}^{m+1}, \psi^{m+1}\right) 
        - \tilde{E}_{h}\left( n^{m}, p^{m}, \bm{u}^{m}, \psi^{m} \right)  \\
        = & 
         - \tau \|\nabla_h \bar{\hat{\bm{u}}}^{m+1/2} \|_2^2  
        - \tau ( [\breve{n}^{m+1/2} \nabla_h \mu_{n}^{m+1/2}, \nabla_h \mu_{n}^{m+1/2} ]
        + [\breve{p}^{m+1/2} \nabla_h \mu_{p}^{m+1/2}, \nabla_h \mu_{p}^{m+1/2} ] ), 
\\
   & \mbox{with} \, \, \, 
        \tilde{E}_{h} (n^{m+1}, p^{m+1}, \bm{u}^{m+1}, \psi^{m+1} ) 
         = E_{h, total} (n^{m+1}, p^{m+1}, \bm{u}^{m+1} ) + \frac{\tau^{2}}{8} \|\nabla_h \psi^{m+1} \|_2^2 .
    \end{aligned} 
    \label{mod energy-1} 
\end{equation} 
\end{theorem}
\begin{proof}
A discrete inner product with \eqref{scheme: u} by $\hat{\bm{u}}^{m+1/2}=\frac{1}{2} (\hat{\bm{u}}^{m+1}+\bm{u}^{m} )$ leads to 
    \begin{equation}\label{egy: u_hat}
        \begin{aligned} 
            & \frac{ \|\hat{\bm{u}}^{m+1}\ \|_{2}^{2}- \|\bm{u}^{m} \|_{2}^{2}}{2 \tau}
            + \langle\nabla_h \psi^{m}, \bm{\hat{\bm{u}}}^{m+1/2} \rangle_{1} 
            + \|\nabla_{h} \hat{\bm{u}}^{m+1/2} \|_{2}^{2}
            \\
            & + \langle\mathcal{A}_{h} \tilde{p}^{m+1/2} \nabla_{h} \mu_{p}^{m+1/2}, \hat{\bm{u}}^{m+1/2} \rangle_{1}
            -  \langle\mathcal{A}_{h} \tilde{n}^{m+1/2} \nabla_{h} \mu_{n}^{m+1/2}, 
            \hat{\bm{u}}^{m+1/2} \rangle_{1} = 0 , 
        \end{aligned} 
    \end{equation}
with an application of the summation-by-parts formula \eqref{lem: discrete inner}:
\begin{equation}
    \langle\hat{\bm{u}}^{m+1/2}, \tilde{\bm{u}}^{m+1/2} \cdot \nabla_{h} \hat{\bm{u}}^{m+1/2}
    +\nabla_{h} \cdot (\hat{\bm{u}}^{m+1/2} (\tilde{\bm{u}}^{m+1/2} )^{T} ) \rangle_{1}=0. 
\end{equation} 
Meanwhile, based on the summation by part formula \eqref{lem: discrete inner of u and f}, a discrete inner product with \eqref{scheme: p} by $\bm{u}^{m+1}$ gives 
\begin{equation}\label{egy: u1}
\begin{aligned}
    &  \|\bm{u}^{m+1} \|_2^2  - \|\hat{\bm{u}}^{m+1} \|_2^2 
    + \|\bm{u}^{m+1}-\hat{\bm{u}}^{m+1} \|_2^2 
    \\
    = & \|\bm{u}^{m+1} \|_2^2 
    -  \|\bar{\hat{\bm{u}}}^{m+1/2} \|_2^2 
    + \frac{1}{4}\tau^{2} \|\nabla_h (\psi^{m+1}-\psi^{m} ) \|_2^2 = 0 . 
    \end{aligned}
\end{equation} 
Subsequently, a combination of \eqref{egy: u_hat} and \eqref{egy: u1} yields 
\begin{equation}\label{egy: u and p}
    \begin{aligned}
        & \frac{\|\bm{u}^{m+1}\|_{2}^{2}- \|\bm{u}^{m} \|_{2}^{2}}{2 \tau} 
        + \langle\nabla_{h} \psi^{m}, \hat{\bm{u}}^{m+1/2} \rangle_{1}
        + \frac{1}{8} \tau \|\nabla_h (\psi^{m+1}-\psi^{m} ) \|_2^2  
        + \|\nabla_{h} \hat{\bm{u}}^{m+1/2} \|_2^2 \\
        & + \langle\mathcal{A}_{h} \tilde{p}^{m+1/2} \nabla_{h} \mu_{p}^{m+1/2}, \hat{\bm{u}}^{m+1/2} \rangle_{1}
        - \langle\mathcal{A}_{h} \tilde{n}^{m+1/2} \nabla_{h} \mu_{n}^{m+1/2}, \hat{\bm{u}}^{m+1/2} \rangle_{1} 
        = 0.
    \end{aligned}
\end{equation} 
Regarding the term associated with the pressure gradient, $\langle\nabla_{h} \psi^{m}, \hat{\bm{u}}^{m+1/2} \rangle_{1}$, we begin with the identity,  $\nabla_{h} \cdot \hat{\bm{u}}^{m+1}= \frac{\tau}{2} \Delta_h ( \psi^{m+1} - \psi^{m})$, which comes from \eqref{scheme: p} and \eqref{scheme: nabla_u}. Then we get 
\begin{equation}\label{egy: p}
    \begin{aligned}
        \langle\nabla_{h} \psi^{m}, \hat{\bm{u}}^{m+1/2} \rangle_{1} 
        & = - \langle \psi^{m}, \nabla_h \cdot \hat{\bm{u}}^{m+1/2} \rangle_{C}
        = -\frac{1}{2} \langle \psi^{m}, \nabla_h \cdot \hat{\bm{u}}^{m+1} \rangle_{C} 
        \\
        & = -\frac{1}{4} \tau \langle \psi^{m}, \Delta_h (\psi^{m+1}-\psi^{m} ) \rangle_{C}
        = \frac{1}{4} \tau \langle\nabla_h \psi^{m}, \nabla_h (\psi^{m+1}-\psi^{m} ) \rangle_{1} \\
        & = \frac{\tau}{8} ( \|\nabla_h \psi^{m+1} \|_{2}^{2} 
        -  \|\nabla_h \psi^{m} \|_2^2  
        - \|\nabla_h (\psi^{m+1}-\psi^{m} ) \|_2^2 ) .
    \end{aligned}
\end{equation} 
In turn, a substitution of \eqref{egy: p} into \eqref{egy: u and p} gives
\begin{equation}\label{egy: u2}
    \begin{aligned}
        & \frac{\|\bm{u}^{m+1} \|_{2}^{2}- \|\bm{u}^{m} \|_{2}^{2}}{2 \tau} 
        + \frac{\tau}{8} ( \|\nabla_{h} \psi^{m+1} \|_{2}^{2}- \|\nabla_{h} \psi^{m}\ \|_{2}^{2} )
        + \| \nabla_{h} \hat{\bm{u}}^{m+1/2} \|_{2}^{2} \\
        & + \langle\mathcal{A}_{h} \tilde{p}^{m+1/2} \nabla_{h} \mu_{p}^{m+1/2}, \hat{\bm{u}}^{m+1/2} \rangle_{1}
        +  \langle\mathcal{A}_{h} \tilde{n}^{m+1/2} \nabla_{h} \mu_{n}^{m+1/2}, \hat{\bm{u}}^{m+1/2} \rangle_{1} 
        = 0. 
    \end{aligned}
\end{equation} 
Meanwhile, taking inner product with \eqref{scheme: n} and \eqref{scheme: p} by $\tau \mu_{n}^{m+1/2}$ and $\tau \mu_{p}^{m+1/2}$ respectively, we see that 
\begin{equation}\label{egy: n and p}
    \begin{aligned}
        & \langle n^{m+1}-n^{m}, \mu_{n}^{m+1/2} \rangle_C  
        +  \langle p^{m+1}-p^{m}, \mu_{p}^{m+1/2} \rangle_C  
        \\
        & - \tau \langle \mathcal{A}_{h} \tilde{n}^{m+1/2} \nabla_{h} \mu_{n}^{m+1/2} , \hat{\bm{u}}^{m+1/2} \rangle_C  
        - \tau \langle \mathcal{A}_{h} \tilde{p}^{m+1/2} \nabla_{h} \mu_{p}^{m+1/2} , \hat{\bm{u}}^{m+1/2} \rangle_C  
        \\
        & + \tau ( [\breve{n}^{m+1/2} \nabla_{h} \mu_{n}^{m+1/2}, \nabla_{h} \mu_{n}^{m+1/2} ]
        + [\breve{p}^{m+1/2} \nabla_{h} \mu_{p}^{m+1/2}, \nabla_{h} \mu_{p}^{m+1/2} ] ) = 0. 
    \end{aligned}
\end{equation}
On the other hand, the following equalities and inequalities are observed:
\begin{subequations}\label{egy: mu1}
    \begin{align}
        & \langle n^{m+1}-n^{m}, F_{n^{m}} (n^{m+1} ) \rangle_C  
        = \langle n^{m+1} \ln n^{m+1}, \mathrm{1} \rangle_C - \langle n^{m} \ln n^{m}, \mathrm{1} \rangle_C ,  
        \\
        & \langle p^{m+1}-p^{m}, F_{p^{m}} (p^{m+1}) \rangle_C  
        = \langle p^{m+1} \ln p^{m+1}, \mathrm{1} \rangle_C - \langle p^{m} \ln p^{m}, \mathrm{1} \rangle_C , 
        \\
        &  \langle n^{m+1}-n^{m}, (-\Delta_h )^{-1} (n^{m+1/2}-p^{m+1/2} ) \rangle_C  
        + \langle p^{m+1}-p^{m}, (-\Delta_h )^{-1} (p^{m+1/2}-n^{m+1/2} ) \rangle_C  
        \nonumber \\
        & = \frac{1}{2} ( \| n^{m+1}-p^{m+1} \|_{-1, h}^{2}- \|n^{m}-p^{m} \|_{-1, h}^{2} ),
    \end{align}
\end{subequations} 
\begin{equation}\label{egy: mu2} 
       \langle n^{m+1}-n^{m}, \ln n^{m+1}-\ln n^{m} \rangle_C  \ge 0,  \quad 
      \langle p^{m+1}-p^{m}, \ln p^{m+1}-\ln p^{m} \rangle_C \ge 0.
\end{equation}
Consequently, a substitution of \eqref{egy: mu1} and \eqref{egy: mu2} into \eqref{egy: n and p}, combined with \eqref{egy: u2}, leads to the following estimate: 
\begin{equation}
    \begin{aligned}
        & \langle n^{m+1} ( \ln n^{m+1} -1 ) , \mathrm{1} \rangle_C 
        - \langle n^{m} ( \ln n^{m} -1 ) , \mathrm{1} \rangle_C  
        + \langle p^{m+1} ( \ln p^{m+1} - 1) , \mathrm{1} \rangle_C 
\\
    & 
        - \langle p^{m} ( \ln p^{m} -1 ) , \mathrm{1} \rangle_C  
        + \frac{1}{2} ( \|n^{m+1}-p^{m+1} \|_{-1, h}^{2}- \|n^{m}-p^{m} \|_{-1, h}^{2} ) 
        + \frac{1}{2} ( \|\bm{u}^{m+1} \|_{2}^{2}- \|\bm{u}^{m} \|_{2}^{2} ) 
        \\
        & + \frac{\tau^{2}}{8} ( \|\nabla_{h} \psi^{m+1} \|_{2}^{2}- \|\nabla_{h} \psi^{m} \|_{2}^{2} )
        + \tau \|\nabla_{h} \hat{\bm{u}}^{m+1/2} \|_{2}^{2} 
        \\ 
        & + \tau ( [\breve{n}^{m+1/2} \nabla_{h} \mu_{n}^{m+1/2}, \nabla_{h} \mu_{n}^{m+1/2} ]
        + [\breve{p}^{m+1/2} \nabla_{h} \mu_{p}^{m+1/2}, \nabla_{h} \mu_{p}^{m+1/2} ] ) = 0 ,  
    \end{aligned}
\end{equation} 
in which the mass conservation identity~\eqref{mass conserv-1} has been used. The proof of Theorem \ref{thm: energy} is finished. 
\end{proof} 

\begin{remark} 
The modified discrete total energy functional, namely $\tilde{E}_{h} (n^{m+1}, p^{m+1}, \bm{u}^{m+1}, \psi^{m+1} )$  defined in~\eqref{mod energy-1}, is composed of the original version of the discrete total energy introduced in~\eqref{total energy-1} (evaluated at the time step $t^{m+1}$), combined with a numerical correction term $\frac{\tau^{2}}{8} \|\nabla_h \psi^{m+1} \|_2^2$. In fact, such a numerical correction term comes from the decoupled Stokes solver, and clearly it is of order $O (\tau^2)$.  Moreover, the discrete total energy functional~\eqref{total energy-1} turns out to be an $O (h^2)$ approximation to the continuous version of the total energy defined in~\eqref{energy dissipation-1}, since the discrete inner product has been proved to be an $O (h^2)$ approximation to its continuous version. A combination of these two arguments reveals that, the modified discrete total energy functional (given by~\eqref{mod energy-1}) is an $O (\tau^2 + h^2)$ approximation to the continuous version of the total energy defined in~\eqref{energy dissipation-1}.   

Meanwhile, since the finite difference algorithm computes the numerical solution only at the numerical grid points, the original version of the discrete total energy in~\eqref{total energy-1} would attract the most attentions in the numerical analysis. Theoretically speaking, the dissipation property of the modified discrete total energy functional, as proved in Theorem~\ref{thm: energy}, does not ensure the dissipation of the original version defined in~\eqref{total energy-1}. On the other hand, since the difference between the modified and the original versions is of $O (\tau^2)$, the dissipation of the original discrete total energy~\eqref{total energy-1} has been observed in all the numerical examples reported in this article, as will be demonstrated in the later section, while the theoretical analysis only ensures the dissipation of the modified discrete total energy. In addition, although the dissipation of the original discrete total energy is not theoretically available, we are able to derive its uniform bound in time: 
\begin{equation} 
\begin{aligned} 
   E_{h, total} (n^{m+1}, p^{m+1}, \bm{u}^{m+1} ) 
   \le & \tilde{E}_{h} (n^{m+1}, p^{m+1}, \bm{u}^{m+1}, \psi^{m+1} )  
   \le \tilde{E}_{h} (n^m, p^m, \bm{u}^m, \psi^m)  \le ... 
\\
  \le &  
   \tilde{E}_{h} (n^0, p^0, \bm{u}^0, \psi^0)  
  =   E_{h, total} (n^0, p^0, \bm{u}^0 )  + \frac{\tau^2}{8} \| \nabla_h \psi^0 \|_2^2 
  := \tilde{C}_0 , 
\end{aligned} 
  \label{mod energy-2} 
\end{equation} 
for any $m \ge 0$. Meanwhile, because of the mass conservation identity, $\overline{n} =\overline{p} = \beta_0$, we observe the following estimates:  
\begin{align} 
  &
  n \ln n \ge n - e - e^{-1} ,   \quad \mbox{so that} \, \, \, 
  \langle n \ln n , 1 \rangle_C  \ge  \| n \|_1 - ( e + e^{-1} ) | \Omega | ,    \label{mod energy-3-1} 
\\
  & 
  \langle p \ln p , 1 \rangle_C  \ge  \| p \|_1 - ( e + e^{-1} ) | \Omega | ,    \quad 
  \mbox{(similar argument)} , \label{mod energy-3-2} 
\\
  & 
  E_h ( n , p) = \langle n \ln n + p \ln p , 1 \rangle_C - | \Omega | ( \overline{n} + \overline{p} ) 
  = \langle n \ln n + p \ln p , 1 \rangle_C - 2 \beta_0 | \Omega |  \nonumber 
\\
  & \qquad \quad \, \, \, 
  \ge \| n \|_1 + \| p \|_1 - 2 ( e + e^{-1} + \beta_0 ) | \Omega |  . \label{mod energy-3-3} 
\end{align} 
In particular, inequality~\eqref{mod energy-3-1} comes from the fact that $x \ln x \ge - e^{-1} \ge x - e - e^{-1}$ for any $0 < x < e$, and $x \ln x \ge x$ for any $x \ge e$. In turn, a combination of \eqref{mod energy-2} and \eqref{mod energy-3-3} leads to   
\begin{equation} 
  E_h ( n^m , p^m ) + \frac12 \| \bm{u}^m \|_2^2 \le \tilde{C}_0 , \quad \mbox{so that} \, \, \, 
  \| n^m \|_1 + \| p^m \|_1 - 2 ( e + e^{-1} + \beta_0 ) | \Omega | + \frac12 \| \bm{u}^m \|_2^2 
  \le \tilde{C}_0 , \label{mod energy-4-1} 
\end{equation} 
for any $m \ge 1$, which in turn yields the following functional bounds for the numerical solution:  
\begin{equation} 
   \| n^m \|_1 , \, \, \| p^m \|_1 \le  \tilde{C}_0 + 2 ( e + e^{-1} + \beta_0 ) | \Omega | , \quad 
   \| \bm{u}^m \|_2 \le \Big( 2 \tilde{C}_0 + 4 ( e + e^{-1} + \beta_0 ) | \Omega |  \Big)^\frac12 , 
   \quad \forall m \ge 1 . 
   \label{mod energy-4-2} 
\end{equation} 

  In addition, it is observed that the uniform-in-time functional bounds~\eqref{mod energy-4-2} for the numerical solution, namely the discrete $\| \cdot \|_1$ norm for the ion concentration variables and the discrete $\| \cdot \|_2$ norm for the velocity variable, turn out to be very weak. These functional bounds are not sufficient in the optimal rate convergence analysis. To overcome this difficulty, a higher order consistency analysis via an asymptotic expansion is needed, and the inverse inequality has to be applied to obtain the $\| \cdot \|_\infty$ and the $W_h^\infty$ bounds of the numerical solution; see the details in the next section.  
\end{remark}

\section{Optimal rate convergence analysis}\label{sec: convergence}

In this section we present the convergence analysis. Denote $({\mathsf N}, {\mathsf P}, \Phi, \mathrm{\bf U}, \Psi)$ as the exact solution for the PNPNS system \eqref{eqn: main system}. With sufficiently regular initial data, the exact solution is assumed to be of the following regularity class:
\begin{equation}
    {\mathsf N} , {\mathsf P} , \mathrm{\bf U}, \Psi \in \mathcal{R}
    :=H^{6}\left(0, T ; C_{\mathrm{per}}(\Omega)\right) 
    \cap H^{5}\left(0, T ; C_{\mathrm{per}}^{2}(\Omega)\right) 
    \cap L^{\infty}\left(0, T ; C_{\mathrm{per}}^{6}(\Omega)\right). 
\end{equation} 
Moreover, a separation property is assumed for the exact ion concentration variables: 
\begin{equation}
    {\mathsf N} \geq \delta, \quad 
    {\mathsf P} \geq \delta, \quad  
    \mbox{for some $\delta >0$} , 
\end{equation}
at a point-wise level, for all $t \in[0, T]$. For the convenience of the $\| \cdot \|_{-1,h}$ error estimate, we introduce the Fourier projection of the exact solution into $\mathcal{B}^{K}$, the space of trigonometric polynomials of degree to $K(N=2 K+1)$:  $\mathsf{N}_{N}(\cdot, t):=\mathcal{P}_{N} \mathsf{N}(\cdot, t), \mathsf{P}_{N}(\cdot, t):=\mathcal{P}_{N} \mathsf{P}(\cdot, t)$. In fact, a standard projection estimate is available: 
\begin{equation} \label{eqn: projection estimate}
        \| {\mathsf N}_{N} - {\mathsf N} \|_{L^{\infty} (0, T ; H^{k} )}    
           \le Ch^{\ell-k} \| {\mathsf N} \|_{L^{\infty} (0, T ; H^{\ell} )} ,  \, \, 
        \| {\mathsf P}_{N} - {\mathsf P} \|_{L^{\infty} (0, T ; H^{k} )}    
           \le Ch^{\ell-k} \| {\mathsf P} \|_{L^{\infty} (0, T ; H^{\ell} )} , 
\end{equation} 
for any $\ell \in \mathbb{N}$ with $0 \leq k \leq \ell,(\mathrm{N}, \mathrm{P}) \in L^{\infty}
(0, T ; H_{\mathrm{per}}^{\ell}(\Omega) )$. In fact, the positivity of the ion concentration variables does not directly come from this Fourier projection estimate; on the other hand, a similar separation bound, ${\mathsf N}_{N} , \, {\mathsf P}_{N} \ge \frac{3 \delta}{4}$, could be derived by taking $h = \frac{L}{N}$ sufficiently small. To simplify the notation in the later analysis, we denote $\mathsf{N}_{N}^{m}=\mathsf{N}_{N}\left(\cdot, t_{m}\right), \mathsf{P}_{N}^{m}=\mathsf{P}_{N}\left(\cdot, t_{m}\right)$ (with $t_{m}=m \cdot \tau$), and notice the mass conservative property of the projection solution at the discrete level: 
\begin{equation} 
    \begin{aligned}
        & \overline{\mathsf{N}_{N}^{m}}
        =\frac{1}{|\Omega|} \int_{\Omega} \mathsf{N}_{N}\left(\cdot, t_{m}\right) \mathsf{d}\bm{x}
        =\frac{1}{|\Omega|} \int_{\Omega} \mathsf{N}_{N}\left(\cdot, t_{m-1}\right) \mathrm{d}\bm{x}
        =\overline{\mathsf{N}_{N}^{m-1}}, \quad \forall m \in \mathbb{N}, \\
        & \overline{\mathsf{P}_{N}^{m}}=\overline{\mathsf{P}_{N}^{m-1}}, \quad \forall m \in \mathbb{N}, 
        \quad \text { (similar argument), }
\end{aligned} 
 \label{mass conserv-2} 
\end{equation} 
which comes from the fact that $\left(\mathsf{N}_{N}, \mathsf{P}_{N}\right) \in \mathcal{B}^{K}$. On the other hand, the discrete mass conservation for the numerical solution \eqref{scheme: n} and \eqref{scheme: p} has been derived in~\eqref{mass conserv-1}. To facilitate the $\| \cdot \|_{-1, h}$ error analysis, the mass conservative projection is applied to the initial data:
\begin{equation}
    \begin{aligned}
        & (n^{0})_{i, j}
        =\mathcal{P}_{h} \mathsf{~N}_{N}(\cdot, t=0)
        :=\mathsf{N}_{N} (x_{i}, y_{j}, t=0 ) , \\
        & (p^{0} )_{i, j}
        =\mathcal{P}_{h} \mathsf{P}_{N}(\cdot, t=0)
        :=\mathsf{P}_{N} (x_{i}, y_{j}, t=0 ) .
\end{aligned}
\end{equation} 

In terms of the electric potential variable, we denote its Fourier projection as $\Phi_h = (-\Delta_h)^{-1} (\mathsf{P}_{N} - \mathsf{N}_{N})$, with a homogeneous Neumann boundary condition. Of course, a standard error estimate for the discrete Poisson equation indicates that $\| \Phi_h - \Phi \|_\infty \le C h^2$. Subsequently, the discrete error functions for the ion concentration and electric potential variables are introduced as
\begin{equation}\label{eqn: error grid definition}
    \mathring{n}^{m}:=\mathcal{P}_{h} \mathsf{N}_{N}^{m}-n^{m}, \quad 
    \mathring{p}^{m}:=\mathcal{P}_{h} \mathsf{P}_{N}^{m}-p^{m}, \quad 
    \mathring{\phi}^{m}:=\mathcal{P}_{h} \Phi_{N}^{m}-\phi^{m}, \quad 
    \forall m \in \mathbb{N} .
\end{equation} 
Because of the discrete mass conservation identities~\eqref{mass conserv-1}, \eqref{mass conserv-2}, it is clear that $\overline{\mathring{n}^{m}}=\overline{\mathring{p}^{m}}=0$, so that the discrete norm $\|\cdot\|_{-1, h}$ is well defined for $\mathring{n}^{m}$ and $\mathring{p}^{m}$, for any $m \in \mathbb{N}$. 

  In terms of the velocity and pressure variables, we just take the associated error functions as
\begin{equation}
    \mathring{\bm{u}}^{m}:=\mathcal{P}_{h} \mathrm{\bf U}^{m}-\bm{u}^{m}= ( \mathring{u}^m , \mathring{v}^m  )^T , \quad 
    \mathring{\psi}^{m}:=\mathcal{P}_{h} \Psi^{m}-\psi^{m}, 
    \quad \forall m \in \mathbb{N} . 
\end{equation} 

The following theorem is the main result of this section.

\begin{theorem}\label{thm: convergence}
    Given initial data $\mathsf{N}$, $\mathsf{P}$, $\Phi\left(\cdot, t=0\right)$, $\mathrm{\bf U}\left(\cdot, t=0\right) \in C_{\mathrm{per}}^{6}(\Omega)$, suppose the exact solution for PNPNS system \eqref{eqn: main system} is of regularity class $\mathcal{R}$. Then, provided $\tau$ and $h$ are sufficiently small, and under the linear refinement requirement $\lambda_1 h \leq \tau \leq \lambda_2 h$, we have
    \begin{equation}\label{eqn: convergence thm}
        \|\mathring{\bm{u}}^{m} \|_{2} +  \|\mathring{n}^{m} \|_{2} 
        + \|\mathring{p}^{m} \|_{2} 
        + \| \mathring{\phi}^{m} \|_{H_{h}^{2}} 
        +  \|\nabla_h \mathring{\psi}^{m} \|_2 
        \leq C ( \tau^2 +h^2 ), 
    \end{equation} 
for all positive integers $k$, such that $t_{k} = k \tau \leq T$, where $C>0$ is independent of $\tau$ and $h$.
\end{theorem} 

In the later analysis, $C$ represents a constant that may depend on $\Omega$ and $\delta$, but is independent on $h$ and $\tau$. 

\subsection{Higher order consistency analysis of the numerical system} 

Based on the detailed Taylor expansion analysis, s substitution of the projection solution $(\mathsf{N}_{N}, \mathsf{P}_{N})$ and the exact profiles $(\mathrm{\bf U}, \Psi)$ into the numerical scheme~\eqref{scheme: main} leads to a second order local truncation error, in both time and space. However, such a leading truncation error would not be sufficient to ensure an a-priori $W_{h}^{1, \infty}$ bound for the numerical solution, which is needed for the nonlinear error estimate. A higher order consistency analysis, accomplished by a perturbation expansion argument, is needed to remedy this effort. In more details, we have to construct a few supplementary functions, $\breve{\mathrm{\bf U}}$, $\breve{\Psi}$, $\breve{\mathsf N}$, $\breve{\mathsf P}$, 
with the following expansion: 
\begin{equation}\label{eqn: construction}
    \begin{aligned}
        & \breve{\mathrm{\bf U}}
        = \mathcal{P}_{H}\left(\mathrm{\bf U} 
        + \tau^{2} \mathrm{\bf U}_{\tau, 1} 
        + \tau^{3} \mathrm{\bf U}_{\tau, 2} 
        + h^{2} \mathrm{\bf U}_{h, 1}\right), \quad 
        \breve{\Psi} 
        = \mathcal{I}_{h}\left(\Psi 
        + \tau^{2} \Psi_{\tau, 1} 
        + \tau^{3} \Psi_{\tau, 2} 
        + h^{2} \Psi_{h, 1}\right), 
        \\
        & \breve{\mathsf{N}} 
        = \mathsf{N}_{N} 
        + \mathcal{P}_{N} (\tau^{2} \mathsf{N}_{\tau, 1} 
        + \tau^{3} \mathsf{N}_{\tau, 2} 
        + h^{2} \mathsf{N}_{h, 1} ), \quad 
        \breve{\mathsf{P}} 
        = \mathsf{P}_{N} 
        + \mathcal{P}_{N} (\tau^{2} \mathsf{P}_{\tau, 1} 
        + \tau^{3} \mathsf{P}_{\tau, 2} 
        + h^{2} \mathsf{P}_{h, 1} ) , 
    \end{aligned}
\end{equation}
in which $\mathcal{P}_{H}$ stands for a discrete Helmholtz interpolation (into the divergence-free space), 
and $\mathcal{I}_{h}$ is the standard point-wise interpolation. In turn, a substitution of these constructed function into the numerical scheme \eqref{scheme: main} gives a higher order $O(\tau^{4}+h^{4})$ consistency. The constructed functions, $\mathrm{\bf U}_{\tau, i}$, $\Psi_{\tau, i}$, $\mathsf{N}_{\tau, i}$, 
$\mathsf{P}_{\tau, i}$, 
$(i=1,2)$, $\mathrm{\bf U}_{h, 1}$, $\Psi_{h, 1}$, $\mathsf{N}_{h, 1}$, $\mathsf{P}_{h, 1}$, 
could be obtained by an asymptotic expansion technique, and they only depend on the exact solution $(\mathrm{\bf U}, \Psi, \mathsf{N}, \mathsf{P})$. 
In turn, the numerical error function between the constructed expansion profile and the numerical solution is analyzed, instead of a direct error analysis between the numerical and projection solutions. 

The following bilinear form is introduced to facilitate the nonlinear analysis:
\begin{equation}
    b(\bm{u}, \bm{v}) = \bm{u} \cdot \nabla \bm{v}, \quad 
    b_{h}(\bm{u}, \bm{v}) 
    = \frac{1}{2}\left(\bm{u} \cdot \nabla_{h} \bm{v} 
    + \nabla \cdot\left(\bm{u v}^{T}\right)\right). 
\end{equation} 
Moreover, the following intermediate velocity vector is needed in the leading order consistency analysis:
\begin{equation}
    \hat{\mathrm{\bf U}}^{m+1} = \mathrm{\bf U}^{m+1} + \frac{1}{2} \tau \nabla\left(\Psi^{m+1}-\Psi^{m}\right). 
\end{equation} 
Subsequently, a careful Taylor expansion (in time) for $(\mathsf{N}_N , \mathsf{P}_N , \mathrm{\bf U})$ and $\hat{\mathrm{\bf U}}$ implies that
\begin{subequations}\label{taylor:1}
    \begin{align}
        \frac{\hat{\mathrm{\bf U}}^{m+1}-\mathrm{\bf U}^{m}}{\tau}
        & +  b (\tilde{\mathrm{\bf U}}^{m+1/2}, \hat{\mathrm{\bf U}}^{m+1/2} )
        + \nabla \Psi^{m} 
        - \Delta \hat{\mathrm{\bf U}}^{m+1/2} 
        \nonumber \\
        = & - \tilde{\mathsf N}_{N}^{m+1/2} \nabla \mathrm{M}_{n}^{m+1/2}
        - \tilde{\mathsf P}_{N}^{m+1/2} \nabla \mathrm{M}_{p}^{m+1/2} 
         + \tau^2 \bm{\mathrm{G}}_{0}^{m+1/2}+O (\tau^{3}+h^{m_{0}} ), 
        \\
        \frac{\mathsf{N}_{N}^{m+1}-\mathsf{N}_{N}^{m}}{\tau} 
        & +  \nabla \cdot (\tilde{\mathsf{N}}_{N}^{m+1/2} \hat{\mathrm{\bf U}}^{m+1/2} ) 
        \nonumber \\
         = &  \nabla \cdot ( \tilde{\mathsf{N}}_{N}^{m+1/2} \nabla \mathrm{M}_{n}^{m+1/2} ) 
        + \tau^{2} H_{n,0}^{m+1/2} 
        + O (\tau^{3}+h^{m_{0}} ), 
        \\
        \mathrm{M}_{n}^{m+1/2} 
          = & 
        F_{\mathsf{N}_N^m} ( \mathsf{N}_{N}^{m+1} ) - 1 
        + \tau ( \ln\mathsf{N}_{N}^{m+1} - \ln \mathsf{N}_{N}^{m} ) 
         + ( - \Delta )^{-1} (\mathsf{N}_{N}^{m+1/2} - \mathsf{P}_{N}^{m+1/2} ),  
        \\
        \frac{\mathsf{P}_{N}^{m+1}-\mathsf{P}_{N}^{m}}{\tau} 
        & +  \nabla \cdot (\tilde{\mathsf{P}}_{N}^{m+1/2} \hat{\mathrm{\bf U}}^{m+1/2} ) 
        \nonumber \\
          = & \nabla \cdot ( \tilde{\mathsf{P}}_{N}^{m+1/2} \nabla \mathrm{M}_{p}^{m+1/2} ) 
        + \tau^{2} H_{p,0}^{m+1/2} + O (\tau^{3} + h^{m_{0}} ), 
        \\
        \mathrm{M}_{p}^{m+1/2} 
          = & F_{\mathsf{P}_N^m} ( \mathsf{P}_{N}^{m+1} ) - 1 
           + \tau ( \ln\mathsf{P}_{N}^{m+1} - \ln \mathsf{P}_{N}^{m} ) 
          + ( - \Delta )^{-1} (\mathsf{P}_{N}^{m+1/2} - \mathsf{N}_{N}^{m+1/2} ), 
        \\
        \frac{\mathrm{\bf U}^{m+1}-\hat{\mathrm{\bf U}}^{m+1}}{\tau} 
        & +  \frac{1}{2}\nabla (\Psi^{m+1}-\Psi^{m} ) = 0, \\
        \nabla \cdot \mathrm{\bf U}^{m+1}  = & 0 , 
    \end{align}
\end{subequations} 
in which $\|\bm{\mathrm{G}}_{0}^{m+1/2} \|, \| H_{0}^{m+1/2} \| \leq C$, and $C$ depends only on the regularity of the exact solutions.

The correction functions $\mathrm{\bf U}_{\tau, 1}$, $\Psi_{\tau, 1}$, $\mathsf{P}_{\tau, 1}$, $\mathsf{N}_{\tau, 1}$, $\mathrm{M}_{n, \tau, 1}$ and $\mathrm{M}_{p, \tau, 1}$, are constructed as the solution of the following PDE system
\begin{subequations}
    \begin{align}
        \partial_{t} \mathrm{\bf U}_{\tau, 1} 
        & +  \left( \mathrm{\bf U}_{\tau, 1} \cdot \nabla \right) \mathrm{\bf U} 
        + (\mathrm{\bf U} \cdot \nabla) \mathrm{\bf U}_{\tau, 1} 
        + \nabla \Psi_{\tau, 1}  - \Delta \mathrm{\bf U}_{\tau, 1} 
        \nonumber \\
         = & - \mathsf{N}_{\tau, 1} \nabla \mathrm{M}_{n} 
        - \mathsf{N}_{N} \nabla \mathrm{M}_{n, \tau, 1} 
        - \mathsf{P}_{\tau, 1} \nabla \mathrm{M}_{p} 
        - \mathsf{P}_{N} \nabla \mathrm{M}_{p, \tau, 1} 
        - \mathrm{G}_{0}, 
        \\
        \partial_{t} \mathsf{N}_{\tau, 1} 
        & +  \nabla \cdot ( \mathsf{N}_{\tau, 1} \mathrm{\bf U} 
        + \mathsf{N}_{N} \mathrm{\bf U}_{\tau, 1} ) 
        = \nabla \cdot\left(\mathsf{N}_{\tau, 1} \nabla \mathrm{M}_{n} 
        + \mathsf{N}_{N} \nabla \mathrm{M}_{n,\tau,1} \right)
        - H_{n,0}, 
        \\ 
        \mathrm{M}_{n} 
          = & \ln \mathsf{N}_{N} 
        + ( -\Delta )^{-1} (\mathsf{N}_{N} - \mathsf{P}_{N} ), 
        \\
        \mathrm{M}_{n,\tau,1} 
         = & \frac{1}{\mathsf{N}_{N}} \mathsf{N}_{\tau, 1} 
        + ( -\Delta )^{-1} (\mathsf{N}_{\tau, 1}-\mathsf{P}_{\tau, 1} ), 
        \\
        \partial_{t} \mathsf{P}_{\tau, 1} 
        & +  \nabla \cdot ( \mathsf{P}_{\tau, 1} \mathrm{\bf U} 
        + \mathsf{P}_{N} \mathrm{\bf U}_{\tau, 1} ) 
        = \nabla \cdot (\mathsf{P}_{\tau, 1} \nabla \mathrm{M}_{p} 
        + \mathsf{P}_{N} \nabla \mathrm{M}_{n,\tau,1} )
        - H_{p,0}, 
        \\ 
        \mathrm{M}_{p} 
          = & \ln \mathsf{P}_{N} 
        + ( -\Delta )^{-1} (\mathsf{P}_{N} - \mathsf{P}_{N} ), 
        \\
        \mathrm{M}_{p,\tau,1} 
          = & \frac{1}{\mathsf{P}_{N}} \mathsf{P}_{\tau, 1} 
        + ( -\Delta )^{-1} (\mathsf{P}_{\tau, 1}-\mathsf{N}_{\tau, 1} ), 
        \\ 
        \nabla \cdot \mathrm{\bf U}_{\tau, 1} = & 0. 
    \end{align}
\end{subequations} 
The homogeneous Neumann boundary condition for $\mathsf{N}_{\tau, 1}$ and $\mathsf{P}_{\tau, 1}$, combined with the no-penetration, free-slip boundary condition for $\mathrm{\bf U}_{\tau, 1}$, are imposed.  Existence and uniqueness of a solution of the above linear and parabolic PDE system is straightforward. Of course, a similar intermediate velocity vector could be introduced as
\begin{equation}
    \hat{\mathrm{\bf U}}_{\tau, 1}^{m+1} 
    = \mathrm{\bf U}_{\tau, 1}^{m+1} 
    + \frac{1}{2} \tau \nabla (\Psi_{\tau, 1}^{m+1} - \Psi_{\tau, 1}^{m} ). 
\end{equation} 
An application of a temporal discretization to the above linear PDE system for $\mathrm{\bf U}_{\tau, 1}, \Psi_{\tau, 1}$, $\mathsf{N}_{\tau, 1}, \mathsf{P}_{\tau, 1}$, $\mathrm{M}_{n,\tau,1}$, $\mathrm{M}_{p,\tau,1}$ and $\hat{\mathrm{\bf U}}_{\tau, 1}$ gives 
\begin{subequations}\label{taylor:2}
    \begin{align}
        \frac{\hat{\mathrm{\bf U}}_{\tau, 1}^{m+1}-\mathrm{\bf U}_{\tau, 1}^{m}}{\tau} 
        & + b (\tilde{\mathrm{\bf U}}_{\tau, 1}^{m+1/2}, \hat{\mathrm{\bf U}}^{m+1/2} )
        + b (\tilde{\mathrm{\bf U}}^{m+1/2}, \hat{\mathrm{\bf U}}_{\tau, 1}^{m+1/2} ) 
        + \nabla \Psi_{\tau, 1}^{m} 
        - \Delta \mathrm{\bf U}_{\tau, 1}^{m+1/2} 
        \nonumber \\
          = & - \tilde{\mathsf{N}}_{\tau, 1}^{m+1/2} \nabla \mathrm{M}_{n}^{m+1/2} 
        - \tilde{\mathsf{N}}_{N}^{m+1/2} \nabla \mathrm{M}_{n,\tau, 1}^{m+1/2} 
        \nonumber \\
         & - \tilde{\mathsf{P}}_{\tau, 1}^{m+1/2} \nabla \mathrm{M}_{p}^{m+1/2} 
        - \tilde{\mathsf{P}}_{N}^{m+1/2} \nabla \mathrm{M}_{p,\tau, 1}^{m+1/2} 
        - \bm{\mathrm{G}}_{0}^{m+1/2} + O (\tau^2 ) , 
        \\
        \frac{\mathsf{N}_{\tau, 1}^{m+1}-\mathsf{N}_{\tau, 1}^{m}}{\tau} 
        & +  \nabla \cdot (\tilde{\mathsf{N}}_{\tau, 1}^{m+1/2} \hat{\mathrm{\bf U}}^{m+1/2}
        + \tilde{\mathsf{N}}_{N}^{m+1/2} \hat{\mathrm{\bf U}}_{\tau, 1}^{m+1/2} ) 
        \nonumber \\ 
         = & \nabla \cdot ( \tilde{\mathsf{N}}_{\tau, 1}^{m+1/2} \nabla \mathrm{M}_{n}^{m+1/2} 
        + \tilde{\mathsf{N}}_{N}^{m+1/2} \nabla \mathrm{M}_{n, \tau, 1}^{m+1/2} ) 
        - H_{n,0}^{m+1/2} +  O (\tau^2 ), 
        \\
        \mathrm{M}_{n,\tau, 1}^{m+1/2} 
         = & \frac{1}{\mathsf{N}_{N}^{m+1/2}} \mathsf{N}_{\tau,1}^{m+1/2}
        + ( -\Delta )^{-1} ( \mathsf{N}^{m+1/2} - \mathsf{P}^{m+1/2} ), 
        \label{derivation: M_n}
        \\
        \frac{\mathsf{P}_{\tau, 1}^{m+1}-\mathsf{P}_{\tau, 1}^{m}}{\tau} 
        & +   \nabla \cdot (\tilde{\mathsf{P}}_{\tau, 1}^{m+1/2} \hat{\mathrm{\bf U}}^{m+1/2}
        + \tilde{\mathsf{P}}_{N}^{m+1/2} \hat{\mathrm{\bf U}}_{\tau, 1}^{m+1/2} ) 
        \nonumber \\ 
         = & \nabla \cdot ( \tilde{\mathsf{P}}_{\tau, 1}^{m+1/2} \nabla \mathrm{M}_{p}^{m+1/2} 
        + \tilde{\mathsf{P}}_{N}^{m+1/2} \nabla \mathrm{M}_{p, \tau, 1}^{m+1/2} ) 
        - H_{p,0}^{m+1/2} + O (\tau^2 ), 
        \\
        \mathrm{M}_{p,\tau, 1}^{m+1/2} 
         = & \frac{1}{\mathsf{P}_{N}^{m+1/2}} \mathsf{P}_{\tau,1}^{m+1/2}
        + ( -\Delta )^{-1} ( \mathsf{P}^{m+1/2} - \mathsf{N}^{m+1/2} ), 
        \label{derivation: M_p}
        \\
        \frac{\mathrm{\bf U}_{\tau, 1}^{m+1}-\hat{\mathrm{\bf U}}_{\tau, 1}^{m+1}}{\tau} 
        & +   \frac{1}{2} \nabla (\Psi_{\tau, 1}^{m+1}-\Psi_{\tau, 1}^{m} ) = 0, 
        \\
        \nabla \cdot \mathrm{\bf U}_{\tau, 1}^{m+1} = & 0. 
    \end{align}
\end{subequations} 
A combination of \eqref{taylor:1} and \eqref{taylor:2} results in the following third order truncation error for 
$\mathrm{\bf U}_{1}:=\mathrm{\bf U}+\tau^{2} \mathrm{\bf U}_{\tau, 1}$, 
$\mathsf{N}_{1}:=\mathsf{N}_{N}+\tau^{2} \mathcal{P}_{N} \mathsf{N}_{\tau, 1}$, 
$\mathsf{P}_{1}:=\mathsf{P}_{N}+\tau^{2} \mathcal{P}_{N} \mathsf{P}_{\tau, 1}$, 
$\Psi_{1}:=\Psi+\tau^{2} \Psi_{\tau, 1}$: 
\begin{subequations}\label{expan:1}
    \begin{align}
        \frac{\hat{\mathrm{\bf U}}_{1}^{m+1} - \mathrm{\bf U}_{1}^{m}}{\tau} 
        & +  b (\tilde{\mathrm{\bf U}}_{1}^{m+1/2}, \hat{\mathrm{\bf U}}_{1}^{m+1/2} ) 
        + \nabla \Psi_{1}^{m} 
        - \Delta \hat{\mathrm{\bf U}}_{1}^{m+1/2} 
        \nonumber \\
         = & - \tilde{\mathsf{N}}_{1}^{m+1/2} \nabla M_{n,1}^{m+1/2} 
        - \tilde{\mathsf{P}}_{1}^{m+1/2} \nabla M_{p,1}^{m+1/2} 
        + \tau^3 \mathrm{\bm{G}}_{1}^{m+1/2} 
        + O (\tau^4+h^{m_{0}} ), 
        \\
        \frac{\mathsf{N}_{1}^{m+1} - \mathsf{N}_{1}^{m}}{\tau} 
        & +  \nabla \cdot (\tilde{\mathsf{N}}_{1}^{m+1/2} \hat{\mathrm{\bf U}}_{1}^{m+1/2} ) 
        \nonumber \\ 
        = & \nabla \cdot ( \tilde{\mathsf N}_{1}^{m+1/2} \nabla M_{n,1}^{m+1/2} ) 
        + \tau^3 H_{n,1}^{m+1/2} + O (\tau^4 + h^{m_{0}} ), 
        \\
         \mathrm{M}_{n, 1}^{m+1/2} 
          = & F_{\mathsf{N}_1^m} ( \mathsf{N}_1^{m+1} ) - 1 
        + \tau ( \ln\mathsf{N}_1^{m+1} - \ln \mathsf{N}_1^{m} ) 
        + ( - \Delta )^{-1} (\mathsf{N}_1^{m+1/2} - \mathsf{P}_1^{m+1/2} ),  
        \\ 
        \frac{\mathsf{P}_{1}^{m+1} - \mathsf{P}_{1}^{m}}{\tau} 
        & +  \nabla \cdot (\tilde{\mathsf{P}}_{1}^{m+1/2} \hat{\mathrm{\bf U}}_{1}^{m+1/2} ) 
        \nonumber \\ 
        = & \nabla \cdot ( \tilde{\mathsf P}_{1}^{m+1/2} \nabla M_{p,1}^{m+1/2} ) 
        + \tau^{3} H_{p,1}^{m+1/2} + O (\tau^4 + h^{m_{0}} ), 
        \\
        \mathrm{M}_{p, 1}^{m+1/2} 
          = & F_{\mathsf{P}_1^m} ( \mathsf{P}_1^{m+1} ) - 1 
        + \tau ( \ln\mathsf{P}_1^{m+1} - \ln \mathsf{P}_1^{m} ) 
        + ( - \Delta )^{-1} (\mathsf{P}_1^{m+1/2} - \mathsf{N}_1^{m+1/2} ), 
        \\
        \frac{\mathrm{\bf U}_{1}^{m+1} - \hat{\mathrm{\bf U}}_{1}^{m+1}}{\tau} 
        & +  \frac{1}{2} \nabla (\Psi_{1}^{m+1} - \Psi_{1}^{m} )=0, 
        \\
        \nabla \cdot \mathrm{\bf U}_{1}^{m+1}  = & 0, 
    \end{align}
\end{subequations} 
where $\|\mathrm{G}_{1} \|$, $\|H_{1} \| \leq C$, and $C$ depends only on the regularity of the exact solutions. In fact, the following linearized expansions have been applied in the above derivation: 
\begin{equation}
    \frac{1}{\mathsf{N}_{N}^{m+1/2}} \mathsf{N}_{\tau, 1}^{m+1/2} 
    = \frac{1}{2} \Big(\frac{1}{\mathsf{N}_{N}^{m}} \mathsf{N}_{\tau, 1}^{m} 
    + \frac{1}{\mathsf{N}_{N}^{m+1}} \mathsf{N}_{\tau, 1}^{m+1} \Big) + O (\tau^{2} ). 
\end{equation} 

Similarly, the next order temporal correction functions, namely $\mathrm{\bf U}_{\tau, 2}$, $\Psi_{\tau, 2}$, $\mathsf{P}_{\tau, 2}$ and $\mathsf{N}_{\tau, 2}$, are given by following linear equations: 
\begin{subequations}\label{taylor:3}
    \begin{align}
        \partial_{t} \mathrm{\bf U}_{\tau, 2} 
        & +  ( \mathrm{\bf U}_{\tau, 2} \cdot \nabla ) \mathrm{\bf U}_{1} 
        + (\mathrm{\bf U}_{1}  \cdot \nabla ) \mathrm{\bf U}_{\tau, 2} 
        + \nabla \Psi_{\tau, 2} - \Delta \mathrm{\bf U}_{\tau, 2} 
        \nonumber \\
        = & - \mathsf{N}_{\tau, 2} \nabla \mathrm{M}_{n , 1} 
        - \mathsf{N}_{1} \nabla \mathrm{M}_{n, \tau, 2} 
        - \mathsf{P}_{\tau, 2} \nabla \mathrm{M}_{p , 1} 
        - \mathsf{P}_{1} \nabla \mathrm{M}_{p, \tau, 2} 
        - \mathrm{G}_{1}, 
        \\
        \partial_{t} \mathsf{N}_{\tau, 2} 
        & +  \nabla \cdot ( \mathsf{N}_{\tau, 2} \mathrm{\bf U}_{1} 
        + \mathsf{N}_{1} \mathrm{\bf U}_{\tau, 2} ) 
        = \nabla \cdot (\mathsf{N}_{\tau, 2} \nabla \mathrm{M}_{n , 1} 
        + \mathsf{N}_{1} \nabla \mathrm{M}_{n,\tau, 2} )
        - H_{n,1}, 
        \\ 
        \mathrm{M}_{n , 1} 
        = & \ln \mathsf{N}_{1} 
        + ( -\Delta )^{-1} (\mathsf{N}_{1} - \mathsf{P}_{1} ), 
        \\
        \mathrm{M}_{n, \tau, 2} 
        = & \frac{1}{\mathsf{N}_{N}} \mathsf{N}_{\tau, 2} 
        + ( -\Delta )^{-1} (\mathsf{N}_{\tau, 2}-\mathsf{P}_{\tau, 2} ), 
        \\
        \partial_{t} \mathsf{P}_{\tau, 2} 
        & +  \nabla \cdot ( \mathsf{P}_{\tau, 2} \mathrm{\bf U}_{1} 
        + \mathsf{P}_{1} \mathrm{\bf U}_{\tau, 2} ) 
        = \nabla \cdot (\mathsf{P}_{\tau, 2} \nabla \mathrm{M}_{p , 1} 
        + \mathsf{P}_{1} \nabla \mathrm{M}_{p, \tau, 2} ) - H_{p,1}, 
        \\ 
        \mathrm{M}_{p , 1} 
        = & \ln \mathsf{P}_{1} 
        + ( -\Delta )^{-1} (\mathsf{P}_{1} - \mathsf{N}_{1} ), 
        \\
        \mathrm{M}_{p, \tau, 2} 
        = & \frac{1}{\mathsf{P}_{1}} \mathsf{P}_{\tau, 2} 
        + ( -\Delta )^{-1} (\mathsf{P}_{\tau, 2}-\mathsf{N}_{\tau, 2} ), 
        \\ 
        \nabla \cdot \mathrm{\bf U}_{\tau, 2} = & 0.
    \end{align}
\end{subequations} 
Again, the homogeneous Neumann boundary condition is imposed for $\mathsf{N}_{\tau, 2}$ and $\mathsf{P}_{\tau, 2}$, combined with the no-penetration, free-slip boundary condition for $\mathrm{\bf U}_{\tau, 2}$. 
Meanwhile, a similar intermediate velocity vector is introduced as
\begin{equation}
    \hat{\mathrm{\bf U}}_{\tau, 2}^{m+1} 
    = \mathrm{\bf U}_{\tau, 2}^{m+1} 
    + \frac{1}{2} \tau \nabla (\Psi_{\tau, 2}^{m+1} - \Psi_{\tau, 2}^{m} ). 
\end{equation} 
In turn, an application of a temporal discretization to the above linear PDE system for $\mathrm{\bf U}_{\tau, 1}$, $\Psi_{\tau, 1}$, $\mathsf{N}_{\tau, 1}$ and $\mathsf{P}_{\tau, 1}$ reveals that
\begin{subequations}\label{taylor:2-n}
    \begin{align}
        \frac{\hat{\mathrm{\bf U}}_{\tau, 2}^{m+1}-\mathrm{\bf U}_{\tau, 2}^{m}}{\tau} 
        & +  b (\tilde{\mathrm{\bf U}}_{\tau, 2}^{m+1/2}, \hat{\mathrm{\bf U}}_{1}^{m+1/2} )
        + b (\tilde{\mathrm{\bf U}}_{1}^{m+1/2}, \hat{\mathrm{\bf U}}_{\tau, 2}^{m+1/2} ) 
        + \nabla \Psi_{\tau, 2}^{m}- \Delta \hat{\mathrm{\bf U}}_{\tau, 2}^{m+1/2} 
        \nonumber \\
        = & - \tilde{\mathsf{N}}_{\tau, 2}^{m+1/2} \nabla \mathrm{M}_{n , 1}^{m+1/2} 
        - \tilde{\mathsf{N}}_{1}^{m+1/2} \nabla \mathrm{M}_{n,\tau, 2}^{m+1/2} 
        \nonumber \\
        & - \tilde{\mathsf{P}}_{\tau, 2}^{m+1/2} \nabla \mathrm{M}_{p , 1}^{m+1/2} 
        - \tilde{\mathsf{P}}_{1}^{m+1/2} \nabla \mathrm{M}_{p,\tau, 2}^{m+1/2} 
        - \bm{\mathrm{G}}_{1}^{m+1/2} + O (\tau^2 ), 
        \\
        \frac{\mathsf{N}_{\tau, 2}^{m+1}-\mathsf{N}_{\tau, 2}^{m}}{\tau} 
        & +  \nabla \cdot (\tilde{\mathsf{N}}_{\tau, 2}^{m+1/2} \hat{\mathrm{\bf U}}_{1}^{m+1/2}
        + \tilde{\mathsf{N}}_{1}^{m+1/2} \hat{\mathrm{\bf U}}_{\tau, 2}^{m+1/2} ) 
        \nonumber \\ 
         = & \nabla \cdot ( \tilde{\mathsf{N}}_{\tau, 2}^{m+1/2} \nabla \mathrm{M}_{n , 1}^{m+1/2} 
        + \tilde{\mathsf{N}}_{1}^{m+1/2} \nabla \mathrm{M}_{n, \tau, 2}^{m+1/2} ) 
        - H_{n,1}^{m+1/2} + O (\tau^2 ), 
        \\
        \mathrm{M}_{n , 1}^{m+1/2} 
        = & F_{\mathsf{N}_{1}^{m}} ( \mathsf{N}_{1}^{m+1} ) 
        + ( -\Delta )^{-1} ( \mathsf{N}_{1}^{m+1/2} - \mathsf{P}_{1}^{m+1/2} ) 
        + \tau (\ln \mathsf{N}_{1}^{m+1} - \ln \mathsf{N}_{1}^{m} ), 
        \\
        \mathrm{M}_{n , \tau , 2}^{m+1/2} 
        = & \frac{1}{\mathsf{N}_{1}^{m+1/2}} \mathsf{N}_{\tau , 2}^{m+1/2}
        + ( -\Delta )^{-1} ( \mathsf{N}_{\tau , 2}^{m+1/2} - \mathsf{P}_{\tau , 2}^{m+1/2} ), 
        \label{derivation: M_n2}
        \\
        \frac{\mathsf{P}_{\tau, 2}^{m+1}-\mathsf{P}_{\tau, 2}^{m}}{\tau} 
        + & \nabla \cdot (\tilde{\mathsf{P}}_{\tau, 2}^{m+1/2} \hat{\mathrm{\bf U}}_{1}^{m+1/2}
        + \tilde{\mathsf{P}}_{1}^{m+1/2} \hat{\mathrm{\bf U}}_{\tau, 2}^{m+1/2} ) 
        \nonumber \\ 
        = & \nabla \cdot ( \tilde{\mathsf{P}}_{\tau, 2}^{m+1/2} \nabla \mathrm{M}_{p , 1}^{m+1/2} 
        + \tilde{\mathsf{P}}_{1}^{m+1/2} \nabla \mathrm{M}_{p, \tau, 2}^{m+1/2} ) 
        - H_{p,1}^{m+1/2} + O (\tau^2 ), 
        \\
        \mathrm{M}_{p , 1}^{m+1/2} 
        = & F_{\mathsf{P}_{1}^{m}} ( \mathsf{P}_{1}^{m+1} ) 
        + ( -\Delta )^{-1} ( \mathsf{P}_{1}^{m+1/2} - \mathsf{N}_{1}^{m+1/2} ) 
        + \tau (\ln \mathsf{P}_{1}^{m+1} - \ln \mathsf{P}_{1}^{m} ), 
        \\
        \mathrm{M}_{p,\tau, 2}^{m+1/2} 
        = & \frac{1}{\mathsf{P}_{1}^{m+1/2}} \mathsf{P}_{\tau , 2}^{m+1/2}
        + ( -\Delta )^{-1} ( \mathsf{P}_{\tau , 2}^{m+1/2} - \mathsf{N}_{\tau , 2}^{m+1/2} ), 
        \label{derivation: M_p2}
        \\
        \frac{\mathrm{\bf U}_{\tau, 2}^{m+1}-\hat{\mathrm{\bf U}}_{\tau, 2}^{m+1}}{\tau} 
        & +  \frac{1}{2} \nabla (\Psi_{\tau, 2}^{m+1} - \Psi_{\tau, 2}^{m} ) = 0, 
        \\
        \nabla \cdot \mathrm{\bf U}_{\tau, 2}^{m+1} = & 0. 
    \end{align}
\end{subequations} 
A combination of \eqref{expan:1} and \eqref{taylor:2-n} yields the following fourth order truncation error for 
$\mathrm{\bf U}_{2}:=\mathrm{\bf U}_{1}+\tau^3 \mathrm{\bf U}_{\tau, 2}$, 
$\mathsf{N}_{2}:=\mathsf{N}_{1}+\tau^3 \mathcal{P}_{N} \mathsf{N}_{\tau, 2}$, 
$\mathsf{P}_{2}:=\mathsf{P}_{1}+\tau^3 \mathcal{P}_{N} \mathsf{P}_{\tau, 2}$ and 
$\Psi_{2}:=\Psi_{1}+\tau^3 \Psi_{\tau, 2}$: 
\begin{subequations}\label{expan:2}
    \begin{align}
        \frac{\hat{\mathrm{\bf U}}_{2}^{m+1} - \mathrm{\bf U}_{2}^{m}}{\tau} 
        & +  b (\tilde{\mathrm{\bf U}}_{2}^{m+1/2}, \hat{\mathrm{\bf U}}_{2}^{m+1/2} ) 
        + \nabla \Psi_{2}^{m} 
        - \Delta \hat{\mathrm{\bf U}}_{2}^{m+1/2} 
        \nonumber \\
        = & - \tilde{\mathsf{N}}_{2}^{m+1/2} \nabla M_{n,2}^{m+1/2} 
        - \tilde{\mathsf{P}}_{2}^{m+1/2} \nabla M_{p,2}^{m+1/2} 
        + O (\tau^4 +h^{m_{0}} ), 
        \\
        \frac{\mathsf{N}_{2}^{m+1} - \mathsf{N}_{2}^{m}}{\tau} 
        & + \nabla \cdot (\tilde{\mathsf{N}}_{2}^{m+1/2} \hat{\mathrm{\bf U}}_{2}^{m+1/2} ) 
        = \nabla \cdot ( \tilde{\mathsf N}_{2}^{m+1/2} \nabla M_{n,2}^{m+1/2} ) 
        + O (\tau^4 + h^{m_{0}} ), 
        \\
        M_{n , 2}^{m+1/2} 
        = & F_{\mathsf{N}_{2}^{m}} (\mathsf{N}_{2}^{m+1} ) 
        + \tau (\ln \mathsf{N}_{2}^{m+1} - \ln \mathsf{N}_{2}^{m} ) 
        + (-\Delta )^{-1} (\mathsf{N}_{2}^{m+1/2}-\mathsf{P}_{2}^{m+1/2} ) , 
        \\ 
        \frac{\mathsf{P}_{2}^{m+1} - \mathsf{P}_{2}^{m}}{\tau} 
        & + \nabla \cdot (\tilde{\mathsf{P}}_{2}^{m+1/2} \hat{\mathrm{\bf U}}_{2}^{m+1/2} ) 
        =  \nabla \cdot ( \tilde{\mathsf P}_{2}^{m+1/2} \nabla M_{p,2}^{m+1/2} ) 
        + O (\tau^{4} + h^{m_{0}} ), 
        \\
        M_{p , 2}^{m+1/2} 
        = & F_{\mathsf{P}_{2}^{m}} (\mathsf{P}_{2}^{m+1} ) 
        + \tau (\ln \mathsf{P}_{2}^{m+1} - \ln \mathsf{P}_{2}^{m} ) 
        + (-\Delta )^{-1} (\mathsf{P}_{2}^{m+1/2}-\mathsf{N}_{2}^{m+1/2} ) , 
        \\ 
        \frac{\mathrm{\bf U}_{2}^{m+1} - \hat{\mathrm{\bf U}}_{2}^{m+1}}{\tau} 
        & +  \frac{1}{2} \nabla (\Psi_{2}^{m+1} - \Psi_{2}^m )=0, 
        \\
        \nabla \cdot \mathrm{\bf U}_{2}^{m+1} = & 0, 
    \end{align} 
\end{subequations} 
with $\|\bm{\mathrm{G}}_{2}\|$, $\|H_{n,2} \|$, $\|H_{p,2} \| \le C$, and $C$ dependent only on the regularity of the exact solution. 

Next, we have to construct the spatial correction function to improve the spatial accuracy order. In terms of the spatial discretization, the key challenge is associated with the fact that the velocity vector $\mathrm{\bf U}_{2}$ is not divergence-free at a discrete level, so that its discrete inner product with the pressure gradient may not vanish. To overcome this difficulty, we make use of a spatial interpolation operator $\mathcal{P}_{H}$. For any $\bm{u} \in H^{1}\left(\Omega\right), \nabla \cdot \bm{u}=0$, there is an exact stream function $\psi$ so that $\bm{u}=\nabla^{\perp} \psi$. Subsequently, we define the following discrete velocity vector: 
\begin{equation}
    \mathcal{P}_{H}(\bm{u})=\nabla_{h}^{\perp} \psi=\left(-D_{y} \psi, D_{x} \psi\right)^{T}. 
\end{equation}
As a result, this definition ensures $\nabla_{h} \cdot \mathcal{P}_{H}\left(\bm{u}\right)=0$ at a point-wise level. Moreover, an $O (h^2)$ truncation error is available between the continuous velocity vector $\bm{u}$ 
and its Helmholtz interpolation, $\mathcal{P}_{H}(\bm{u})$.

Subsequently, we denote $\mathrm{\bf U}_{2, PH}=\mathcal{P}_{H} (\mathrm{\bf U}_{2} )$. An application of the finite difference spatial approximation over the MAC mesh grid indicates the following truncation error estimate: 
\begin{subequations}\label{spatial taylor:1}
    \begin{align}
        \frac{\hat{\mathrm{\bf U}}_{2, PH}^{m+1}- \mathrm{\bf U}_{2, PH}^{m}}{\tau} 
        & +  b_{h} (\tilde{\mathrm{\bf U}}_{2,PH}^{m+1/2}, \hat{\mathrm{\bf U}}_{2,PH}^{m+1/2} ) 
        + \nabla_{h} \Psi_{2}^{m}- \Delta_h \hat{\mathrm{\bf U}}_{2, PH}^{m+1/2} 
        \nonumber \\ 
        = & - \mathcal{A}_{h} \tilde{\mathsf{N}}_{2}^{m+1/2} \nabla_{h} \mathrm{M}_{n, 2, h}^{m+1/2} 
        - \mathcal{A}_{h} \tilde{\mathsf{P}}_{2}^{m+1/2} \nabla_{h} \mathrm{M}_{p, 2, h}^{m+1/2} 
        \nonumber \\
        & + h^2 \bm{\mathrm{G}}_{h}^{m+1/2}+O (\tau^4 +h^4 ), 
        \\
        \frac{\mathsf{N}_{2}^{m+1}-\mathsf{N}_{2}^{m}}{\tau} 
        & +  \nabla_h \cdot (\mathcal{A}_{h} \tilde{\mathsf{N}}_{2}^{m+1/2} \hat{\mathrm{\bf U}}_{2 , PH}^{m+1/2} ) 
        \nonumber \\
        = & \nabla_{h} \cdot ( \mathcal{A}_{h} \tilde{\mathsf{N}}_{2}^{m+1/2} \nabla_{h} \mathrm{M}_{n , 2 , h}^{m+1/2} ) + h^{2} H_{n , h}^{m+1/2} + O (\tau^{4} + h^{4} ), 
        \\
        M_{n , 2 , h}^{m+1/2} 
         = & F_{\mathsf{N}_{2}^{m}} (\mathsf{N}_{2}^{m+1} ) 
        + \tau (\ln \mathsf{N}_{2}^{m+1} - \ln \mathsf{N}_{2}^{m} ) 
        + (-\Delta_h )^{-1} (\mathsf{N}_{2}^{m+1/2}-\mathsf{P}_{2}^{m+1/2} ), 
        \\
        \frac{\mathsf{P}_{2}^{m+1}-\mathsf{P}_{2}^{m}}{\tau} 
        & +  \nabla_{h} \cdot (\mathcal{A}_{h} \tilde{\mathsf{P}}_{2}^{m+1/2} \hat{\mathrm{\bf U}}_{2 , PH}^{m+1/2} ) 
        \nonumber \\
         = & \nabla_{h} \cdot ( \mathcal{A}_{h} \tilde{\mathsf{P}}_{2}^{m+1/2} \nabla_{h} \mathrm{M}_{p , 2 , h}^{m+1/2} ) + h^{2} H_{p , h}^{m+1/2} 
        + O (\tau^{4} + h^{4} ), 
        \\
        \mathrm{M}_{p , 2 , h}^{m+1/2} 
        = & F_{\mathsf{P}_{2}^{m}} (\mathsf{P}_{2}^{m+1} ) 
        + \tau (\ln \mathsf{P}_{2}^{m+1} - \ln \mathsf{P}_{2}^{m} ) 
        + (-\Delta_h )^{-1} (\mathsf{P}_{2}^{m+1/2}-\mathsf{N}_{2}^{m+1/2} ), 
        \\
        \frac{\mathrm{\bf U}_{2 , PH}^{m+1} - \hat{\mathrm{\bf U}}_{2, PH}^{m+1}}{\tau} 
        & +  \frac{1}{2} \nabla_{h} (\Psi_{2}^{m+1}-\Psi_{2}^{m} ) = 0, 
        \\
        \nabla_{h} \cdot \mathrm{\bf U}_{2, PH}^{m+1} = & 0 , 
    \end{align}  
\end{subequations} 
with $\|\bm{\mathrm{G}}_{h} \|_2$, $\|H_{n , h} \|_2$, $\|H_{p , h} \|_2 \le C$, and $C$ dependent only on the regularity of the exact solution. Subsequently, the spatial correction functions, $\mathrm{\bf U}_{h, 1}$, $\mathsf{N}_{h , 1}$, $\mathsf{P}_{h , 1}$ and $\Psi_{h , 1}$, are determined by the following linear PDE system
\begin{subequations}
    \begin{align}
        \partial_{t} \mathrm{\bf U}_{h, 1} 
        & +  (\mathrm{\bf U}_{h, 1} \cdot \nabla ) \mathrm{\bf U}_{2} 
        + (\mathrm{\bf U}_{2} \cdot \nabla ) \mathrm{\bf U}_{h, 1} 
        + \nabla \Psi_{h, 1} - \Delta \mathrm{\bf U}_{h , 1} 
        \nonumber \\ 
         = & - \mathsf{N}_{h, 1} \nabla \mathrm{M}_{n , 2} 
        - \mathsf{N}_{2} \nabla \mathrm{M}_{n , h, 1} 
        - \mathsf{P}_{h, 1} \nabla \mathrm{M}_{p , 2} 
        - \mathsf{P}_{2} \nabla \mathrm{M}_{p , h, 1} 
        - \bm{\mathrm{G}}_{h},  \\
        \partial_{t} \mathsf{N}_{h , 1} 
        & +  \nabla \cdot (\mathsf{N}_{h, 1} \mathrm{\bf U}_{2} + \mathsf{N}_{2} \mathrm{\bf U}_{h, 1} ) 
        = \nabla \cdot ( \mathsf{N}_{h , 1} \nabla \mathrm{M}_{n , 2} + \mathsf{N}_{2} \nabla \mathrm{M}_{n , h , 1} ) - H_{n , h}, 
        \\ 
        \mathrm{M}_{n , 2} 
         = & \ln \mathsf{N}_{2} + ( - \Delta )^{-1} ( \mathsf{N}_{2} - \mathsf{P}_{2} ), 
        \\ 
        \mathrm{M}_{n , h , 1} 
         = & \frac{1}{\mathsf{N}_{2}} \mathsf{N}_{n , h , 1} 
        + ( - \Delta )^{-1} ( \mathsf{N}_{h, 1} - \mathsf{P}_{h , 1} ), 
        \\
        \partial_{t} \mathsf{P}_{h , 1} 
        & +  \nabla \cdot (\mathsf{P}_{h, 1} \mathrm{\bf U}_{2} + \mathsf{P}_{2} \mathrm{\bf U}_{h, 1} ) 
        = \nabla \cdot ( \mathsf{P}_{h , 1} \nabla \mathrm{M}_{p , 2} + \mathsf{P}_{2} \nabla \mathrm{M}_{p , h , 1} ) - H_{p , h}, 
        \\ 
        \mathrm{M}_{p , 2} 
        = & \ln \mathsf{P}_{2} + ( - \Delta )^{-1} ( \mathsf{P}_{2} - \mathsf{N}_{2} ), 
        \\ 
        \mathrm{M}_{p , h , 1} 
        = & \frac{1}{\mathsf{P}_{2}} \mathsf{N}_{p , h , 1} 
        + ( - \Delta )^{-1} ( \mathsf{P}_{h, 1} - \mathsf{N}_{h , 1} ), 
        \\
        \nabla \cdot \mathrm{\bf U}_{h, 1}  = & 0. 
    \end{align}
\end{subequations} 
Similarly, the homogeneous Neumann boundary condition is imposed for $\mathsf{N}_{h, 1}$ and $\mathsf{P}_{h, 1}$, combined with the no-penetration, free-slip boundary condition for $\mathrm{\bf U}_{h, 1}$. Afterwards, we denote $\mathrm{\bf U}_{h, 1, PH}=\mathcal{P}_{H} (\mathrm{\bf U}_{h, 1})$, and $\hat{\mathrm{\bf U}}_{h, 1, PH}^{m+1}=\mathrm{\bf U}_{h, 1, PH}^{m+1} + \frac{1}{2} \tau \nabla_{h} (\Psi_{h, 1}^{m+1}-\Psi_{h, 1}^{m} )$. In turn, an application of both the temporal and spatial approximations to the above PDE system indicates that 
\begin{subequations}\label{spatial taylor:2}
    \begin{align}
      & 
        \frac{\hat{\mathrm{\bf U}}_{h, 1, PH}^{m+1}-\mathrm{\bf U}_{h, 1, PH}^{m}}{\tau} 
        +  b_{h} (\tilde{\mathrm{\bf U}}_{h, 1, PH}^{m+1/2}, \mathrm{\bf U}_{2 , PH}^{m+1/2} ) 
        + b_{h} (\tilde{\mathrm{\bf U}}_{2 , PH}^{m+1/2}, \mathrm{\bf U}_{h, 1, PH}^{m+1/2} )  
         + \nabla_{h} \Psi_{h, 1}^{m} 
        -  \Delta_h \mathrm{\bf U}_{h, 1, PH}^{m+1/2}  \nonumber 
        \\
         = & - \mathcal{A}_{h} \tilde{\mathsf{N}}_{h, 1}^{m+1/2} \nabla_{h}\mathrm{M}_{n , 2 , h}^{m+1/2} 
        - \mathcal{A}_{h} \tilde{\mathsf{N}}_{2}^{m+1/2} \nabla_{h} \mathrm{M}_{n , h, 1}^{m+1/2} 
        \nonumber \\ 
        & - \mathcal{A}_{h} \tilde{\mathsf{P}}_{h, 1}^{m+1/2} \nabla_{h}\mathrm{M}_{p , 2 , h}^{m+1/2} 
        - \mathcal{A}_{h} \tilde{\mathsf{P}}_{2}^{m+1/2} \nabla_{h} \mathrm{M}_{p , h , 1}^{m+1/2} 
         - \bm{\mathrm{G}}_{h}^{m+1/2} 
        + O (\tau^2 + h^2 ), 
        \\
        & 
        \frac{\mathsf{N}_{h, 1}^{m+1} - \mathsf{N}_{h, 1}^{m}}{\tau} 
        +  \nabla_{h} \cdot (\mathcal{A}_{h} \tilde{\mathsf{N}}_{h , 1}^{m+1/2} \mathrm{\bf U}_{2 , PH}^{m+1/2} 
        + \mathcal{A}_{h} \tilde{\mathsf{N}}_{2}^{m+1/2} \mathrm{\bf U}_{h , 1, PH}^{m+1/2} ) 
        \nonumber \\ 
        = & \nabla_{h} \cdot (\mathcal{A}_{h} \tilde{\mathsf{N}}_{h, 1}^{m+1/2} \nabla_{h} \mathrm{M}_{n , 2}^{m+1/2} 
        + \mathcal{A}_{h} \tilde{\mathsf{N}}_{2}^{m+1/2} \nabla_{h} \mathrm{M}_{n , h , 1}^{m+1/2} ) 
        - H_{n , h}^{m+1/2} + O (\tau^{2} + h^{2} ), 
        \\
        & 
        \mathrm{M}_{n , 2, h}^{m+1/2} 
        = F_{\mathsf{N}_{2}^{m}} (\mathsf{N}_{2}^{m+1} ) 
        + (-\Delta_h )^{-1} (\mathsf{N}_{2}^{m+1/2}-\mathsf{P}_{2}^{m+1/2} ), 
        \\
        & 
        \mathrm{M}_{n , h , 1}^{m+1/2} 
        = \frac{1}{\mathsf{N}_{2}^{m+1/2}} \mathsf{N}_{h, 1}^{m+1/2} 
        + (-\Delta_h )^{-1} (\mathsf{N}_{h, 1}^{m+1/2}-\mathsf{P}_{h, 1}^{m+1/2} ), 
        \\ 
         & 
        \frac{\mathsf{P}_{h, 1}^{m+1} - \mathsf{P}_{h, 1}^{m}}{\tau} 
        +  \nabla_{h} \cdot (\mathcal{A}_{h} \tilde{\mathsf{P}}_{h , 1}^{m+1/2} \mathrm{\bf U}_{2 , PH}^{m+1/2} 
        + \mathcal{A}_{h} \tilde{\mathsf{P}}_{2}^{m+1/2} \mathrm{\bf U}_{h , 1, PH}^{m+1/2} ) 
        \nonumber \\ 
         = & \nabla_{h} \cdot (\mathcal{A}_{h} \tilde{\mathsf{P}}_{h, 1}^{m+1/2} \nabla_{h} \mathrm{M}_{p , 2, h}^{m+1/2} 
        + \mathcal{A}_{h} \tilde{\mathsf{P}}_{2}^{m+1/2} \nabla_{h} \mathrm{M}_{p , h , 1}^{m+1/2} ) 
        - H_{p , h}^{m+1/2} + O (\tau^{2} + h^{2} ), 
        \\
        & 
        \mathrm{M}_{p , 2, h}^{m+1/2} 
        = F_{\mathsf{P}_{2}^{m}} (\mathsf{P}_{2}^{m+1} ) 
        + (-\Delta_h )^{-1} (\mathsf{P}_{2}^{m+1/2}-\mathsf{N}_{2}^{m+1/2} ), 
        \\
        & 
        \mathrm{M}_{p , h , 1}^{m+1/2} 
        = \frac{1}{\mathsf{P}_{2}^{m+1/2}} \mathsf{P}_{h, 1}^{m+1/2} 
        + (-\Delta_h )^{-1} (\mathsf{P}_{h, 1}^{m+1/2}-\mathsf{N}_{h, 1}^{m+1/2} ), 
        \\ 
        & 
        \nabla_{h} \cdot \mathrm{\bf U}_{h, 1, PH}^{m+1} = 0.
    \end{align}
\end{subequations} 
Finally, a combination of \eqref{spatial taylor:1} and \eqref{spatial taylor:2} yields an $O(\tau^{4}+h^{4})$ local truncation error for $\breve{\bf U}$, $\breve{\mathsf N}$, $\breve{\mathsf P}$ and $\breve{\Psi}$:  
\begin{subequations}\label{consistency}
    \begin{align}
        \frac{\hat{\breve{\mathrm{\bf U}}}^{m+1} - \breve{\mathrm{\bf U}}^{m}}{\tau}
        & +  b_{h} (\tilde{\mathrm{\bf U}}^{m+1/2},\hat{\breve{\mathrm{\bf U}}}^{m+1/2} ) 
        + \nabla_{h} \breve{\Psi}^{m} - \Delta_h \hat{\breve{\mathrm{\bf U}}}^{m+1/2} 
        \nonumber \\ 
        = & - \mathcal{A}_{h} \tilde{\breve{\mathsf{N}}}^{m+1/2} \nabla_{h} \breve{\mathrm{M}}_{n}^{m+1/2} 
        - \mathcal{A}_{h} \tilde{\breve{\mathsf{P}}}^{m+1/2} \nabla_{h} \breve{\mathrm{M}}_{p}^{m+1/2} 
        + \bm{\zeta}_u^m , 
        \\
        \frac{\breve{\mathsf{N}}^{m+1}-\breve{\mathsf{N}}^{m}}{\tau} 
        & +  \nabla_{h} \cdot (\mathcal{A}_{h} \tilde{\mathsf{N}}^{m+1/2} \hat{\breve{\mathrm{\bf U}}}^{m+1/2} ) 
        = \nabla_{h} \cdot (\mathcal{A}_{h} \tilde{\mathsf{N}}^{m+1/2} \nabla_{h} \breve{\mathrm{M}}_{n}^{m+1/2} )  + \zeta_{n}^m ,  
        \label{consistency:N} \\
        \breve{\mathrm{M}}_{n}^{m+1/2} 
        = & F_{\breve{\mathsf{N}}^{m}} (\breve{\mathsf{N}}^{m+1} ) 
        + (-\Delta_h )^{-1} (\breve{\mathsf{N}}^{m+1/2}-\breve{\mathsf{P}}^{m+1/2} ) 
        + \tau (\ln \breve{\mathsf{N}}^{m+1} - \ln \breve{\mathsf{N}}^{m} ), 
        \\ 
        \frac{\breve{\mathsf{P}}^{m+1}-\breve{\mathsf{P}}^{m}}{\tau} 
        & + \nabla_{h} \cdot (\mathcal{A}_{h} \tilde{\mathsf{P}}^{m+1/2} \hat{\breve{\mathrm{\bf U}}}^{m+1/2} ) 
        = \nabla_{h} \cdot (\mathcal{A}_{h} \tilde{\mathsf{P}}^{m+1/2} \nabla_{h} \breve{\mathrm{M}}_{p}^{m+1/2} ) + \zeta_{p}^m ,  
        \label{consistency:P} \\
        \breve{\mathrm{M}}_{p}^{m+1/2} 
        = & F_{\breve{\mathsf{P}}^{m}} (\breve{\mathsf{P}}^{m+1} ) 
        + (-\Delta_h )^{-1} (\breve{\mathsf{P}}^{m+1/2}-\breve{\mathsf{N}}^{m+1/2} ) 
        + \tau (\ln \breve{\mathsf{P}}^{m+1} - \ln \breve{\mathsf{P}}^{m} ), 
        \\ 
        \frac{\breve{\mathrm{\bf U}}^{m+1} - \hat{\breve{\mathrm{\bf U}}}^{m+1}}{\tau} 
        & +  \frac{1}{2} \nabla_{h} (\breve{\Psi}^{m+1}-\breve{\Psi}^{m} ) = 0, \\
        \nabla_{h} \cdot \breve{\mathrm{\bf U}}^{m+1} = & 0 , 
    \end{align}
\end{subequations} 
with $\| \bm{\zeta}_u^m \|_2$, $\| \zeta_n^m \|_2$, $\| \zeta_p^m \|_2 \le C (\tau^4 + h^4)$. 

A few more highlight explanations are provided for this higher order consistency analysis. 
\begin{enumerate} 
  \item In terms of the ion concentration variables, the following mass conservative identities and zero-average property for the local truncation error are available: 
\begin{equation} \label{mass conserv-3} 
    \begin{aligned}
        & n^{0} \equiv \breve{\mathrm{N}}^{0}, \quad p^{0} \equiv \breve{\mathrm{P}}^{0}, \quad 
        \overline{n^{k}}=\overline{n^{0}}, \quad \overline{p^{k}}=\overline{p^{0}}, \quad 
        \forall k \geq 0, 
        \\
        & \overline{\breve{\mathsf{N}}^{k}} 
        = \frac{1}{|\Omega|} \int_{\Omega} \breve{\mathsf{N}}\left(\cdot, t_{k}\right) \mathrm{d}\bm{x} 
        = \frac{1}{|\Omega|} \int_{\Omega} \breve{\mathsf{N}}^{0} \mathrm{d}\bm{x} 
        = \overline{n^{0}}, \quad \overline{\breve{\mathsf{P}}^{k}}=\overline{p^{0}}, \quad 
        \forall k \geq 0, 
        \\
        & \overline{\zeta_{n}^{m}}=\overline{\zeta_{p}^{m}}=0, \quad \forall m \ge 0. 
\end{aligned}
\end{equation} 
\item A similar phase separation estimate could be derived for the constructed $(\breve{\mathsf{N}}, \breve{\mathsf{P}})$, by taking $\tau$ and $h$ sufficiently small:  
\begin{equation}\label{eqn: separation}
    \breve{\mathsf{N}} \geq \frac{5 \delta}{8} , \quad 
    \breve{\mathsf{P}} \geq \frac{5 \delta}{8} ,  \quad 
    \mbox{for $\delta > 0$, at a point-wise level}.  
\end{equation} 
\item A discrete $W_h^{1, \infty}$ bound for the constructed profile $(\breve{\mathrm{\bf U}}, \breve{\mathsf{N}}, \breve{\mathsf{P}})$, as well as its discrete temporal derivative, is available:
  \begin{equation}\label{estimate: W}
      \begin{aligned}
        & \|\breve{\mathsf{N}}^{k} \|_{\infty} \leq C^*, \quad 
         \|\breve{\mathsf{P}}^{k} \|_{\infty} \leq C^*, \quad 
         \|\breve{\mathrm{\bf U}}^{k} \|_{\infty} \leq C^*, \quad 
        \|\hat{\breve{\mathrm{\bf U}}}^{m+1/2} \|_{\infty} \leq C^{*}, \quad 
        \forall k \geq 0, \\
        &  \|\nabla_{h} \breve{\mathsf{N}}^{k} \|_{\infty} \leq C^*, \quad 
         \|\nabla_{h} \breve{\mathsf{P}}^{k} \|_{\infty} \leq C^*, \quad 
         \|\nabla_{h} \breve{\mathrm{\bf U}}^{k} \|_{\infty} \leq C^*, \quad 
        \forall k \geq 0, \\
        &  \|\breve{\mathsf{N}}^{k+1}-\breve{\mathsf{N}}^{k} \|_{\infty} \leq C^* \tau, 
        \quad \|\breve{\mathsf{P}}^{k+1}-\breve{\mathsf{P}}^{k} \|_{\infty} \leq C^* \tau, 
        \quad \forall k \geq 0. 
    \end{aligned} 
\end{equation}
Furthermore, by the fact that $\breve{M}_{n}^{m+\frac{1}{2}}$ and $\breve{M}_{p}^{m+\frac{1}{2}}$ only depend on the exact solution $(\mathsf{N}, \mathsf{P})$, respectively, combined with a few correction functions, it is natural to assume a discrete $W_{h}^{1, \infty}$ bound
\begin{equation}\label{estimate: M}
    \|\nabla_{h} \breve{M}_{n}^{m+1/2} \|_{\infty} \leq C^{*}, \quad 
    \|\nabla_{h} \breve{M}_{p}^{m+1/2} \|_{\infty} \leq C^{*}. 
\end{equation} 
\end{enumerate}
\begin{remark}
    Based on the phase separation estimate \eqref{eqn: separation} for the constructed functions $(\breve{\mathsf{N}}, \breve{\mathsf{P}})$, as well as their regularity in time, it is clear that an explicit extrapolation formula in the mobility function in \eqref{consistency:N} and \eqref{consistency:P} has to create a point-wise positive mobility concentration value, a numerical approximation at time instant $t^{m+1/2}$. Therefore, the positive regularization formula in \eqref{eqn:interpolation}) could be avoided in the consistency analysis.
\end{remark}

\subsection{A rough error estimate} 
The following error functions are defined:
\begin{equation}\label{def: error function}
    \begin{array}{lll}
    e_{\bm{u}}^{m} = \breve{\mathrm{\bf U}}^{m}-\bm{u}^{m}, 
    & e_{\bm{u}}^{m+1/2} = \breve{\mathrm{\bf U}}^{m+1/2} - \bm{u}^{m+1/2}, 
    & \tilde{e}_{\bm{u}}^{m+1/2} = \tilde{\mathrm{\bf U}}^{m+1/2}-\tilde{\bm{u}}^{m+1/2}, 
    \\
    \hat{e}_{\bm{u}}^{m+1} = \hat{\breve{\mathrm{\bf U}}}^{m+1}-\hat{\bm{u}}^{m+1}, 
    & \hat{e}_{\bm{u}}^{m+1/2} = \frac{1}{2} (e_{\bm{u}}^{m} + \hat{e}_{\bm{u}}^{m+1} ), 
    & e_{\mu_{n}}^{m+1/2} = \breve{\mathrm{M}}_{n}^{m+1/2}-\mu_{n}^{m+1/2}, 
    \\
    e_{\psi}^{m} = \breve{\Psi}^{m}-\psi^{m}, 
    & e_{\psi}^{m+1/2} = \breve{\Psi}^{m+1/2} - \psi^{m+1/2}, 
    & e_{\mu_{p}}^{m+1/2} = \breve{\mathrm{M}}_{p}^{m+1/2}-\mu_{p}^{m+1/2}, 
    \\
    e_{n}^{m} = \breve{\mathsf{N}}^{m} - n^{m}, 
    & e_{n}^{m+1/2} = \frac{1}{2} ( e_{n}^{m+1} + e_{n}^{m} ), 
    & \tilde{e}_{n}^{m+1/2} = \frac{3}{2} {e}_{n}^{m} - \frac{1}{2} {e}_{n}^{m-1}, 
    \\
    e_{p}^{m} = \breve{\mathsf{P}}^{m} - p^{m}, 
    & e_{p}^{m+1/2} = \frac{1}{2} ( e_{p}^{m+1} + e_{p}^{m} ), 
    & \tilde{e}_{p}^{m+1/2} = \frac{3}{2} {e}_{p}^{m} - \frac{1}{2} {e}_{p}^{m-1}, 
    \\
    e_{\phi}^{m} = ( - \Delta_h )^{-1} ( e_{p}^{m} - e_{n}^{m} ), 
    & e_{\phi}^{m+1/2} = \frac{1}{2} ( e_{\phi}^{m+1} + e_{\phi}^{m} ), 
    & \tilde{e}_{\phi}^{m+1/2} = \frac{3}{2} {e}_{\phi}^{m} - \frac{1}{2} {e}_{\phi}^{m-1}. 
\end{array}
\end{equation} 
Because of the mass conservation identity~\eqref{mass conserv-3}, it is clear that the error functions $e_{n}^{k}$ and $e_{p}^{k}$ are always average-free: $\overline{e_{n}^{k}}=0$ and $\overline{e_{p}^{k}}=0$, for any $k \geq 0$. In turn, their $\|\cdot\|_{-1, h}$ norms are well defined. 
\begin{lemma} \label{estimate:B and Bnorm}
A trilinear form is introduced as 
$\mathcal{B}\left(\bm{u}, \bm{v}, \bm{w}\right) = \left\langle b_{h}\left(\bm{u}, \bm{v}\right), \bm{w}\right\rangle_{1}$. 
The following estimates are valid:
\begin{subequations}
    \begin{align}
        & \mathcal{B}\left(\bm{u}, \bm{v}, \bm{v}\right)=0, 
        \label{estimate:B} \\ 
        & \left|\mathcal{B}\left(\bm{u}, \bm{v}, \bm{w}\right)\right| 
        \leq \frac{1}{2}\left\|\bm{u}\right\|_{2}\left(\left\|\nabla_{h} \bm{v}\right\|_{\infty} \cdot\|\bm{w}\|_{2} 
        + \left\|\nabla_{h} \bm{w}\right\|_{2} \cdot\left\|\bm{v}\right\|_{\infty}\right). 
        \label{estimate:Bnorm}
    \end{align}
\end{subequations}
\end{lemma}
\begin{proof}
Identity \eqref{estimate:B} comes from the summation by parts formula
\begin{equation}
    \mathcal{B}\left(\bm{u}, \bm{v}, \bm{v}\right) 
    = \left\langle b_{h}\left(\bm{u}, \bm{v}\right), \bm{v}\right\rangle_{1} 
    = \frac{1}{2}\left(\left\langle\bm{u} \cdot \nabla_{h} \bm{v}, \bm{v}\right\rangle_{1} 
    + \left\langle\nabla_{h} \cdot\left(\bm{u} \bm{v}^{T}\right), \bm{v}\right\rangle_{1}\right) = 0. 
\end{equation} 

Inequality \eqref{estimate:Bnorm} could be derived as follows: 
\begin{equation}
    \begin{aligned}
        \left|\mathcal{B}\left(\bm{u}, \bm{v}, \bm{w}\right)\right| 
        & = \frac{1}{2} \left| \left\langle\bm{u} \cdot \nabla_{h} \bm{v}, \bm{w}\right\rangle_{1} 
        + \left\langle\nabla_{h} \cdot\left(\bm{u} \bm{v}^{T}\right), \bm{w}\right\rangle_{1} \right| 
        \\
        & = \frac{1}{2}\left|\left\langle\bm{u} \cdot \nabla_{h} \bm{v}, \bm{w}\right\rangle_{1} 
        - \left\langle\bm{u} \cdot \nabla_{h} \bm{w}, \bm{v}\right\rangle_{1}\right| 
        \\
        & \leq \frac{1}{2}\left( \left| \left\langle \bm{u} \cdot \nabla_{h} \bm{v}, \bm{w} \right \rangle_{1} \right| 
        + \left| \left\langle \bm{u} \cdot \nabla_{h} \bm{w}, \bm{v} \right\rangle_{1} \right| \right) 
        \\
        & \leq \frac{1}{2}\|\bm{u}\|_{2}\left(\left\|\nabla_{h} \bm{v}\right\|_{\infty} \cdot\|\bm{w}\|_{2} 
        + \left\|\nabla_{h} \bm{w}\right\|_{2} \cdot\|\bm{v}\|_{\infty}\right). 
    \end{aligned} 
\end{equation} 
This completes the proof. 
\end{proof}

Taking a difference between the numerical system \eqref{scheme: main} and the consistency estimate \eqref{consistency} results in the following error evolutionary equations: 
\begin{subequations}
    \begin{align}
        \frac{\hat{e}_{\bm{u}}^{m+1}-e_{\bm{u}}^{m}}{\tau} 
        & +  b_{h} (\tilde{e}_{\bm{u}}^{m+1/2}, \hat{\breve{\mathrm{\bf U}}}^{m+1/2} ) 
        + b_{h} (\tilde{\bm{u}}^{m+1/2}, \hat{e}_{\bm{u}}^{m+1/2} ) 
        + \nabla_{h} e_{\psi}^{m} 
        - \Delta_{h} \hat{e}_{\bm{u}}^{m+1/2} 
        \nonumber \\
        & +  \mathcal{A}_{h} \tilde{e}_{n}^{m+1/2} \nabla_{h} \breve{\mathrm{M}}_{n}^{m+1/2} 
        + \mathcal{A}_{h} \tilde{n}^{m+1/2} \nabla_{h} e_{\mu_{n}}^{m+1/2} 
        \nonumber \\ 
        & +  \mathcal{A}_{h} \tilde{e}_{p}^{m+1/2} \nabla_{h} \breve{\mathrm{M}}_{p}^{m+1/2} 
        + \mathcal{A}_{h} \tilde{p}^{m+1/2} \nabla_{h} e_{\mu_{p}}^{m+1/2} 
        = \bm{\zeta}_u^{m}, 
        \label{error: u_hat} \\
        \frac{e_{n}^{m+1}-e_{n}^{m}}{\tau} 
        & +  \nabla_{h} \cdot (\mathcal{A}_{h} \tilde{e}_{n}^{m+1/2} \hat{\breve{\mathrm{\bf U}}}^{m+1/2} 
        + \mathcal{A}_{h} \tilde{n}^{m+1/2} \hat{e}_{\bm{u}}^{m+1/2} ) 
        \nonumber \\ 
        = & \nabla_{h} \cdot (\mathcal{A}_{h} \tilde{e}_{n}^{m+1/2} \nabla_{h} \breve{\mathrm{M}}_{n}^{m+1/2} 
        + \mathcal{A}_{h} \tilde{n}^{m+1/2} \nabla_{h} e_{\mu_{n}}^{m+1/2} ) 
        + \zeta_{n}^{m}, 
        \label{error: n}\\
        e_{\mu_{n}}^{m+1/2} 
        = & F_{\breve{\mathsf{N}}^{m}} (\breve{\mathsf{N}}^{m+1} ) 
        - G_{n^{m}} (n^{m+1} ) 
        + ( - \Delta_h )^{-1} ( e_{n}^{m+1/2} - e_{p}^{m+1/2} ) 
        \nonumber \\ 
        &  + \tau ( \ln \breve{\mathsf{N}}^{m+1} - \ln n^{m+1} - \ln \breve{\mathsf{N}}^{m} + \ln n^{m} ), 
        \label{error: mu_n}\\ 
        \frac{e_{p}^{m+1}-e_{p}^{m}}{\tau} 
        & +  \nabla_{h} \cdot (\mathcal{A}_{h} \tilde{e}_{p}^{m+1/2} \hat{\breve{\mathrm{\bf U}}}^{m+1/2} 
        + \mathcal{A}_{h} \tilde{p}^{m+1/2} \hat{e}_{\bm{u}}^{m+1/2} ) 
        \nonumber \\ 
        = & \nabla_{h} \cdot (\mathcal{A}_{h} \tilde{e}_{p}^{m+1/2} \nabla_{h} \breve{\mathrm{M}}_{p}^{m+1/2} 
        + \mathcal{A}_{h} \tilde{p}^{m+1/2} \nabla_{h} e_{\mu_{p}}^{m+1/2} ) 
        + \zeta_{p}^{m}, 
        \label{error: p}\\
        e_{\mu_{p}}^{m+1/2} 
        = & G_{\breve{\mathsf{P}}^{m}} (\breve{\mathsf{P}}^{m+1} ) 
        - G_{p^{m}} (p^{m+1} ) 
        + ( - \Delta_h )^{-1} ( e_{p}^{m+1/2} - e_{n}^{m+1/2} ) 
        \nonumber \\ 
        & + \tau ( \ln \breve{\mathsf{P}}^{m+1} - \ln p^{m+1} - \ln \breve{\mathsf{P}}^{m} + \ln p^{m} ), 
        \label{error: mu_p}\\ 
        \frac{e_{\bm{u}}^{m+1} - \hat{e}_{\bm{u}}^{m+1}}{\tau} 
        & +  \frac{1}{2} \nabla_{h} (e_{\psi}^{m+1} - e_{\psi}^{m} ) = 0, 
        \label{error: psi}\\
        \nabla_{h} \cdot e_{\bm{u}}^{m+1} = & 0. 
    \end{align}
\end{subequations} 
To proceed with the convergence analysis, the following a-priori assumption is made for the numerical error functions at the previous time steps:
\begin{equation}\label{prior assumption}
    \| e_{\bm{u}}^{k} \|_{2}, \, \| e_{n}^{k} \|_{2}, \, \| e_{p}^{k} \|_{2} 
    \leq\tau^{\frac{15}{4}} + h^{\frac{15}{4}}, 
    \quad k = m, m-1, \quad 
    \| \nabla_{h} e_{\psi}^{m} \|_{2} \leq \tau^{\frac{11}{4}} + h^{\frac{11}{4}}. 
\end{equation} 
Such an a-priori assumption will be recovered by the error estimate in the next time step, 
which will be demonstrated later. Of course, this a-priori assumption leads to a $W_{h}^{1, \infty}$ bound for the numerical error function at the previous time steps, which comes from the inverse inequality, the linear refinement requirement $\lambda_{1} h \leq \tau \leq \lambda_{2} h$, as well as the discrete Poincar\'e inequality (stated in Proposition~\ref{prop:1}):
\begin{equation}\label{eqn: estimate of en and ep}
    \begin{aligned}
        & \| e_{n}^{k} \|_{\infty} 
        \leq \frac{C \| e_{n}^{k} \|_{2}}{h} 
        \leq C ( \tau^{\frac{11}{4}} + h^{\frac{11}{4}} ) 
        \leq \tau^\frac52 + h^\frac52 \leq \frac{\delta}{8}, 
        \\ 
        & \| e_{p}^{k} \|_{\infty} 
        \leq \frac{C \|e_{p}^{k} \|_{2}}{h} 
        \leq C ( \tau^{\frac{11}{4}} + h^{\frac{11}{4}} ) 
        \leq \tau^\frac52 + h^\frac52 \leq \frac{\delta}{8}, 
        \\  
        & \| \nabla_{h} e_{n}^{k} \|_{\infty} 
        \leq \frac{2 \| e_{n}^{k} \|_{\infty}}{h} 
        \leq C ( \tau^{\frac{7}{4}} + h^{\frac{7}{4}} ) 
        \leq \tau + h \leq 1, 
        \\ 
        & \| \nabla_{h} e_{p}^{k} \|_{\infty} 
        \leq \frac{\left\| e_{p}^{k} \right\|_{\infty}}{h} 
        \leq C (\tau^{\frac{7}{4}} + h^{\frac{7}{4}} ) 
        \leq \tau + h \leq 1 , 
    \end{aligned}
\end{equation} 
for $k=m, m-1$, provided that $\tau$ and $h$ are sufficiently small. Subsequently, with the help of the regularity assumption \eqref{estimate: W}, an $\ell^{\infty}$ bound for the numerical solution could be derived at the previous time steps: 
\begin{equation}\label{eqn: Linfty estimate of nk and pk}
    \begin{aligned}
        \| n^{k} \|_{\infty} 
        \leq \| \breve{\mathsf{N}}^{k} \|_{\infty} 
        + \| e_{n}^{k} \|_{\infty} 
        \leq \tilde{C}_1 := C^{\star} + 1,  \quad 
        \| p^{k} \|_{\infty} 
        \leq \| \breve{\mathsf{P}}^{k} \|_{\infty} 
        + \| e_{p}^{k} \|_{\infty} 
        \leq \tilde{C}_1 . 
    \end{aligned}
\end{equation} 
Moreover, a combination of the $\ell^{\infty}$ estimate \eqref{eqn: estimate of en and ep} for the numerical error function and the separation estimate \eqref{eqn: separation} leads to a similar separation bound for the numerical solution at the previous time steps: 
\begin{equation}\label{eqn: estimate of nk and pk}
    n^{k} \geq \breve{\mathsf{N}}^k - \| e_{n}^{k} \|_{\infty} 
    \geq \frac{\delta}{2} , 
    \quad \text{ and } \quad 
    p^{k} \geq \breve{\mathsf{P}}^k - \| e_{p}^{k} \|_{\infty} 
    \geq \frac{\delta}{2} . 
\end{equation} 
Therefore, at the intermediate time instant $t^{m+1/2}$, the following estimates would be available: 
\begin{equation}
    \begin{aligned}
        \frac{3}{2} \breve{\mathsf{N}}^{m}
        -\frac{1}{2} \breve{\mathsf{N}}^{m-1} 
        & = \frac{1}{2} (\breve{\mathsf{N}}^{m+1} 
        + \breve{\mathsf{N}}^{m} ) 
        + O (\tau^2) , 
        \quad \text { since } 
        \breve{\mathsf{N}}^{m+1} 
        -2 \breve{\mathsf{N}}^{m} 
        + \breve{\mathsf{N}}^{m-1} 
        = O (\tau^{2}), 
        \\
        \frac{1}{2} (\breve{\mathsf{N}}^{m+1} 
        + \breve{\mathsf{N}}^{m} ) 
        & = \breve{\mathsf{N}} (t^{m+1/2} ) 
        + O (\tau^{2} ), 
        \quad 
        \breve{\mathsf{N}} (t^{m+1/2}) 
        \geq \frac{5 \delta}{8} , \quad(\text {by }\eqref{eqn: separation}), 
        \\
        & \quad \text {so that } 
        \frac{3}{2} \breve{\mathsf{N}}^{m}
        -\frac{1}{2} \breve{\mathsf{N}}^{m-1} 
        \geq \frac{5 \delta}{8}  
        - O (\tau^{2}), 
        \\
        \|\frac{3}{2} e_{n}^{m} -\frac{1}{2} e_{n}^{m-1} \|_{\infty} 
        & \le C (\tau^{\frac{11}{4}}+h^{\frac{11}{4}} ), 
        \quad(\text { by }\eqref{eqn: estimate of en and ep}), 
        \\
        \tilde{n}^{m+1/2} & = \frac{3}{2} n^{m} - \frac{1}{2} n^{m-1} 
        = \frac{3}{2} \breve{\mathsf{N}}^{m} - \frac{1}{2} \breve{\mathsf{N}}^{m-1} 
        - \frac{3}{2} e_{n}^{m} - \frac{1}{2} e_{n}^{m-1} 
        \\ 
        & \geq \frac{5 \delta}{8}  
        - O (\tau^{2} ) 
        - O ( \tau^{\frac{11}{4}} + h^{\frac{11}{4}} ) 
        \geq \frac{\delta}{2}, 
        \\ 
        \tilde{p}^{m+1/2} & = \frac{3}{2} p^{m} - \frac{1}{2} p^{m-1} 
        = \frac{3}{2} \breve{\mathrm{P}}^{m} - \frac{1}{2} \breve{\mathrm{P}}^{m-1} 
        - \frac{3}{2} e_{p}^{m} - \frac{1}{2} e_{p}^{m-1} 
        \ge \frac{\delta}{2}. 
    \end{aligned}
    \label{separation est-3} 
\end{equation} 
As a result, the phase separation bound for the average mobility functions, $\tilde{n}^{m+1/2}$ and $\tilde{p}^{m+1/2}$, has also been established, and such a bound will be useful in the later analysis.

Taking a discrete inner product with \eqref{error: u_hat} by $2\hat{e}_{\bm{u}}^{m+1/2} = \hat{e}_{\bm{u}}^{m+1} + e_{\bm{u}}^{m}$ leads to  
\begin{equation}\label{error estimate: u_hat}
    \begin{aligned}
        & \frac{1}{\tau} ( \|\hat{e}_{\bm{u}}^{m+1} \|_{2}^{2} - \| e_{\bm{u}}^{m} \|_{2}^{2} ) 
        + 2 \mathcal{B} (\tilde{e}_{\bm{u}}^{m+1/2}, \hat{\mathrm{\bf U}}^{m+1/2}, \hat{e}_{\bm{u}}^{m+1/2} ) 
        + 2 \mathcal{B} (\tilde{\bm{u}}^{m+1/2}, \hat{e}_{\bm{u}}^{m+1/2}, \hat{e}_{\bm{u}}^{m+1/2} ) 
        \\ 
        & + 2 \|\nabla_{h} \hat{e}_{\bm{u}}^{m+1/2} \|_{2}^{2} 
        = - \langle\nabla_{h} e_{\psi}^{m}, \hat{e}_{\bm{u}}^{m+1}+e_{\bm{u}}^{m} \rangle_{1} 
        - 2 \langle\mathcal{A}_{h} \tilde{e}_{n}^{m+1/2} \nabla_{h} \breve{\mathrm{M}}_{n}^{m+1/2}, \hat{e}_{\bm{u}}^{m+1/2} \rangle_{1} 
        \\
        &- 2 \langle\mathcal{A}_{h} \tilde{n}^{m+1/2} \nabla_{h} e_{\mu_{n}}^{m+1/2}, \hat{e}_{\bm{u}}^{m+1/2} \rangle_{1} 
        - 2 \langle\mathcal{A}_{h} \tilde{e}_{p}^{m+1/2} \nabla_{h} \breve{\mathrm{M}}_{p}^{m+1/2}, \hat{e}_{\bm{u}}^{m+1/2} \rangle_{1} 
        \\
        &- 2 \langle\mathcal{A}_{h} \tilde{p}^{m+1/2} \nabla_{h} e_{\mu_{p}}^{m+1/2}, \hat{e}_{\bm{u}}^{m+1/2} \rangle_{1} 
        + 2 \langle\bm{\zeta}_u^{m}, \hat{e}_{\bm{u}}^{m+1/2} \rangle_{1}. 
    \end{aligned}
\end{equation} 
With an application of the nonlinear identity \eqref{estimate:B} in Lemma \ref{estimate:B and Bnorm}, we immediately get 
\begin{equation}\label{error identity: B}
    \mathcal{B} (\tilde{\bm{u}}^{m+1/2}, \hat{e}_{\bm{u}}^{m+1/2}, \hat{e}_{\bm{u}}^{m+1/2} ) = 0. 
\end{equation} 
The second term on the left hand side of \eqref{error estimate: u_hat} could be bounded with the help of inequality \eqref{estimate:Bnorm}: 
\begin{equation}\label{error estimate: B}
    \begin{aligned}
        & 2 |\mathcal{B} (\tilde{e}_{\bm{u}}^{m+1/2}, \hat{\mathrm{\bf U}}^{m+1/2}, \hat{e}_{\bm{u}}^{m+1/2} ) | 
        \\ 
        \leq & \|\tilde{e}_{\bm{u}}^{m+1/2} \|_{2} ( \|\nabla_{h} \hat{\mathrm{\bf U}}^{m+1/2} \|_{\infty} \cdot \|\hat{e}_{\bm{u}}^{m+1/2} \|_{2} 
        + \|\hat{\mathrm{U}}^{m+1/2} \|_{\infty} \cdot \|\nabla_{h} \hat{e}_{\bm{u}}^{m+1/2} \|_{2} ) 
        \\
        \leq & C^{*} \|\tilde{e}_{\bm{u}}^{m+1/2} \|_{2} ( \|\hat{e}_{\bm{u}}^{m+1/2} \|_{2} 
        + \|\nabla_{h} \hat{e}_{\bm{u}}^{m+1/2} \|_{2} ) 
        \\ 
        \leq & C^{*} \|\tilde{e}_{\bm{u}}^{m+1/2} \|_{2} \cdot (C_{0}+1 ) \|\nabla_{h} \hat{e}_{\bm{u}}^{m+1/2} \|_{2} 
        \le \frac{(C^{*} (C_{0}+1 ) )^{2}}{2} \|\tilde{e}_{\bm{u}}^{m+1/2} \|_{2}^{2} 
        + \frac{1}{2} \|\nabla_{h} \hat{e}_{\bm{u}}^{m+1/2} \|_{2}^{2} 
        \\ 
        \leq & \tilde{C}_{2} (3 \|e_{\bm{u}}^{m} \|_{2}^{2} 
        +  \|e_{\bm{u}}^{m-1} \|_{2}^{2} ) 
        + \frac{1}{2} \|\nabla_{h} \hat{e}_{\bm{u}}^{m+1/2} \|_{2}^{2}, 
\end{aligned}
\end{equation} 
in which $\tilde{C}_{2}=\frac{\left(C^{*}\left(C_{0}+1\right)\right)^{2}}{2}$, and the $W_{h}^{1, \infty}$ assumption \eqref{estimate: W} (for the constructed solution $\breve{\mathrm{\bf U}}$) has been used in the derivation. In fact, the discrete Pincar\'e inequality, $\|\hat{e}_{\bm{u}}^{m+1/2} \|_{2} \leq C_{0} \|\nabla_{h} \hat{e}_{\bm{u}}^{m+1/2} \|_{2}$, (which comes from Proposition~\ref{prop:1}), was applied in the second step, because of the no-penetration boundary condition for $\hat{e}_{\bm{u}}^{m+1/2}$.

In terms of numerical error inner product associated with the pressure gradient, we have 
\begin{equation}\label{error idendity: psi}
    \langle\nabla_{h} e_{\psi}^{m}, e_{\bm{u}}^{k} \rangle_{1} 
    = - \langle e_{\psi}^{m}, \nabla_{h} \cdot e_{\bm{u}}^{k} \rangle_{C} = 0, \quad \text { since } 
    \nabla_{h} \cdot e_{\bm{u}}^{m}=0, \quad k = m, m+1, 
\end{equation} 
in which the summation by parts formula \eqref{lem: discrete inner of u and f} has been applied. Regarding the other pressure gradient inner product term, an application of \eqref{error: psi} indicates that 
\begin{equation}\label{error estimate: psi}
    \begin{aligned}
        \langle\nabla_{h} e_{\psi}^{m}, \hat{e}_{\bm{u}}^{m+1} \rangle_{1} 
        & = \langle\nabla_{h} e_{\psi}^{m}, e_{\bm{u}}^{m+1} \rangle_{1} 
        + \frac{1}{2} \tau \langle\nabla_{h} e_{\psi}^{m}, \nabla_{h} (e_{\psi}^{m+1}-e_{\psi}^{m} ) \rangle_{1} 
        \\
        & = \frac{1}{2} \tau \langle\nabla_{h} e_{\psi}^{m}, \nabla_{h} (e_{\psi}^{m+1}-e_{\psi}^{m} ) \rangle_{1} 
        \\
        & = \frac{1}{4} \tau ( \|\nabla_{h} e_{\psi}^{m+1} \|_{2}^{2} 
        -  \|\nabla_{h} e_{\psi}^{m} \|_{2}^{2} 
        - \|\nabla_{h} (e_{\psi}^{m+1} - e_{\psi}^{m} ) \|_{2}^{2} ), 
    \end{aligned}
\end{equation} 
where the second step is based on the fact that $\langle\nabla_{h} e_{\psi}^{m}, e_{\bm{u}}^{m+1} \rangle_{1}=0$. In terms of the second term on the right hand side of \eqref{error estimate: u_hat}, an application of discrete H\"older inequality gives 
\begin{subequations}\label{error estimate: U_NP}
    \begin{align}
        & -2 \langle\mathcal{A}_{h} \tilde{e}_{n}^{m+1/2} \nabla_{h} \breve{\mathrm{M}}_{n}^{m+1/2}, 
        \hat{e}_{\bm{u}}^{m+1/2} \rangle_{1} 
        \leq 2 \|\tilde{e}_{n}^{m+1/2} \|_{2} 
        \cdot \|\nabla_{h} \breve{\mathrm{M}}_{n}^{m+1/2} \|_{\infty} 
        \cdot \|\hat{e}_{\bm{u}}^{m+1/2} \|_{2} 
        \nonumber \\
        \leq & 2  C^{*} \|\tilde{e}_{n}^{m+1/2} \|_{2} 
        \cdot \|\hat{e}_{\bm{u}}^{m+1/2} \|_{2} 
        \leq 2  C_{0} C^{*} \|\tilde{e}_{n}^{m+1/2} \|_{2} 
        \cdot  \|\nabla_{h} \hat{e}_{\bm{u}}^{m+1/2} \|_{2} 
        \nonumber \\
        \leq & 8 C_{0}^{2} (C^{*})^{2} \|\tilde{e}_{n}^{m+1/2} \|_{2}^{2} 
        + \frac{1}{8} \|\nabla_{h} \hat{e}_{\bm{u}}^{m+1/2} \|_{2}^{2} 
        \leq \tilde{C}_{3} (3 \|e_{n}^{m} \|_{2}^{2} + \|e_{n}^{m-1} \|_{2}^{2} ) 
        + \frac{1}{8} \|\nabla_{h} \hat{e}_{\bm{u}}^{m+1/2} \|_{2}^{2}, 
        \\ 
        & -2 \langle\mathcal{A}_{h} \tilde{e}_{p}^{m+1/2} \nabla_{h} \breve{\mathrm{M}}_{p}^{m+1/2}, 
        \hat{e}_{\bm{u}}^{m+1/2} \rangle_{1} 
        \leq 2 \|\tilde{e}_{p}^{m+1/2} \|_{2} 
        \cdot \|\nabla_{h} \breve{\mathrm{M}}_{p}^{m+1/2} \|_{\infty} 
        \cdot \|\hat{e}_{\bm{u}}^{m+1/2} \|_{2} 
        \nonumber \\
        \leq & 2  C^* \|\tilde{e}_{p}^{m+1/2} \|_{2} 
        \cdot \|\hat{e}_{\bm{u}}^{m+1/2} \|_{2} 
        \leq 2  C_{0} C^* \|\tilde{e}_{p}^{m+1/2} \|_{2} 
        \cdot \|\nabla_{h} \hat{e}_{\bm{u}}^{m+1/2} \|_{2} 
        \nonumber \\
        \leq & 8 C_{0}^{2} (C^*)^2 \|\tilde{e}_{p}^{m+1/2} \|_{2}^{2} 
        + \frac{1}{8} \|\nabla_{h} \hat{e}_{\bm{u}}^{m+1/2} \|_{2}^{2} 
        \leq \tilde{C}_{3} (3 \|e_{p}^{m} \|_{2}^{2} + \| e_{p}^{m-1} \|_{2}^{2} ) 
        + \frac{1}{8} \|\nabla_{h} \hat{e}_{\bm{u}}^{m+1/2} \|_{2}^{2}, 
    \end{align}
\end{subequations} 
with $\tilde{C}_{3} = 8 C_{0}^{2}\left(C^{*}\right)^{2}$. Again, the $W_{h}^{1, \infty}$ assumption \eqref{estimate: M} and the discrete Poincaré inequality, $\|\hat{e}_{\bm{u}}^{m+1/2} \|_{2} \leq C_0 \|\nabla_{h} e_{\bm{u}}^{m+1/2} \|_{2}$, have been applied in the derivation. A bound for the local truncation error term would be straightforward:
\begin{equation}\label{error estimate: zeta1}
        2 \langle\bm{\zeta}_u^{m}, \hat{e}_{\bm{u}}^{m+\frac12} \rangle_{1} 
        \le 2 \|\bm{\zeta}_u^{m} \|_{2} \cdot \|\hat{e}_{\bm{u}}^{m+\frac12} \|_{2} 
        \le 2 C_{0} \|\bm{\zeta}_u^{m} \|_{2} \cdot \|\nabla_{h} \hat{e}_{\bm{u}}^{m+\frac12} \|_{2} 
        \leq 4 C_{0}^{2} \|\bm{\zeta}_u^{m} \|_{2}^{2}+\frac{1}{4} \|\nabla_{h} \hat{e}_{\bm{u}}^{m+\frac12} \|_{2}^{2}. 
\end{equation} 
As a consequence, a substitution of \eqref{error identity: B}-\eqref{error estimate: zeta1} into \eqref{error estimate: u_hat} results in 
\begin{equation}\label{error estimate: eu_hat}
    \begin{aligned}
        & \frac{1}{\tau} ( \|\hat{e}_{\bm{u}}^{m+1} \|_{2}^{2} - \|e_{\bm{u}}^{m} \|_{2}^{2} ) 
        + \|\nabla_{h} \hat{e}_{\bm{u}}^{m+1/2} \|_{2}^{2} 
        + \frac{\tau}{4} ( \|\nabla_{h} e_{\psi}^{m+1} \|_{2}^{2} - \|\nabla_{h} e_{\psi}^{m} \|_{2}^{2} )  
        \\
        \leq & - 2 \langle\mathcal{A}_{h} \tilde{n}^{m+1/2} \nabla_{h} e_{\mu_{n}}^{m+1/2}, \hat{e}_{\bm{u}}^{m+1/2} \rangle_{1} 
        - 2 \langle\mathcal{A}_{h} \tilde{p}^{m+1/2} \nabla_{h} e_{\mu_{p}}^{m+1/2}, \hat{e}_{\bm{u}}^{m+1/2} \rangle_{1} 
        \\
        & + \frac{\tau}{4} \|\nabla_{h} (e_{\psi}^{m+1}-e_{\psi}^{m} ) \|_{2}^{2} 
        + \tilde{C}_{2} (3 \|e_{\bm{u}}^{m} \|_{2}^{2} 
        + \|e_{\bm{u}}^{m-1} \|_{2}^{2} ) 
        \\
        & + \tilde{C}_{3} (3 \|e_{n}^{m} \|_{2}^{2}+ \|e_{n}^{m-1} \|_{2}^{2} ) 
        + \tilde{C}_{3} (3 \|e_{p}^{m} \|_{2}^{2}+ \|e_{p}^{m-1} \|_{2}^{2} ) 
        + 4 C_{0}^{2} \|\bm{\zeta}_u^{m} \|_{2}^{2}. 
    \end{aligned}
\end{equation} 
On the other hand, a discrete inner product with \eqref{error: psi} by $2 e_{\bm{u}}^{m+1}$ yields 
\begin{equation}\label{error estimate: u and psi}
    \begin{aligned}
        & \|e_{\bm{u}}^{m+1} \|_{2}^{2}  - \|\hat{e}_{\bm{u}}^{m+1} \|_{2}^{2} 
        + \|e_{\bm{u}}^{m+1}-\hat{e}_{\bm{u}}^{m+1} \|_{2}^{2} = 0, \quad 
        \text{so that}
        \\ 
        & \|e_{\bm{u}}^{m+1} \|_{2}^{2} - \|\hat{e}_{\bm{u}}^{m+1} \|_{2}^{2} 
        + \frac{\tau^{2}}{4} \|\nabla_{h} (e_{\psi}^{m+1}-e_{\psi}^{m} ) \|_{2}^{2} = 0 ,  
    \end{aligned}
\end{equation} 
where the discrete divergence-free condition for $e_{\bm{u}}^{m+1}$ has been applied. Subsequently, a combination of \eqref{error estimate: eu_hat} and \eqref{error estimate: u and psi} leads to 
\begin{equation}\label{error estimate: e}
    \begin{aligned}
        & \frac{1}{\tau} ( \|e_{\bm{u}}^{m+1} \|_{2}^{2}- \|e_{\bm{u}}^{m} \|_{2}^{2} ) 
        + \|\nabla_{h} \hat{e}_{\bm{u}}^{m+1/2} \|_{2}^{2} 
        + \frac{\tau}{4} ( \|\nabla_{h} e_{\psi}^{m+1} \|_{2}^{2}- \|\nabla_{h} e_{\psi}^{m} \|_{2}^{2} ) 
        \\
        \leq & - 2 \langle\mathcal{A}_{h} \tilde{n}^{m+1/2} \nabla_{h} e_{\mu_{n}}^{m+1/2}, \hat{e}_{\bm{u}}^{m+1/2} \rangle_{1} 
        - 2 \langle\mathcal{A}_{h} \tilde{p}^{m+1/2} \nabla_{h} e_{\mu_{p}}^{m+1/2}, \hat{e}_{\bm{u}}^{m+1/2} \rangle_{1} 
        \\ 
        & + \tilde{C}_{2} (3 \|e_{\bm{u}}^{m} \|_{2}^{2} 
        + \|e_{\bm{u}}^{m-1} \|_{2}^{2} ) 
        + \tilde{C}_{3} (3 \|e_{n}^{m} \|_{2}^{2}+ \|e_{n}^{m-1} \|_{2}^{2} 
        + 3 \|e_{p}^{m} \|_{2}^{2}+ \|e_{p}^{m-1} \|_{2}^{2} ) 
        + 4 C_{0}^{2} \|\bm{\zeta}_u^{m} \|_{2}^{2}. 
\end{aligned}
\end{equation} 
Now we proceed into a rough error estimate for the PNP error evolutionary equation. Taking a discrete inner product with \eqref{error: n} and \eqref{error: p} by $e_{\mu_{n}}^{m+1/2}$ and $e_{\mu_{p}}^{m+1/2}$ respectively, and a summation gives  
\begin{equation}\label{estimate: en and ep}
    \begin{aligned}
        & \frac{1}{\tau} \langle e_{n}^{m+1}, e_{\mu_{n}}^{m+1/2} \rangle_{C} 
        + \frac{1}{\tau} \langle e_{p}^{m+1}, e_{\mu_{p}}^{m+1/2} \rangle_{C} 
        \\ 
        & - \langle\mathcal{A}_{h} \tilde{n}^{m+1/2} \nabla_{h} e_{\mu_{n}}^{m+1/2}, \hat{e}_{\bm{u}}^{m+1/2} \rangle_{1} 
        + \langle\mathcal{A}_{h} \tilde{n}^{m+1/2} \nabla_{h} e_{\mu_{n}}^{m+1/2} , 
        \nabla_{h} e_{\mu_{n}}^{m+1/2} \rangle_{1} 
        \\ 
        & - \langle\mathcal{A}_{h} \tilde{p}^{m+1/2} \nabla_{h} e_{\mu_{p}}^{m+1/2}, \hat{e}_{\bm{u}}^{m+1/2} \rangle_{1} 
        + \langle\mathcal{A}_{h} \tilde{p}^{m+1/2} \nabla_{h} e_{\mu_{p}}^{m+1/2} , 
        \nabla_{h} e_{\mu_{p}}^{m+1/2} \rangle_{1} 
        \\
        & = \langle\mathcal{A}_{h} \tilde{e}_{n}^{m+1/2} \nabla_{h} e_{\mu_{n}}^{m+1/2}, \hat{\breve{\mathrm{\bf U}}}^{m+1/2} \rangle_{1} 
        - \langle\mathcal{A}_{h} \tilde{e}_{n}^{m+1/2} \nabla_{h} \breve{\mathrm{M}}_{n}^{m+1/2} , 
        \nabla_{h} e_{\mu_{n}}^{m+1/2} \rangle_{1} 
        \\ 
        & + \langle\mathcal{A}_{h} \tilde{e}_{p}^{m+1/2} \nabla_{h} e_{\mu_{p}}^{m+1/2}, \hat{\breve{\mathrm{\bf U}}}^{m+1/2} \rangle_{1} 
        - \langle\mathcal{A}_{h} \tilde{e}_{p}^{m+1/2} \nabla_{h} \breve{\mathrm{M}}_{p}^{m+1/2} , 
        \nabla_{h} e_{\mu_{p}}^{m+1/2} \rangle_{1} 
        \\
        & + \langle\zeta_{n}^{m}, e_{\mu_{n}}^{m+1/2} \rangle_{C} 
        + \frac{1}{\tau} \langle e_{n}^{m}, e_{\mu_{n}}^{m+1/2} \rangle_{C} 
        + \langle\zeta_{p}^{m}, e_{\mu_{p}}^{m+1/2} \rangle_{C} 
        + \frac{1}{\tau} \langle e_{p}^{m}, e_{\mu_{p}}^{m+1/2} \rangle_{C}, 
\end{aligned}
\end{equation}
with an application of summation by parts formula \eqref{lem: discrete inner of g and afu}. Meanwhile, the right hand side terms could be estimated as follows, with the help of the $\ell^{\infty}$ bound \eqref{estimate: W} and \eqref{estimate: M}:
\begin{subequations} \label{estimate: n and p}
    \begin{align}
        & \langle\mathcal{A}_{h} \tilde{e}_{n}^{m+1/2} \nabla_{h} e_{\mu_{n}}^{m+1/2}, \hat{\breve{\mathrm{\bf U}}}^{m+1/2} \rangle_{1} 
        \leq \|\tilde{e}_{n}^{m+1/2} \|_{2} \cdot |\nabla_{h} e_{\mu_{n}}^{m+1/2} \|_{2} 
        \cdot \|\hat{\breve{\mathrm{\bf U}}}^{m+1/2} \|_{\infty} 
        \nonumber \\ 
        \leq & C^* \|\tilde{e}_{n}^{m+1/2} \|_{2} \cdot \|\nabla_{h} e_{\mu_{n}}^{m+1/2} \|_{2} 
        \leq 4 (C^*)^2 \delta^{-1}  (3 \|e_{n}^{m} \|_{2}^{2} 
        + \|e_{n}^{m-1} \|_{2}^{2} ) 
        + \frac{\delta}{16} \|\nabla_{h} e_{\mu_{n}}^{m+1/2} \|_{2}^{2}, 
        \\ 
        & \langle\mathcal{A}_{h} \tilde{e}_{p}^{m+1/2} \nabla_{h} e_{\mu_{p}}^{m+1/2}, \hat{\breve{\mathrm{\bf U}}}^{m+1/2} \rangle_{1} 
        \leq \|\tilde{e}_{p}^{m+1/2} \|_{2} \cdot \|\nabla_{h} e_{\mu_{p}}^{m+1/2} \|_{2} 
        \cdot \|\hat{\breve{\mathrm{\bf U}}}^{m+1/2} \|_{\infty} 
        \nonumber \\ 
        \leq & C^{*} \|\tilde{e}_{p}^{m+1/2} \|_{2} \cdot \|\nabla_{h} e_{\mu_{p}}^{m+1/2} \|_{2} 
        \leq 4 (C^{*})^{2} \delta^{-1}  (3 \|e_{p}^{m} \|_{2}^{2} + \|e_{p}^{m-1} \|_{2}^2 ) 
        + \frac{\delta}{16} \|\nabla_{h} e_{\mu_{p}}^{m+1/2} \|_{2}^{2}, 
        \\ 
        & - \langle\mathcal{A}_{h} \tilde{e}_{n}^{m+1/2} \nabla_{h} \breve{\mathrm{M}}_{n}^{m+1/2} , 
        \nabla_{h} e_{\mu_{n}}^{m+1/2} \rangle_{1} 
        \leq  \|\tilde{e}_{n}^{m+1/2} \|_{2} \cdot \|\nabla_{h} e_{\mu_{n}}^{m+1/2} \|_{2} 
        \cdot \| \nabla_{h} \breve{\mathrm{M}}_{n}^{m+1/2} \|_{\infty} 
        \nonumber \\ 
        \leq & C^{*} \|\tilde{e}_{n}^{m+1/2} \|_{2} \cdot \|\nabla_{h} e_{\mu_{n}}^{m+1/2} \|_{2} 
        \leq 4 (C^{*} )^{2} \delta^{-1} (3 \|e_{n}^{m} \|_{2}^{2} + \|e_{n}^{m-1} \|_{2}^2 ) 
        + \frac{\delta}{16} \|\nabla_{h} e_{\mu_{n}}^{m+1/2} \|_{2}^{2}, 
        \\ 
        & - \langle\mathcal{A}_{h} \tilde{e}_{p}^{m+1/2} \nabla_{h} \breve{\mathrm{M}}_{p}^{m+1/2} , 
        \nabla_{h} e_{\mu_{p}}^{m+1/2} \rangle_{1} 
        \leq \|\tilde{e}_{p}^{m+1/2} \|_{2} \cdot \|\nabla_{h} e_{\mu_{p}}^{m+1/2} \|_{2} 
        \cdot \| \nabla_{h} \breve{\mathrm{M}}_{p}^{m+1/2} \|_{\infty} 
        \nonumber \\ 
        \leq & C^{*} \|\tilde{e}_{p}^{m+1/2} \|_{2} \cdot \|\nabla_{h} e_{\mu_{p}}^{m+1/2} \|_{2} 
        \leq 4 (C^{*} )^{2} \delta^{-1}  (3 \|e_{p}^{m} \|_{2}^{2} + \|e_{p}^{m-1} \|_{2}^2 ) 
        + \frac{\delta}{16} \|\nabla_{h} e_{\mu_{p}}^{m+1/2} \|_{2}^{2}, 
        \\ 
        & \langle\zeta_{n}^{m}, e_{\mu_{n}}^{m+1/2} \rangle_{C} 
        \leq \|\zeta_{n}^{m} \|_{-1, h} \cdot \|\nabla_{h} e_{\mu_{n}}^{m+1/2} \|_{2} 
        \leq 4 \delta^{-1} \|\zeta_{n}^{m} \|_{-1, h}^{2} 
        + \frac{\delta}{16} \|\nabla_{h} e_{\mu_{n}}^{m+1/2} \|_{2}^{2}, 
        \\ 
        & \langle\zeta_{p}^{m}, e_{\mu_{p}}^{m+1/2} \rangle_{C} 
        \leq \|\zeta_{p}^{m} \|_{-1, h} \cdot \|\nabla_{h} e_{\mu_{p}}^{m+1/2} \|_{2} 
        \leq 4 \delta^{-1} \|\zeta_{p}^{m} \|_{-1, h}^{2} 
        + \frac{\delta}{16} \|\nabla_{h} e_{\mu_{p}}^{m+1/2} \|_{2}^{2}, 
        \\ 
        & \langle e_{n}^{m}, e_{\mu_{n}}^{m+1/2} \rangle_{C} 
        \leq \|e_{n}^{m} \|_{-1, h} \cdot \|\nabla_{h} e_{\mu_{n}}^{m+1/2} \|_{2} 
        \leq \frac{4}{\tau\delta} \|e_{n}^{m} \|_{-1, h}^{2} 
        + \frac{\tau\delta}{16} \|\nabla_{h} e_{\mu_{n}}^{m+1/2} \|_{2}^{2}, 
        \\ 
        & \langle e_{p}^{m}, e_{\mu_{p}}^{m+1/2} \rangle_{C} 
        \leq \|e_{p}^{m} \|_{-1, h} \cdot \|\nabla_{h} e_{\mu_{p}}^{m+1/2} \|_{2} 
        \leq \frac{4}{\tau\delta} \|e_{p}^{m} \|_{-1, h}^{2} 
        + \frac{\tau\delta}{16} \|\nabla_{h} e_{\mu_{p}}^{m+1/2} \|_{2}^{2}. 
    \end{align}
\end{subequations} 
Because of the phase separation bound~\eqref{separation est-3} for the average mobility functions, the following estimate becomes available:  
\begin{equation}\label{error estimate: An and Ap}
    \begin{aligned}
        \langle\mathcal{A}_{h} \tilde{n}^{m+1/2} \nabla_{h} e_{\mu_{n}}^{m+1/2} , 
        \nabla_{h} e_{\mu_{n}}^{m+1/2} \rangle_{1} 
        \ge & \frac{\delta}{2} \| \nabla_{h} e_{\mu_{n}}^{m+1/2} \|_2^{2},  
        \\ 
        \langle\mathcal{A}_{h} \tilde{p}^{m+1/2} \nabla_{h} e_{\mu_{p}}^{m+1/2} , 
        \nabla_{h} e_{\mu_{p}}^{m+1/2} \rangle_{1} 
        \ge & 
          \frac{\delta}{2} \| \nabla_{h} e_{\mu_{p}}^{m+1/2} \|_2^{2} . 
    \end{aligned}
\end{equation}
In terms of the first two terms on the left hand side of \eqref{estimate: en and ep}, we have to recall the following preliminary rough estimate, which has been established in an existing work~\cite{Liu2023PNP}. 

\begin{lemma}  \cite{Liu2023PNP} \label{lem: rough integral estimate} 
The regularity requirement~\eqref{estimate: W}, and phase separation~\eqref{eqn: separation} assumptions are made for the constructed approximate solution $( \breve{\mathsf N}, \breve{\mathsf P})$, as well as the  a-priori assumption~\eqref{prior assumption} for the numerical solution at the previous time steps. In addition, we define the following sets: 
\begin{equation} 
  \Lambda_n = \left\{ ( i,j): n_{i,j}^{m+1} \ge 2 C^* +1 \right\} , \quad 
  \Lambda_p = \left\{ ( i,j): p_{i,j}^{m+1} \ge 2 C^* +1 \right\} , 
  \label{defi-Lambda-1} 
\end{equation} 
and denote $K_n^* := | \Lambda_n |$, $K_p^* := | \Lambda_p |$, the number of grid points in $\Lambda_n$ and $\Lambda_p$, respectively. Then we have a rough bound control of the following nonlinear inner products:   
\begin{eqnarray} 
\begin{aligned} 
  & 
   \langle e_n^{m+1} , F_{\breve{\mathsf N}^m} ( \breve{\mathsf N}^{m+1} ) 
    - F_{n^m} ( n^{m+1} )  \rangle_C  
\\
  &
    + \tau \langle e_n^{m+1} , \ln \breve{\mathsf N}^{m+1} - \ln n^{m+1}  
    - ( \ln \breve{\mathsf N}^m - \ln n^m ) \rangle_C    
  \ge  
  \frac12 C^* K_n^* h^2 - \tilde{C}_4 \| e_n^m \|_2^2 , 
\\
  & 
   \langle e_p^{m+1} , F_{\breve{\mathsf P}^m} ( \breve{\mathsf P}^{m+1} ) 
    - F_{p^m} ( p^{m+1} )  \rangle_C 
\\
  &
    + \tau \langle e_p^{m+1} , \ln \breve{\mathsf P}^{m+1} - \ln p^{m+1} 
    - ( \ln \breve{\mathsf P}^m - \ln p^m ) \rangle_C     
  \ge  
  \frac12 C^* K_p^* h^2 - \tilde{C}_4 \| e_p^m \|_2^2 ,   
\end{aligned}    
    \label{integral-rough-0} 
\end{eqnarray}  
in which $\tilde{C}_4$ is a constant only dependent on $\delta$ and $C^*$, independent of $\tau$ and $h$. In addition, if $K_n^* =0$ and $K_p^*=0$, i.e, both $\Lambda_n$ and $\Lambda_p$ are empty sets, we have an improved bound control: 
\begin{eqnarray} 
\begin{aligned} 
  & 
   \langle e_n^{m+1} , F_{\breve{\mathsf N}^m} ( \breve{\mathsf N}^{m+1} ) 
    - F_{n^m} ( n^{m+1} )  \rangle_C  
\\
  &
    + \tau \langle e_n^{m+1} , \ln \breve{\mathsf N}^{m+1} - \ln n^{m+1}  
    - ( \ln \breve{\mathsf N}^m - \ln n^m ) \rangle_C 
  \ge \tilde{C}_5  \| e_n^{m+1} \|_2^2 - \tilde{C}_4 \| e_n^m \|_2^2 , 
\\
  & 
  \langle e_p^{m+1} , F_{\breve{\mathsf P}^m} ( \breve{\mathsf P}^{m+1} ) 
    - F_{p^m} ( p^{m+1} )  \rangle_C 
\\
  &
    + \tau \langle e_p^{m+1} , \ln \breve{\mathsf P}^{m+1} - \ln p^{m+1} 
    - ( \ln \breve{\mathsf P}^m - \ln p^m ) \rangle_C  
  \ge  
  \tilde{C}_5  \| e_p^{m+1} \|_2^2 - \tilde{C}_4 \| e_p^m \|_2^2 ,   
\end{aligned}    
    \label{integral-rough-1} 
\end{eqnarray}  
in which $\tilde{C}_5$ stands for another constant only dependent on $\delta$ and $C^*$.  
\end{lemma}

As a direct consequence of Lemma~\ref{lem: rough integral estimate}, we see that 
\begin{subequations} \label{integral-rough-2} 
    \begin{align}
        \langle e_{n}^{m+1}, e_{\mu_{n}}^{m+1/2} \rangle_{C} 
        = & \langle e_{n}^{m+1} , F_{\breve{\mathsf{N}}^{m}} (\breve{\mathsf{N}}^{m+1} ) 
        - F_{n^{m}} (n^{m+1} ) 
        + ( - \Delta_h )^{-1} ( e_{n}^{m+1/2} - e_{p}^{m+1/2} ) \rangle_{C} 
        \nonumber \\ 
        & + \tau \langle e_{n}^{m+1} , \ln \breve{\mathsf{N}}^{m+1} - \ln n^{m+1} \rangle_{C} 
        - \tau \langle e_{n}^{m+1} , \ln \breve{\mathsf{N}}^{m} - \ln n^{m} \rangle_{C}, 
        \nonumber \\ 
        \geq & \frac{1}{2} C^{*} K_{n}^{*} h^2 -\tilde{C}_{4} \|\tilde{e}_{n}^{m} \|_{2}^{2} 
        + \langle\tilde{e}_{n}^{m+1}, (-\Delta_h )^{-1} (\tilde{e}_{n}^{m+1/2}-\tilde{e}_{p}^{m+1/2} ) \rangle_C , 
        \\ 
        \langle e_{p}^{m+1}, e_{\mu_{p}}^{m+1/2} \rangle_{C} 
        = & \langle e_{p}^{m+1} , F_{\breve{\mathsf{P}}^{m}} (\breve{\mathsf{P}}^{m+1} ) 
        - F_{p^{m}} (p^{m+1} ) 
        + ( - \Delta_h )^{-1} ( e_{p}^{m+1/2} - e_{p}^{m+1/2} ) \rangle_{C} 
        \nonumber \\ 
        & + \tau \langle e_{p}^{m+1} , \ln \breve{\mathsf{P}}^{m+1} - \ln p^{m+1} \rangle_{C} 
        - \tau \langle e_{p}^{m+1} , \ln \breve{\mathrm{P}}^{m} - \ln p^{m} \rangle_{C}, 
        \nonumber \\ 
        \geq & \frac{1}{2} C^{*} K_{p}^{*} h^2 -\tilde{C}_{4} \|\tilde{e}_{p}^{m} \|_{2}^{2} 
        + \langle\tilde{e}_{p}^{m+1}, (-\Delta_h )^{-1} (\tilde{e}_{p}^{m+1/2}-\tilde{e}_{p}^{m+1/2} ) \rangle_C .  
    \end{align}
\end{subequations} 
On the other hand, the following fact is observed:
\begin{equation}
    \begin{aligned}
        & \langle e_{n}^{m+1},
        (-\Delta_h )^{-1} (e_{n}^{m+1/2}-e_{p}^{m+1/2} ) \rangle_C 
        + \langle e_{p}^{m+1},
        (-\Delta_h )^{-1} (e_{p}^{m+1/2}-e_{n}^{m+1/2} ) \rangle_C  
        \\
        & = \frac{1}{2} \langle (-\Delta_h )^{-1} (e_{n}^{m+1}-e_{p}^{m+1}+e_{n}^{m}-e_{p}^{m} ), e_{n}^{m+1}-e_{p}^{m+1} \rangle_C \\
        & \geq \frac{1}{4} ( \|e_{n}^{m+1}-e_{p}^{m+1} \|_{-1, h}^{2}- \|e_{n}^{m}-e_{p}^{m} \|_{-1, h}^{2} ) .
\end{aligned}
\end{equation}
Going back~\eqref{integral-rough-2}, we arrive at  
\begin{equation}\label{estimate: left en and ep}
    \begin{aligned}
        & \langle e_{n}^{m+1}, e_{\mu_{n}}^{m+1/2} \rangle_C  
        + \langle e_{p}^{m+1/2}, e_{\mu_{p}}^{m+1} \rangle_C  
        \\
        \ge & 
          \frac{1}{2} C^{*} (K_{n}^{*} + K_{p}^{*} ) h^2  
        - \tilde{C}_{4} ( \|e_{n}^{m} \|_{2}^{2} + \|e_{p}^{m} \|_{2}^{2} ) 
        - \frac{1}{4} \|e_{n}^{m}-e_{p}^{m} \|_{-1, h}^{2}. 
    \end{aligned}
\end{equation} 
Subsequently, a substitution of \eqref{estimate: n and p}-\eqref{estimate: left en and ep} into \eqref{estimate: en and ep} yields 
\begin{equation}\label{rough error estimate: NP}
    \begin{aligned}
        & \frac{1}{\tau} \Big( \frac{1}{2} C^{*} (K_{n}^{*} + K_{p}^{*} ) h^2 
        - \tilde{C}_{4} ( \|e_{n}^{m} \|_{2}^{2} +  \|e_{p}^{m} \|_{2}^{2} ) 
        - \frac{1}{4} \|e_{n}^{m}-e_{p}^{m} \|_{-1, h}^{2} \Big) 
        + \frac{\delta}{4} \| \nabla_{h} e_{\mu_{n}}^{m+1/2} \|_2^{2} 
        \\
         & 
        + \frac{\delta}{4} \| \nabla_{h} e_{\mu_{p}}^{m+1/2} \|_2^{2}   
         -  \langle\mathcal{A}_{h} \tilde{n}^{m+1/2} \nabla_{h} e_{\mu_{n}}^{m+1/2}, \hat{e}_{\bm{u}}^{m+1/2} \rangle_{1} 
        -  \langle\mathcal{A}_{h} \tilde{p}^{m+1/2} \nabla_{h} e_{\mu_{p}}^{m+1/2}, \hat{e}_{\bm{u}}^{m+1/2} \rangle_{1} 
        \\ 
        \leq & 8 (C^{*} )^{2} \delta^{-1} (3 \|e_{n}^{m} \|_{2}^{2} + \|e_{n}^{m-1} \|_{2}^{2} 
        + 3 \|e_{p}^{m} \|_{2}^{2}  + \|e_{p}^{m-1} \|_{2}^{2} ) 
        \\ 
        & + 4 \delta^{-1}  ( \|\zeta_{n}^{m} \|_{-1, h}^{2}  + \|\zeta_{p}^{m} \|_{-1, h}^{2} ) 
        + 4 ( \tau\delta)^{-1} ( \|e_{n}^{m} \|_{-1, h}^{2}  + \|e_{p}^{m} \|_{-1, h}^{2} ) . 
    \end{aligned} 
\end{equation}
Moreover, a combination of \eqref{error estimate: e} and \eqref{rough error estimate: NP} leads to  
\begin{equation}\label{rough error estimate}
    \begin{aligned}
        & \frac{1}{2} \|e_{\bm{u}}^{m+1} \|_{2}^{2} 
        + \frac{\tau^{2}}{8} \|\nabla_{h} e_{\psi}^{m+1} \|_{2}^{2} 
        + \frac{\tau \delta}{4} ( \| \nabla_{h} e_{\mu_{n}}^{m+1/2} \|^{2} 
        + \| \nabla_{h} e_{\mu_{p}}^{m+1/2} \|^{2} ) 
        + \frac{1}{2} C^{*} (K_{n}^{*} + K_{p}^{*} ) h^2 
        \\ 
        \leq & \frac{1}{2} \|e_{\bm{u}}^{m} \|_{2}^{2} 
        + \frac{\tau^{2}}{8} \|\nabla_{h} e_{\psi}^{m} \|_{2}^{2}  
        + \tilde{C}_{4} ( \|e_{n}^{m} \|_{2}^{2} + \|e_{p}^{m} \|_{2}^{2} ) 
        + \frac{1}{4} \|e_{n}^{m}-e_{p}^{m} \|_{-1, h}^{2} 
        \\ 
        & + 8 (C^{*} )^{2} \delta^{-1} \tau  (3 \|e_{n}^{m} \|_{2}^{2} + \|e_{n}^{m-1} \|_{2}^{2}  
        + 3 \|e_{p}^{m} \|_{2}^{2}  + \|e_{p}^{m-1} \|_{2}^{2} ) 
        \\ 
        & + 4 \delta^{-1} \tau ( \|\zeta_{n}^{m} \|_{-1, h}^{2}  
        + \|\zeta_{p}^{m} \|_{-1, h}^{2} ) 
        + 4 \delta^{-1} \tau^{-1} ( \|e_{n}^{m} \|_{-1, h}^{2} 
        + \|e_{p}^{m} \|_{-1, h}^{2} )  
        + 2 C_{0}^{2}\tau \|\bm{\zeta}_u^{m} \|_{2}^{2} 
        \\ 
        & + \frac{\tilde{C}_{3}\tau}{2} (3 \|e_{n}^{m} \|_{2}^{2}+ \|e_{n}^{m-1} \|_{2}^{2}  
        + 3 \|e_{p}^{m} \|_{2}^{2}+ \|e_{p}^{m-1} \|_{2}^{2} ) 
        + \frac{\tilde{C}_{2}\tau}{2} (3 \|e_{\bm{u}}^{m} \|_{2}^{2} 
        + \|e_{\bm{u}}^{m-1} \|_{2}^{2} ) . 
    \end{aligned} 
\end{equation}
Meanwhile, the following estimates are available for the right hand side of \eqref{rough error estimate}, 
with the help of a-priori assumption \eqref{prior assumption}:
\begin{subequations}\label{rough error estimate: last step}
    \begin{align}
        & 4 \delta^{-1} \tau^{-1} ( \|e_{n}^{m} \|_{-1, h}^{2} + \|e_{p}^{m} \|_{-1, h}^{2} ) 
        \leq C \delta^{-1} \tau^{-1} ( \|e_{n}^{m} \|_{2}^{2}+ \|e_{p}^{m} \|_{2}^{2} ) 
        \leq C ( \tau^{\frac{13}{2}}+h^{\frac{13}{2}} ), 
        \\ 
        & 4 \delta^{-1} \tau ( \|\zeta_{n}^{m} \|_{-1, h}^{2} + \|\zeta_{p}^{m} \|_{-1, h}^{2} )  , \, \, 
        2 C_{0}^{2}\tau \|\bm{\zeta}_u^{m} \|_{2}^{2} 
        \leq C (\tau^{9}+\tau h^{8} ), 
        \\ 
        & \Big( 8 (C^{*})^{2} \delta^{-1} + \frac{\tilde{C}_{3}}{2} \Big) \tau 
         (3 \|e_{n}^{m} \|_{2}^{2} + \|e_{n}^{m-1} \|_{2}^{2} 
         + 3 \|e_{p}^{m} \|_{2}^{2} + \|e_{p}^{m-1} \|_{2}^{2} ) 
        \leq C (\tau^{\frac{17}{2}}+h^{\frac{17}{2}} ), 
        \\
        & \tilde{C}_{4} ( \|e_{n}^{m} \|_{2}^{2} +  \|e_{p}^{m} \|_{2}^{2} )  , \, \, 
        \frac{1}{4} \|e_{n}^{m}-e_{p}^{m} \|_{-1, h}^{2}  
        \leq C (\tau^{\frac{15}{2}}+h^{\frac{15}{2}} ) , 
        \\ 
        & \Big( \frac{1}{2} + \frac{3\tilde{C}_{2}\tau}{2} \Big) \|e_{\bm{u}}^{m} \|_{2}^{2} 
        + \frac{\tilde{C}_{2}\tau}{2} \|e_{\bm{u}}^{m-1} \|_{2}^{2} 
        \leq C (\tau^{\frac{15}{2}}+h^{\frac{15}{2}} ), 
        \\ 
        & \frac{\tau^{2}}{8} \|\nabla_{h} e_{\psi}^{m} \|_{2}^{2} 
        \leq C (\tau^{\frac{15}{2}}+h^{\frac{15}{2}} ) , 
    \end{align}
\end{subequations} 
in which the discrete Poincar\'e inequality, $\|f\|_{-1, h} \leq C_0 \|f\|_{2}$, as well as the linear refinement constraint $\lambda_{1} h \leq \tau \leq \lambda_{2} h$, have been repeatedly applied. 
A substitution of~\eqref{rough error estimate: last step} into \eqref{rough error estimate} gives  
\begin{equation} 
  \frac{1}{2} C^{*} (K_{n}^{*} + K_{p}^{*} ) h^2 \le C ( \tau^{\frac{13}{2}}+h^{\frac{13}{2}} ) .  
  \label{integral-rough-3} 
\end{equation} 
If $K_n^* \ge 1$ or $K_p^* \ge 1$, the above inequality will make a contradiction, provided that $\tau$ and $h$ are sufficiently small. Subsequently, we conclude that $K_n^* =0$ and $K_p^* =0$, so that both $\Lambda_n$ and $\Lambda_p$ are empty sets. As a result, an application of \eqref{integral-rough-1} (in Lemma~\ref{lem: rough integral estimate}) yields an improved estimate: 
\begin{equation}
    \begin{aligned}
        &  \langle e_{n}^{m+1}, e_{\mu_{n}}^{m+1/2} \rangle_C 
        + \langle e_{p}^{m+1/2}, e_{\mu_{p}}^{m+1/2} \rangle_C 
        \\ 
        \geq & \tilde{C}_{5} ( \| e_{n}^{m+1} \|_{2}^{2} + \| e_{p}^{m+1} \|_{2}^{2} ) 
        -\tilde{C}_{4} ( \| e_{n}^{m} \|_{2}^{2} + \| e_{p}^{m} \|_{2}^{2} ) 
        - \frac{1}{4} \| e_{n}^{m} - e_{p}^{m} \|_{-1, h}^{2} . 
    \end{aligned} 
    \label{integral-rough-4} 
\end{equation} 
Furthermore, its combination with \eqref{rough error estimate: last step} and \eqref{error estimate: e}-\eqref{error estimate: An and Ap} implies that
\begin{equation}
    \frac12 \| e_{\bm{u}}^{m+1} \|_{2}^2 
    + \tilde{C}_{5} (\| e_{n}^{m+1} \|_{2}^{2} +  \| e_{p}^{m+1} \|_{2}^{2} ) 
    \leq C (\tau^{\frac{13}{2}}+h^{\frac{13}{2}} ) .  
\end{equation}
In particular, the following rough error estimate becomes available: 
\begin{equation}
     \| e_{\bm{u}}^{m+1} \|_{2} 
    + \| e_{n}^{m+1} \|_{2} + \| e_{p}^{m+1} \|_{2} 
    \leq \hat{C} (\tau^{\frac{13}{4}} + h^{\frac{13}{4}} ) 
    \leq \tau^{3}+h^{3} , 
\end{equation} 
under the linear refinement requirement $\lambda_{1} h \leq \tau \leq \lambda_{2} h$, with $\hat{C}$ dependent only on the physical parameters and the computational domain. 

As a direct consequence of the above rough error estimate, an application of 2-D inverse inequality indicates that 
\begin{equation}\label{eqn: Linfty of en and ep}
 \|e_{n}^{m+1} \|_{\infty}+ \|e_{p}^{m+1} \|_{\infty} 
\leq \frac{C ( \|e_{n}^{m+1} \|_{2}+ \|e_{p}^{m+1} \|_{2} )}{h} 
\leq C (\tau^{2}+h^{2} ) \leq \frac{\delta}{8}, 
\end{equation}
under the same linear refinement requirement, provided that $\tau$ and $h$ are sufficiently small. With the help of the separation estimate~\eqref{eqn: separation} for the constructed approximate solution $(\breve{\mathsf N}, \breve{\mathsf P})$, a similar property becomes available for the numerical solution at time step $t^{m+1}$:
\begin{equation}\label{eqn: estimate of n and p}
\frac{\delta}{2} \leq n^{m+1} \leq C^*+\frac{\delta}{2} \leq \tilde{C}_1 , 
\quad \text { and } \quad  
\frac{\delta}{2} \leq p^{m+1} \leq C^*+\frac{\delta}{2} \leq \tilde{C}_1 . 
\end{equation} 
Such a $\|\cdot\|_{\infty}$ bound will play a crucial role in the refined error estimate. Moreover, the following bound for the discrete temporal derivative of the numerical solution could also be derived: 
\begin{equation}\label{eqn: rough estimate of n and p}
\begin{aligned}
& \|e_{n}^{m+1}-e_{n}^{m} \|_{\infty} 
\leq \|e_{n}^{m+1} \|_{\infty}+ \|e_{n}^{m} \|_{\infty} 
\leq C (\tau^{2} + h^{2}) \leq \tau, \quad (\text {by }\eqref{eqn: estimate of en and ep}, 
\eqref{eqn: Linfty of en and ep}), 
\\
& \|\breve{\mathsf N}^{m+1}-\breve{\mathsf N}^{m} \|_{\infty} 
\leq C^* \tau, \quad (\text {by }\eqref{estimate: W}), 
\\
& \|n^{m+1}-n^{m} \|_{\infty} 
\leq \|\breve{\mathsf{N}}^{m+1}-\breve{\mathsf{N}}^{m} \|_{\infty}
+ \|e_{n}^{m+1}-e_{n}^{m} \|_{\infty} 
\leq (C^{*}+1) \tau, \\
& \|p^{m+1}-p^{m} \|_{\infty} 
\leq \left(C^{*}+1\right) \tau, \quad \text {(similar analysis).}
\end{aligned}
\end{equation}

\subsection{A refined error estimate}
The following preliminary result, which has been established in an existing work~\cite{Liu2023PNP}, is recalled. 

\begin{lemma}\cite{Liu2023PNP} \label{lem:priori refine error estimate}
Under the a-priori $\|\cdot\|_{\infty}$ estimate \eqref{eqn: Linfty estimate of nk and pk}, 
\eqref{eqn: estimate of nk and pk} for the numerical solution at the previous time steps and the rough $\|\cdot\|_{\infty}$ estimates \eqref{eqn: estimate of n and p}, \eqref{eqn: rough estimate of n and p} for the one at the next time step, we have
\begin{subequations}\label{priori refine error estimate}
    \begin{align}
        & \langle e_{n}^{m+1} - e_{n}^{m}, 
        F_{\breve{\mathsf{N}}^{m}} (\hat{\mathsf{N}}^{m+1}) - F_{n^{m}} (n^{m+1} ) \rangle_C  
        \nonumber \\ 
        & \geq \frac{1}{2} \Big( \Big\langle \frac{1}{\breve{\mathsf{N}}^{m+1}}, (e_{n}^{m+1} )^{2} \Big\rangle_C 
        - \Big\langle \frac{1}{\breve{\mathsf{N}}^{m}}, (e_{n}^{m} )^{2} \Big\rangle_C \Big) 
        - \tilde{C}_{6} \tau ( \|e_{n}^{m+1} \|_{2}^{2}+ \|e_{n}^{m} \|_{2}^{2} ), 
        \\
        &  \langle e_{n}^{m+1}-e_{n}^{m}, 
        \ln \hat{\mathsf{N}}^{m+1}-\ln n^{m+1}- (\ln \hat{\mathsf{N}}^{m}-\ln n^{m} ) \rangle_C 
        \geq -\tilde{C}_{7} \tau ( \|e_{n}^{m+1} \|_{2}^{2}+ \|e_{n}^{m} \|_{2}^{2} ), 
        \\ 
        & \langle e_{p}^{m+1} - e_{p}^{m}, 
        F_{\breve{\mathsf{P}}^{m}} (\hat{\mathsf{P}}^{m+1}) - F_{p^{m}} (p^{m+1} ) \rangle_C  
        \nonumber \\ 
        & \geq \frac{1}{2} \Big( \Big\langle \frac{1}{\breve{\mathsf{P}}^{m+1}}, (e_{p}^{m+1} )^{2} \Big\rangle_C 
        - \Big\langle \frac{1}{\breve{\mathsf{P}}^{m}}, (e_{p}^{m} )^{2} \Big\rangle_C \Big) 
        - \tilde{C}_{6} \tau ( \|e_{p}^{m+1} \|_{2}^{2}+ \|e_{p}^{m} \|_{2}^{2} ), 
        \\
        &  \langle e_{p}^{m+1}-e_{p}^{m}, 
        \ln \hat{\mathsf{P}}^{m+1}-\ln p^{m+1}- (\ln \hat{\mathsf{P}}^{m}-\ln p^{m} ) \rangle_C 
        \geq -\tilde{C}_{7} \tau ( \|e_{p}^{m+1} \|_{2}^{2}+ \|e_{p}^{m} \|_{2}^{2} ), 
    \end{align}
\end{subequations} 
in which the constants $\tilde{C}_{6}$ and $\tilde{C}_{7}$ only depend on $\delta$, and $C^*$. 
\end{lemma}

Now we proceed into the refined error estimate. Again, a combination of the inner product equation \eqref{estimate: en and ep} and the estimates \eqref{estimate: n and p}, \eqref{error estimate: An and Ap}, leads to 
\begin{equation}\label{refine error estimate: n and p}
    \begin{aligned}
        & \langle e_{n}^{m+1} - e_{n}^{m}, e_{\mu_{n}}^{m+1/2} \rangle_C  
        + \langle e_{p}^{m+1} - e_{p}^{m}, e_{\mu_{p}}^{m+1/2} \rangle_C  
        + \frac{5 \tau\delta}{16}  ( \|\nabla_{h} e_{\mu_{n}}^{m+1/2} \|_{2}^{2} 
        +  \|\nabla_{h} e_{\mu_{p}}^{m+1/2} \|_{2}^{2} ) 
        \\ 
        & - \tau \langle\mathcal{A}_{h} \tilde{n}^{m+1/2} \nabla_{h} e_{\mu_{n}}^{m+1/2}, \hat{e}_{\bm{u}}^{m+1/2} \rangle_{1} 
        - \tau \langle\mathcal{A}_{h} \tilde{p}^{m+1/2} \nabla_{h} e_{\mu_{p}}^{m+1/2}, \hat{e}_{\bm{u}}^{m+1/2} \rangle_{1} 
        \\ 
        & \leq 8 (C^{*})^{2} \delta^{-1} \tau (3 \|e_{n}^{m} \|_{2}^{2} 
        + \|e_{n}^{m-1} \|_{2}^{2} + 3 \|e_{p}^{m} \|_{2}^{2} 
        + \|e_{p}^{m-1} \|_{2}^{2} ) 
        \\ 
        & + 4 \delta^{-1} \tau ( \|\zeta_{n}^{m} \|_{-1, h}^{2} + \|\zeta_{p}^{m} \|_{-1, h}^{2} ). 
    \end{aligned}
\end{equation} 
On the other hand, the temporal stencil inner product has to be analyzed more precisely. A detailed expansion for $e_{\mu_{n}}^{m+1/2}$ and $e_{\mu_{p}}^{m+1/2}$ reveals that 
\begin{equation}
    \begin{aligned}
        &  \langle e_{n}^{m+1} - e_{n}^{m}, e_{\mu_{n}}^{m+1/2} \rangle_C 
        + \langle e_{p}^{m+1} - e_{p}^{m}, e_{\mu_{p}}^{m+1/2} \rangle_C \\
       = & \langle e_{n}^{m+1} - e_{n}^{m} , 
        F_{\breve{\mathsf{N}}^{m}} (\breve{\mathsf{N}}^{m+1}) - F_{n^{m}} (n^{m+1} ) \rangle_C  
         + \langle e_{p}^{m+1} - e_{p}^{m}, 
        F_{\breve{\mathsf{P}}^{m}} (\breve{\mathsf{P}}^{m+1} ) - F_{p^{m}} (p^{m+1} ) \rangle_C 
        \\ 
        & + \tau \langle e_{n}^{m+1} - e_{n}^{m}, 
        \ln \breve{\mathsf{N}}^{m+1}-\ln n^{m+1}- (\ln \breve{\mathsf{N}}^{m}-\ln n^{m} ) \rangle_C  
        \\ 
        & + \tau \langle e_{p}^{m+1} - e_{p}^{m}, 
        \ln \breve{\mathsf{P}}^{m+1}-\ln p^{m+1}- (\ln \breve{\mathsf{P}}^{m}-\ln p^{m} ) \rangle_C 
        \\ 
        & + \langle e_{n}^{m+1} - e_{n}^{m} - ( e_{p}^{m+1} - e_{p}^{m} ) , 
        (-\Delta_h )^{-1} (e_{n}^{m+1/2} - e_{p}^{m+1/2} ) \rangle_C .  
    \end{aligned}
\end{equation} 
In terms of the last term, a careful calculation implies the following equality: 
\begin{equation}
    \begin{aligned}
        & \langle e_{n}^{m+1} - e_{n}^{m} - ( e_{p}^{m+1} - e_{p}^{m} ) , 
        (-\Delta_h )^{-1} (e_{n}^{m+1/2} - e_{p}^{m+1/2} ) \rangle_C 
        \\
        & =\frac{1}{2} ( \|e_{n}^{m+1}-e_{p}^{m+1} \|_{-1, h}^{2} - \|e_{n}^{m}-e_{p}^{m} \|_{-1, h}^{2} ). 
    \end{aligned} 
    \label{refined est-1} 
\end{equation} 
A combination of~\eqref{refined est-1} with the refined estimates \eqref{priori refine error estimate} (in Lemma \ref{lem:priori refine error estimate}) yields 
\begin{equation}\label{refine error estimate: left n and p}
    \begin{aligned}
        & \langle e_{n}^{m+1} - e_{n}^{m}, e_{\mu_{n}}^{m+1/2} \rangle_C  
        + \langle e_{p}^{m+1} - e_{p}^{m}, e_{\mu_{p}}^{m+1/2} \rangle_C 
        \\
        & \ge \mathcal{F}^{m+1}-\mathcal{F}^{m} 
        - (\tilde{C}_{6} + \tilde{C}_{7} ) \tau 
        ( \| e_{n}^{m+1} \|_{2}^{2} + \|e_{n}^{m} \|_{2}^{2} 
        +  \|e_{p}^{m+1} \|_{2}^{2} +  \|e_{p}^{m} \|_{2}^{2} ), 
        \\
        & \mathcal{F}^{m+1} 
        = \frac{1}{2} \Big( \Big\langle\frac{1}{\breve{\mathsf{N}}^{m+1}}, (e_{n}^{m+1} )^{2} \Big\rangle_C  
        + \Big\langle \frac{1}{\breve{\mathsf{P}}^{m+1}}, (e_{p}^{m+1} )^{2} \Big\rangle_C  
        + \|e_{n}^{m+1}-e_{p}^{m+1} \|_{-1, h}^{2} \Big). 
    \end{aligned}
\end{equation} 
In turn, a substitution of \eqref{refine error estimate: left n and p} into \eqref{refine error estimate: n and p} results in 
\begin{equation}\label{refine error estimate n and p}
    \begin{aligned}
        & \mathcal{F}^{m+1}-\mathcal{F}^{m} 
        - \tau \langle\mathcal{A}_{h} \tilde{n}^{m+1/2} \nabla_{h} e_{\mu_{n}}^{m+1/2}, \hat{e}_{\bm{u}}^{m+1/2} \rangle_{1} 
        - \tau \langle\mathcal{A}_{h} \tilde{p}^{m+1/2} \nabla_{h} e_{\mu_{p}}^{m+1/2}, \hat{e}_{\bm{u}}^{m+1/2} \rangle_{1} 
        \\ 
        & + \frac{5\tau\delta}{16} ( |\nabla_{h} e_{\mu_{n}}^{m+1/2} \|_{2}^{2} 
        + \|\nabla_{h} e_{\mu_{p}}^{m+1/2} \|_{2}^{2} ) 
        \\ 
        \leq & (\tilde{C}_{6} + \tilde{C}_{7} ) \tau ( \|e_{n}^{m+1} \|_{2}^{2} + \|e_{n}^{m} \|_{2}^{2} 
        + \|e_{p}^{m+1} \|_{2}^{2} + \|e_{p}^{m} \|_{2}^{2} ) 
        \\ 
        & + 8 (C^{*} )^{2} \delta^{-1} \tau (3 \|e_{n}^{m} \|_{2}^{2} 
        + \|e_{n}^{m-1} \|_{2}^{2}  + 3 \|e_{p}^{m} \|_{2}^{2} + \|e_{p}^{m-1} \|_{2}^{2} ) 
        + 4 \delta^{-1} \tau ( \|\zeta_{n}^{m} \|_{-1, h}^{2} + \|\zeta_{p}^{m} \|_{-1, h}^{2} )  
        \\
        \le & 
        \tilde{C}_8 \tau ( \mathcal {F}^{m+1} + \mathcal{F}^m + \mathcal{F}^{m-1} ) 
        + 4 \delta^{-1} \tau ( \|\zeta_{n}^{m} \|_{-1, h}^{2} + \|\zeta_{p}^{m} \|_{-1, h}^{2} )  , 
    \end{aligned}
\end{equation} 
with $\tilde{C}_{8} = 2 ( \tilde{C}_{6} + \tilde{C}_{7} + 24 (C^*)^{2} \delta^{-1}) C^*$,  and the following bound has been applied in the last step: 
\begin{equation}
   \Big\langle \frac{1}{\breve{\mathsf{N}}^{k}}, (e_{n}^{k} )^{2} \Big\rangle_C  
        \geq \frac{1}{C^*} \|e_{n}^{k} \|_{2}^{2}, \, \, \,  
        \Big\langle \frac{1}{\breve{\mathsf{P}}^{k}}, (e_{p}^{k} )^{2} \Big\rangle_C  
        \geq \frac{1}{C^*} \|e_{p}^{k} \|_{2}^{2} , \quad \mbox{so that} \, \, \, 
         \mathcal{F}^{k} \geq \frac{1}{2 C^*} ( \|e_{n}^{k} \|_{2}^{2}
        + \|e_{p}^{k} \|_{2}^{2} ) . 
\end{equation} 
In addition, a combination of \eqref{refine error estimate n and p} and \eqref{error estimate: e} leads to \begin{equation} 
    \begin{aligned}
    & \frac{1}{2} ( \|e_{\bm{u}}^{m+1} \|_{2}^{2}- \|e_{\bm{u}}^{m} \|_{2}^{2} ) 
        + \frac{\tau}{2} \|\nabla_{h} \hat{e}_{\bm{u}}^{m+1/2} \|_{2}^{2} 
        + \frac{\tau^{2}}{8} ( \|\nabla_{h} e_{\psi}^{m+1} \|_{2}^{2}- \|\nabla_{h} e_{\psi}^{m} \|_{2}^{2} ) 
        \\
        & + \mathcal{F}^{m+1}-\mathcal{F}^{m} 
        + \frac{5\tau\delta}{16}  ( \|\nabla_{h} e_{\mu_{n}}^{m+1/2} \|_{2}^{2} 
        +  \|\nabla_{h} e_{\mu_{p}}^{m+1/2} \|_{2}^{2} ) 
        \\ 
        \leq & \tilde{C}_{8} \tau (\mathcal{F}^{m+1}+\mathcal{F}^{m}+\mathcal{F}^{m-1} ) 
         + 4 C_{0}^{2} \tau \|\bm{\zeta}_u^{m} \|_{2}^{2} 
        + 4 \delta^{-1} \tau ( \|\zeta_{n}^{m} \|_{-1, h}^{2} 
        + \|\zeta_{p}^{m} \|_{-1, h}^{2} ) 
        \\ 
        & + \tilde{C}_2 \tau (3 \|e_{\bm{u}}^{m} \|_{2}^{2} 
        +  \|e_{\bm{u}}^{m-1} \|_{2}^{2} ) 
        + \tilde{C}_{3} \tau (3 \|e_{n}^{m} \|_{2}^{2}+ \|e_{n}^{m-1} \|_{2}^{2}  
        + 3 \|e_{p}^{m} \|_{2}^{2}+ \|e_{p}^{m-1} \|_{2}^{2} )  . 
    \end{aligned}
\end{equation} 
With an introduction of a unified error functional, $\mathcal{G}^{k} 
= \mathcal{F}^{k} + \frac{1}{2} \|e_{\bm{u}}^{k} \|_{2}^{2} 
+ \frac{\tau^{2}}{8} \|\nabla_{h} e_{\psi}^{m+1} \|_{2}^{2}$, we see that  
\begin{equation}
    \begin{aligned}
        \mathcal{G}^{m+1} - \mathcal{G}^{m} 
        & \le \tilde{C}_{9} \tau (\mathcal{G}^{m+1}+\mathcal{G}^{m}+\mathcal{G}^{m-1} )  
        + 4 C_{0}^{2} \tau \|\bm{\zeta}_u^{m} \|_{2}^{2} 
        + 4 \delta^{-1} \tau ( \|\zeta_{n}^{m} \|_{-1, h}^{2} 
        + \|\zeta_{p}^{m} \|_{-1, h}^{2} ) , 
\end{aligned}
\end{equation}
where $\tilde{C}_9 =\max (  \tilde{C}_8 + 6 \tilde{C}_3 C^* , 6 \tilde{C}_2 )$. Therefore, with sufficiently small $\tau$ and $h$, an application of discrete Gronwall inequality results the desired higher order convergence estimate
\begin{equation}\label{eqn: high order estimate}
   \mathcal{G}^{m+1} \leq C (\tau^8+h^8 ), 
        \text { so that } 
          \|e_{\bm{u}}^{m+1} \|_{2} 
        + \|e_{n}^{m+1} \|_{2} + \|e_{p}^{m+1} \|_{2} 
        + \tau \|\nabla_{h} e_{\psi}^{m+1} \|_{2} 
        \leq C (\tau^4+h^4 ), 
\end{equation}
in which the higher order truncation error accuracy, $ \| \bm{\zeta}_u^{m} \|_2 , \, \|\zeta_n^{m}\|_2, \,  \|\zeta_p^{m} \|_2 \le C (\tau^4+h^4 )$, has been applied in the analysis. This completes the refined error estimate.

\subsection{Recovery of the a-priori assumption \eqref{prior assumption}}
With the the higher order convergence estimate \eqref{eqn: high order estimate} in hand, the a-priori assumption in \eqref{prior assumption} is recovered at the next time step $t^{m+1}$ :
\begin{equation}
    \begin{aligned}
    &  \|e_{\bm{u}}^{m+1} \|_{2}, \, \|e_{n}^{m+1} \|_{2}, \, \|e_{p}^{m+1} \|_{2}  
    \leq C (\tau^{4}+h^{4}) 
    \leq \tau^{\frac{15}{4}}+h^{\frac{15}{4}}, 
    \\ 
    & \|\nabla_{h} e_{\psi}^{m+1} \|_{2} 
    \leq C (\tau^{3}+h^{3} ) \leq \tau^{\frac{11}{4}}+h^{\frac{11}{4}}, 
    \end{aligned}
\end{equation}
provided $\tau$ and $h$ are sufficiently small, under the linear refinement constraint. As a result, an induction analysis could be effectively applied, and the higher order convergence analysis is finished.

As a further result, the error estimate \eqref{eqn: convergence thm} for the ion concentration variables comes from a combination of \eqref{eqn: high order estimate} with the constructed expansion \eqref{eqn: construction} of the approximate solution $(\breve{\mathsf{N}}, \breve{\mathsf{P}})$, as well as the projection estimate \eqref{eqn: projection estimate}. The error estimate \eqref{eqn: convergence thm} for the pressure variable could be obtained by a similar argument. 

To get a convergence estimate for the electric potential variable $\phi$, we have to recall \eqref{def: error function}, the definition for $e_{\phi}^{m}$, and make use of an elliptic regularity:  
\begin{equation}\label{estimate: phi}
    \| e_{\phi}^{m} \|_{H_{h}^{2}} 
    \leq C \| \Delta_{h} e_{\phi}^{m} \|_{2} 
    \leq C \| e_{n}^{m} -e_{p}^{m} \|_{2} 
    \leq C (\tau^{4}+h^{4} ) .   
\end{equation}
Meanwhile, the following observation is made: 
\begin{equation}\label{estimate: ephi} 
\begin{aligned} 
  & 
  - \Delta_h (e_{\phi}^{m} - \mathring{\phi}^{m} ) 
    = \mathcal{P}_{h} (\tau^{2} ( \mathsf{P}_{\tau,1} - \mathsf{N}_{\tau, 1}) 
    + \tau^{3} ( \mathsf{P}_{\tau,2} - \mathsf{N}_{\tau, 2} ) 
    + h^2 ( \mathsf{P}_{h,1} - \mathsf{N}_{h,1} ) )
 \\ 
  & 
    \mbox{so that} \quad  \| e_{\phi}^{m} - \mathring{\phi}^{m} \|_{H_{h}^{2}} 
    \leq C \| \Delta_{h} (e_{\phi}^{m} - \mathring{\phi}^{m} ) \|_{2} 
    \leq \hat{C}_1 (\tau^{2} + h^{2}), 
\end{aligned} 
\end{equation} 
This in turn gives 
\begin{equation}
    \| \mathring{\phi}^{m} \|_{H_{h}^{2}} 
    \leq \| e_{\phi}^{m} \|_{H_{h}^{2}} 
    +  \| e_{\phi}^{m} - \mathring{\phi}^{m} \|_{H_{h}^{2}} 
    \leq C (\tau^{4}+h^{4} ) 
    + \hat{C}_1  (\tau^{2} + h^{2} )
    \leq (\hat{C}_1 +1 ) (\tau^{2} + h^{2} ) . 
\end{equation} 
This finishes the proof of Theorem \ref{thm: convergence}.

\section{Numerical results}\label{sec: validation}
In this section, we present a few numerical experiment results to validate the theoretical analysis, including the numerical tests for convergence rate, energy stability, mass conservation and the concentration positivity. Since the proposed numerical scheme~\eqref{scheme: main} is nonlinear and coupled, its implementation turns out to be quite technical. A linearized iteration solver is applied to implement the numerical algorithm. In more details, the nonlinear parts are evaluated in terms of the numerical solution at the previous stage, while the linear diffusion and temporal derivative parts are implicitly computed at each iteration stage. In turn, only a linear numerical solver is needed at each iteration stage, although the numerical scheme \eqref{scheme: main} is nonlinear. Such a linearized iteration solver has been widely reported for various nonlinear numerical schemes; in particular, a geometric iteration convergence rate has been theoretically justified for the Poisson-Nernst-Planck (PNP) system~\cite{LiuC2022a}, a highly nonlinear and singular gradient flow model. A similar theoretical analysis is expected for the linearized iteration approach to the numerical scheme~\eqref{scheme: main}, while the technical details will be left in the future works. Such a linearized iteration method is highly efficient; the theoretical analysis in~\cite{LiuC2022a} indicates a geometric iteration convergence rate, while the practical computations have revealed an even better iteration convergence rate in the implementation process. Only five to ten linear solvers are needed in the iteration process for most computational examples reported in this article, and the computational cost of the linear solver is comparable with a standard Poisson solver. Moreover, other than the linearized linear solver, some other alternate iteration approaches, such as preconditioned steepest descent (PSD) solver~\cite{ChenXC22a, feng2017}, could be chosen, and a comparison between difference iteration methods will be considered in the future works.

A two dimensional domain is set as $\Omega=(-2,2)^{2}$. At the initial time step, a first-order scheme is used to obtain the numerical solution. In the subsequent time steps, an iterative algorithm (similar to the one in \cite{Liu2023PNP}) is used to implement the fully nonlinear scheme \eqref{scheme: main}. 
 
The initial data is chosen as 
\begin{equation}\label{eqn: initial condition}
\begin{aligned}
&p_0 (x,y) = 0.6 + 0.2\cos\left(\pi x\right)\cos\left(0.5\pi y\right), 
\\
&n_0 (x,y) = 0.6 + 0.2\cos\left(0.5\pi x\right)\cos\left(\pi y\right),
\\ 
&u_0 (x,y) = -0.25\sin^{2}\left(\pi x\right)\sin\left(2\pi y\right), 
\\ 
&v_0 (x,y) = 0.25\sin\left(2\pi x\right)\sin^{2}\left(\pi y\right), 
\\ 
&\psi_0 (x,y) = \cos\left(0.5\pi x\right)\cos\left(0.5\pi y\right),
\end{aligned}
\end{equation} 
where periodic boundary condition is used. The computation is performed with a sequence of uniform mesh resolutions, and the time step size is taken as $\tau = 0.1h$. Since the exact solution could not be explicitly represented, we measure the Cauchy error to test the convergence rate, a similar approach to that of \cite{Wise2007Solving}. In particular, the error between coarse and fine grid spacings $h$ and $h/2$ is recorded by $\|e_{\zeta}\|=\|\zeta_{h}-\zeta_{h/2}\|$. We present the $\ell^{2}$ and $\ell^{\infty}$ errors of all the physical variables at a final time $T=0.1$. An almost perfect second order accuracy, in both time and space, has been observed in this numerical experiment, which agrees with the theoretical analysis. 

In addition, the simulation results are used to demonstrate the numerical performance to preserve certain physical properties. The total mass conservation of the ion concentration variables (over the computational domain) has been perfectly confirmed in the upper panel of the Figure~\ref{fig: energy_mass}. Moreover, in the same figure, a monotone dissipation property of the discrete total energy $E_{h}$ is also clearly observed, which confirms the theoretical analysis. To explore the positivity-preserving property, we focus on the evolution of the minimum concentration value, i.e., $C_{\min} := \min_{i,j}{(\min_{i,j}n_{i,j}^{m},\min_{i,j}p_{i,j}^{m})}$. As displayed in Figure~\ref{fig: positivity}, the numerical solutions of ion concentration variables remain positive all the time, even though their values could become very low. Overall, these numerical evidences have demonstrated that, the proposed numerical scheme is capable of maintaining mass conservation, total energy dissipation, and positivity at a discrete level. 

\begin{table}[t!]
\centering
\caption{The $\ell^2$ numerical error and convergence rate for $p$, $n$ and $\phi$ 
at $T = 0.1$, with $\tau = 0.1 h$, in a 2-D simulation with the initial data 
$\eqref{eqn: initial condition}$.}\label{table: L2PNP}
\vskip 0.3cm
\begin{tabular}{lcccccccccc}
\toprule
$h$ & Error$(p)$ & Order & Error$(n)$ & Order & Error$(\phi)$ & Order \\
\midrule
$2^{-3}$ & 1.9814e-02 &  --  & 1.9814e-02 &  --  & 1.2041e-03 &  --  \\
$2^{-4}$ & 4.2380e-03 & 2.23 & 4.2380e-03 & 2.23 & 1.4096e-04 & 3.09 \\
$2^{-5}$ & 1.0167e-03 & 2.06 & 1.0167e-03 & 2.06 & 2.6455e-05 & 2.41 \\
$2^{-6}$ & 2.5224e-04 & 2.01 & 2.5224e-04 & 2.01 & 6.6204e-06 & 2.00 \\
$2^{-7}$ & 6.3859e-05 & 1.98 & 6.3859e-05 & 1.98 & 2.0524e-06 & 1.70 \\
\bottomrule
\end{tabular}
\end{table} 

\begin{table}[t!]
\centering
\caption{The $\ell^2$ numerical error and convergence rate for $u$, $v$ and $\psi$ at $T = 0.1$, with $\tau = 0.1 h$, in a 2-D simulation with the initial condition  $\eqref{eqn: initial condition}$.}\label{table: L2NS}
\vskip 0.3cm
\begin{tabular}{lcccccccccc}
\toprule
$h$ & Error$(u)$ & Order & Error$(v)$ & Order & Error$(\psi)$ & Order \\
\midrule
$2^{-3}$ & 3.4045e-02 &  --  & 3.4045e-02 &  --  & 1.1670e-01 &  --  \\
$2^{-4}$ & 1.7835e-03 & 4.25 & 1.7835e-03 & 4.25 & 2.9702e-02 & 1.97 \\
$2^{-5}$ & 2.8403e-04 & 2.65 & 2.8403e-04 & 2.65 & 7.4107e-03 & 2.00 \\
$2^{-6}$ & 6.1956e-05 & 2.20 & 6.1956e-05 & 2.20 & 1.8390e-03 & 2.01 \\
$2^{-7}$ & 1.4938e-05 & 2.05 & 1.4938e-05 & 2.05 & 4.6457e-04 & 1.99 \\
\bottomrule
\end{tabular}
\end{table} 
\begin{table}[t!]
\centering
\caption{The $\ell^{\infty}$ numerical error and convergence rate for $p$, $n$ and $\phi$ at $T = 0.1$, with $\tau = 0.1 h$, in a 2-D simulation with the initial data 
$\eqref{eqn: initial condition}$.}\label{table: L8PNP}
\vskip 0.3cm
\begin{tabular}{lcccccccccc}
\toprule
$h$ & Error$(p)$ & Order & Error$(n)$ & Order & Error$(\phi)$ & Order \\
\midrule
$2^{-3}$ & 1.0010e-02 &  --  & 1.0010e-02 &  --  & 5.6139e-04 &  --  \\
$2^{-4}$ & 2.2635e-03 & 2.14 & 2.2635e-03 & 2.14 & 9.6335e-05 & 2.54 \\
$2^{-5}$ & 5.5796e-04 & 2.02 & 5.5796e-04 & 2.02 & 1.6750e-05 & 2.52 \\
$2^{-6}$ & 1.3907e-04 & 2.00 & 1.3907e-04 & 2.00 & 3.4815e-06 & 2.27 \\
$2^{-7}$ & 3.4712e-05 & 2.00 & 3.4712e-05 & 2.00 & 9.6136e-07 & 1.86 \\
\bottomrule
\end{tabular}
\end{table} 

\begin{table}[t!]
\centering
\caption{The $\ell^{\infty}$ numerical error and convergence rate for $u$, $v$ and $\psi$ at $T = 0.1$, with $\tau = 0.1 h$, in a 2-D simulation with the initial data 
$\eqref{eqn: initial condition}$.}\label{table: L8NS}
\vskip 0.3cm
\begin{tabular}{lcccccccccc}
\toprule
$h$ & Error$(u)$ & Order & Error$(v)$ & Order & Error$(\psi)$ & Order \\
\midrule
$2^{-3}$ & 8.8146e-03 &  --  & 8.8146e-03 &  --  & 5.8266e-02 &  --  \\
$2^{-4}$ & 4.7646e-04 & 4.21 & 4.7646e-04 & 4.21 & 1.4692e-02 & 1.99 \\
$2^{-5}$ & 1.0576e-04 & 2.17 & 1.0576e-04 & 2.17 & 3.4509e-03 & 2.09 \\
$2^{-6}$ & 2.5684e-05 & 2.04 & 2.5684e-05 & 2.04 & 7.4629e-04 & 2.21 \\
$2^{-7}$ & 6.3948e-06 & 2.01 & 6.3948e-06 & 2.01 & 1.6649e-04 & 2.16 \\
\bottomrule
\end{tabular}
\end{table}

\begin{figure}[t!]
\setlength{\abovecaptionskip}{-0.2cm}
\setlength{\belowcaptionskip}{-0.2cm}
\begin{center}
\includegraphics[width=1.0\textwidth]{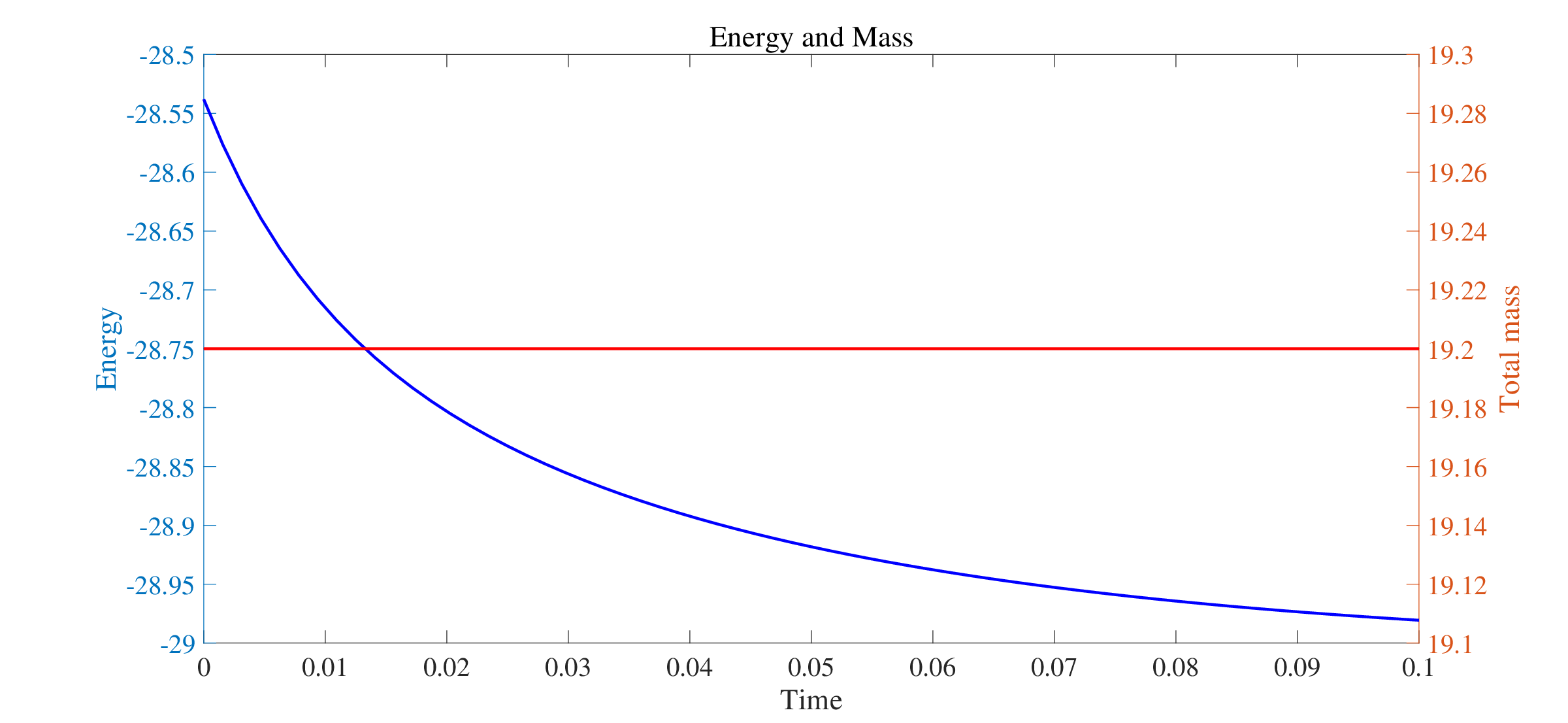}\vspace{-0.1cm}
\end{center}
\caption{Time evolution of the total energy functional and mass of positive ion for the numerical example with initial data $\eqref{eqn: initial condition}$.}  
\label{fig: energy_mass}
\end{figure}

\begin{figure}[t!]
\setlength{\abovecaptionskip}{-0.2cm}
\setlength{\belowcaptionskip}{-0.2cm}
\begin{center}
\includegraphics[width=1.0\textwidth]{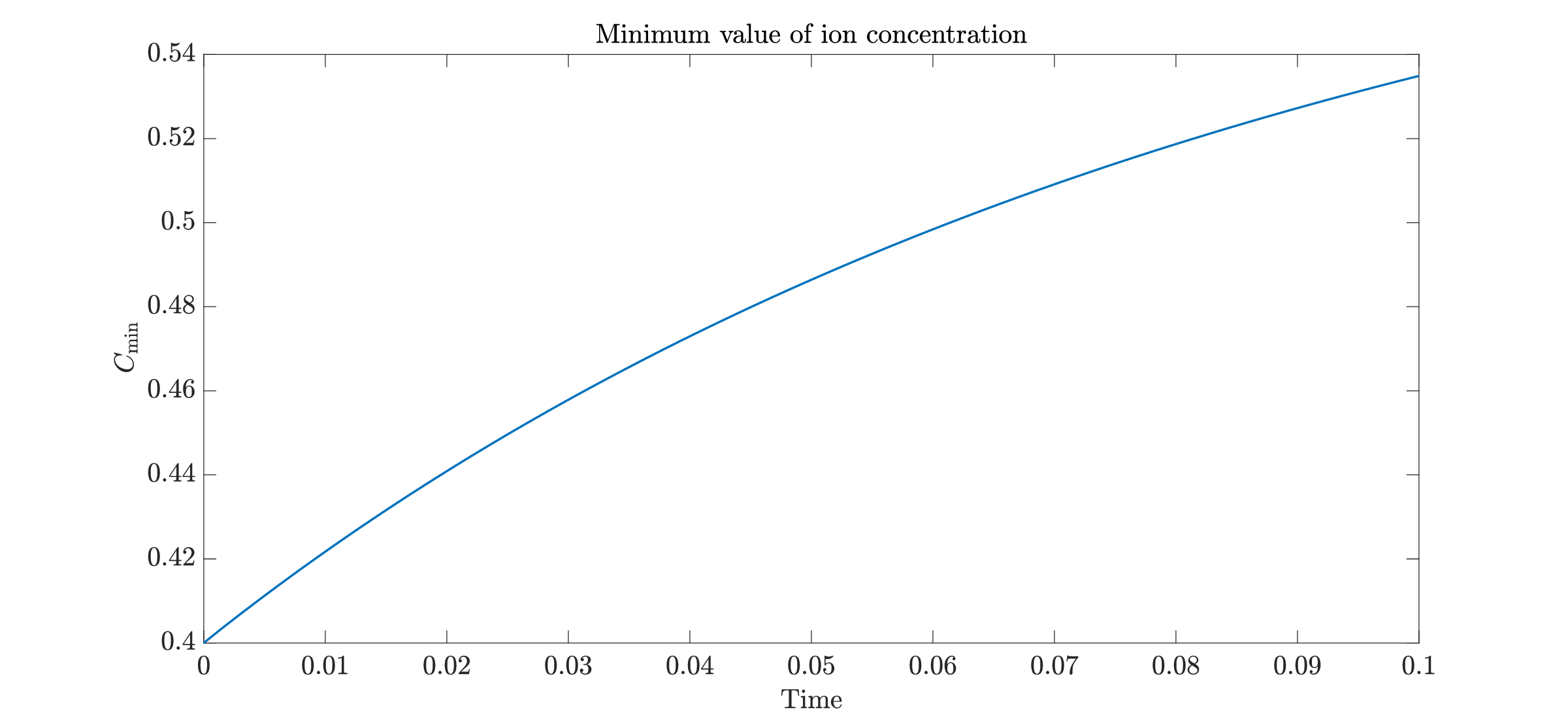}\vspace{-0.1cm}
\end{center}
\caption{Time evolution of the minimum value of positive ion for example, $\eqref{eqn: initial condition}$. 
The curve shows the minimum concentration of positive ion is always positive.}
\label{fig: positivity}
\end{figure}

\section{Conclusion}\label{sec: conclusion}
A second order accurate (in both time and space) numerical scheme has been proposed and analyzed for the 
Poisson-Nernst-Planck-Navier-Stokes (PNPNS) system. The PNP equation is reformulated as a non-constant mobility $H^{-1}$ gradient flow, and the Energetic Variational Approach (EnVarA) leads to a total energy dissipation law. The marker and cell (MAC) finite difference is taken as the spatial discretization, while a modified Crank-Nicolson approximation is applied to the singular logarithmic nonlinear term.  In turn, its inner product with the discrete temporal derivative exactly gives the corresponding nonlinear energy difference, so that the energy stability is ensured for the logarithmic part. The mobility function is explicitly computed by a second order accurate extrapolation formula, and the elliptic nature of the temporal derivative part is preserved and the unique solvability could be ensured. Moreover, nonlinear artificial regularization terms are added in the numerical design to facilitate the positivity-preserving analysis, with the help of the singularity associated with the logarithmic function. Meanwhile, the convective term in the PNP evolutionary equation and the fluid momentum equation are updated in a semi-implicit way, with second order accurate temporal approximation. The unique solvability/positivity preserving and total energy stability analysis has been theoretically established. In addition, an optimal rate convergence analysis is provided, in which the higher order asymptotic expansion for the numerical solution, the rough and refined error estimate techniques have to be included to accomplish such an analysis. In the authors' knowledge, this is the first work to combine the three theoretical properties for any second order accurate scheme for the PNPNS system.

\section*{Acknowledgements} 
The research of Y. Qin was partially supported by the National Natural Science Foundation of China (Grant No.12201369), and the Natural Science Foundation of Shanxi Province (Grant No. 202303021211004). The work of C. Wang was partially supported by National Science Foundation, DMS-2012269 and DMS-2309548.

\end{document}